\documentclass[a4paper,11pt,twoside]{amsart}
\usepackage{amssymb,amsmath}
\usepackage[margin=3cm]{geometry}
\usepackage{graphicx}
\usepackage{dsfont}

\newcommand{\diam}{\diamondsuit}
\newcommand{\h}{\heartsuit}
\newcommand{\bul}{\bullet}
\newcommand{\rar}{\rightarrow}
\newcommand{\hrar}{\hookrightarrow}

\newcommand{\Z}{\mathds{Z}}
\newcommand{\N}{\mathds{N}}
\newcommand{\I}{\mathcal{I}}

\newcommand{\nadrow}[1]{\stackrel{\text{#1}}{=}}

\newcommand{\nad}[2]{\stackrel{{#1}}{#2}}
\newcommand{\wh}[2]{\widehat{#2}^{#1}}
\newcommand{\what}[1]{\widehat{#1}}
\newcommand{\wt}[1]{\widetilde{#1}}

\def\ol{\overline}
\def\la{\langle}
\def\ra{\rangle}
\def\ssq{\subseteq}
\def\g{\gamma}
\def\lam{\lambda}
\def\ka{\kappa}
\def\mc{\mathcal}
\def\na{\twoheadrightarrow}

\usepackage{hyperref}

\usepackage{varwidth}
\usepackage{xypic}
\xyoption{curve}
\newcommand{\diag}[1]{\xymatrix@C=4mm@R=5mm@M=-1mm@W=-1mm{#1}}
\newcommand{\podpis}[2]{{\mathop{#1}\limits_{\makebox[2ex][c]{$\scriptstyle\vphantom{+}#2$}}^{\,^{\,^{\,}}}}}

\begin{document}

\newcounter{licznik}[subsection]
\renewcommand{\thelicznik}{\thesubsection.\arabic{licznik}}
\newenvironment{TWIERDZENIE}[1]
    {\refstepcounter{licznik}\medskip\par\noindent\textbf{\thelicznik\hspace{0.6em}#1.}\hspace{1em}\ignorespaces}
    {\medskip}
\newenvironment{tw}    {\begin{TWIERDZENIE}{Theorem}}      {\end{TWIERDZENIE}}
\newenvironment{lem}   {\begin{TWIERDZENIE}{Lemma}}            {\end{TWIERDZENIE}}
\newenvironment{df}    {\begin{TWIERDZENIE}{Definition}}        {\end{TWIERDZENIE}}
\newenvironment{dfe}   {\begin{TWIERDZENIE}{Definitions}}        {\end{TWIERDZENIE}}
\newenvironment{wn}    {\begin{TWIERDZENIE}{Corollary}}          {\end{TWIERDZENIE}}
\newenvironment{ozn}   {\begin{TWIERDZENIE}{Notation}}       {\end{TWIERDZENIE}}
\newenvironment{ozna}  {\begin{TWIERDZENIE}{Notation}}       {\end{TWIERDZENIE}}
\newenvironment{stw}   {\begin{TWIERDZENIE}{Proposition}}     {\end{TWIERDZENIE}}
\newenvironment{fakt}  {\begin{TWIERDZENIE}{Fact}}             {\end{TWIERDZENIE}}
\newenvironment{uw}    {\begin{TWIERDZENIE}{Remark}}            {\end{TWIERDZENIE}}
\newenvironment{uwi}   {\begin{TWIERDZENIE}{Remarks}}            {\end{TWIERDZENIE}}
\newenvironment{prz}   {\begin{TWIERDZENIE}{Example}}         {\end{TWIERDZENIE}}
\newenvironment{przy}  {\begin{TWIERDZENIE}{Examples}}        {\end{TWIERDZENIE}}
\newenvironment{konstr}{\begin{TWIERDZENIE}{Construction}}      {\end{TWIERDZENIE}}

\title{Structure of Chinese algebras}
\author{Joanna Jaszu\'nska and Jan Okni\'nski}
\thanks{Work supported in part by a MNiSW
research grant N201 004 32/0088 (Poland)}

\begin{abstract}
The structure of the algebra $K[M]$ of the Chinese monoid $M$ over
a field $K$ is studied. The minimal prime ideals are described.
They are determined by certain  homogeneous congruences on $M$ and
they are in a one to one correspondence with diagrams of certain
special type. There are finitely many such ideals. It is also
shown that the prime radical $B(K[M])$ of $K[M]$ coincides with
the Jacobson radical and the monoid $M$ embeds into the algebra
$K[M]/B(K[M])$. A new representation of $M$ as a submonoid of the
product $B^{d}\times \Z^{e}$ for some natural numbers $d,e$, where
$B$ stands for the bicyclic monoid, is derived. Consequently, $M$
satisfies a nontrivial identity.
\end{abstract}

\maketitle

\noindent {\footnotesize{J. Jaszu\'nska: Center for Advanced
Studies, Warsaw University of Technology, pl. Politechniki 1,\\
00-661 Warsaw, Poland, \url{J.Jaszunska@mimuw.edu.pl}

\noindent J. Okni\'nski\footnote{Corresponding author.}: Institute
of Mathematics, Warsaw University, Banacha 2, 02-097 Warsaw,
Poland, \\
\url{okninski@mimuw.edu.pl}}}

\

\noindent {\footnotesize{Mathematics Subject Classification:
16S15, 16S36, 16N60, 20M05, 20M25


}}

\

\tableofcontents

\newpage

\section*{Introduction}

For a positive integer $n$ we consider the monoid $M =\langle
a_{1},\ldots ,a_{n}\rangle$ defined by the relations
\begin{equation} \label{relations}
a_{j}a_{i}a_{k}=a_{j}a_{k}a_{i}= a_{k}a_{j}a_{i} \, \mbox{ for }
\,  i\leq k\leq j . \end{equation} It is called the Chinese monoid
of rank $n$. It is known that every element $x\in M$ has a unique
presentation of the form
\begin{equation}
\label{canon} x=b_{1}b_{2}\cdots b_{n},
\end{equation}
where
 \begin{eqnarray*}
b_{1} &=&a_{1}^{k_{11}} \\
  b_{2}&=&(a_{2}a_{1})^{k_{21}} a_{2}^{k_{22}}\\
  b_{3}&=&(a_{3}a_{1})^{k_{31}}(a_{3}a_{2})^{k_{32}}a_{3}^{k_{33}}\\
   &&  \cdots \\
 b_{n}&=&(a_{n}a_{1})^{k_{n1}}(a_{n}a_{2})^{k_{n2}}\cdots
 (a_{n}a_{n-1})^{k_{n(n-1)}}a_{n}^{k_{nn}},
\end{eqnarray*}
with all exponents nonnegative \cite{cass}. We call it the
canonical form of the element $x$. In particular, $M$ has
polynomial growth of degree $n(n+1)/2$, \cite{KL}.  The Chinese
monoid is related to the so called plactic monoid, introduced and
studied in \cite{las-lec, las-schut}. Both constructions are
strongly related to Young tableaux, and therefore to
representation theory and algebraic combinatorics. The latter
construction has already been established as a classical and
powerful tool of the respective theories \cite{fulton}. The
Chinese monoid appeared in the classification of classes of
monoids with the growth function coinciding with that of the
plactic monoid \cite{duchamp}. Combinatorial properties of $M$
were studied in detail in \cite{cass}. In case $n=2$, the Chinese
and the plactic monoids coincide. The monoid algebra $K[M]$ over a
field $K$ is the unital algebra defined by the algebra
presentation determined by relations (\ref{relations}). It is
called the Chinese algebra of rank $n$. If $n=2$, the structure of
$K[M]$ is described in \cite{jofc}. In particular, this algebra is
prime and semiprimitive, it is not noetherian and it does not
satisfy any polynomial identity.  For $n=3$, some information on
$K[M]$ was obtained in \cite{praca}. In particular the Jacobson
radical of $K[M]$ is nonzero but it is nilpotent and the prime
spectrum of $K[M]$ is pretty well understood. One of the
motivations for a study of the Chinese monoid is based on an
expectation that it can play a similar role as the plactic monoid
in several aspects of representation theory, quantum algebras, and
in algebraic combinatorics. Another motivation stems from
difficult open problems concerning the radical of finitely
presented algebras.

The results of this paper contribute to the general program of
studying finitely presented algebras defined by homogeneous
semigroup presentations. We say that an algebra $A$ with unity is
defined by homogeneous semigroup relations if it is given by a
presentation $A=\langle X \colon R \rangle $, where $X$ is a set
of free generators of a free algebra over $K$ and $R$ is a set of
relations of the form $u=w$, where $u,w$ are words of equal
lengths in the generators from $X$. In this case $A$ may be
identified with the semigroup algebra $K[S]$, where $S$ is the
monoid defined by the same presentation, \cite{JO}. Notice that
there is a natural length function on the underlying monoid $S$.
Certain important classes of such algebras, and of the underlying
monoids, have been recently considered, in particular see
\cite{jofc,gateva,f}. Clearly, the Chinese algebra $K[M]$ is of
this type. The plactic algebra is also defined by semigroup
relations of degree $3$. Algebras corresponding to the set
theoretic solutions of the Yang-Baxter equation are defined by
quadratic semigroup relations \cite{f}.

For certain important constructions of algebras defined by
homogeneous semigroup relations it was shown that the minimal
prime ideals have a very special form, which proved to have far
reaching consequences for the properties of the algebra,
\cite{binom,f}. One might expect that this is a more general
phenomenon occurring in this class of algebras. Our aim is to
consider problems of this type for the class of Chinese algebras.
We establish a remarkable form of minimal prime ideals of the
algebra $K[M]$ and derive several consequences.

By $J(K[S])$ and $B(K[S])$ we denote the Jacobson and the prime
radical of $K[S]$, respectively. If $\eta $ is a congruence on a
semigroup $S$ then $\I_\eta$ stands for the ideal of $K[S]$
spanned as a vector space over $K$ by the set $\{s-t \colon s,t\in
S, (s,t)\in \eta \}$. So $K[S]/\I_\eta=K[S/\eta ]$. If $\phi
:S\rightarrow T$ is a semigroup homomorphism, then by $ker (\phi)$
we mean the congruence on $S$ determined by $\phi$.

The paper is organized as follows. In Section~\ref{r-typy}, two
finite families of ideals of $K[M]$ are introduced, referred to as
ideals of $\h$ and $\diam$ type, respectively, and
 it is shown that every prime ideal $P$ of
$K[M]$ contains one of these ideals (Theorem~\ref{idealy}). Each
of these ideals is of the form $\I_\rho$ for a congruence $\rho $
on $M$. Moreover, each $\rho$ is a homogeneous congruence, which
means that $(s,t)\in \rho $ for $s,t\in M$ implies that $s,t$ have
equal length. The structure of the corresponding monoids
$M_{\rho}=M/\rho$ is described in Lemma~\ref{Mkier} and
Lemma~\ref{Mkaro}.

A much more involved construction allows us to continue this
process in Section~\ref{r-min-id-pier} by showing that every prime
$P$ contains an ideal of the form $\I_{\rho_2}$ for some
homogeneous congruence $\rho _{2}$ containing $\rho$. Proceeding
this way, we construct a finite tree $D$ whose vertices $d$
correspond to certain homogeneous congruences $\rho (d)$ on $M$
and such that $\rho (d)\subseteq \rho (d')$ if the vertex in $d\in
D$ is above the vertex $d'$ (Definition~\ref{drzewo} and
Construction~\ref{interpr}). Moreover, congruences corresponding
to vertices lying in different branches of the tree $D$ are
incomparable under inclusion. The main result of the paper,
Theorem~\ref{bijekcja}, asserts that $d\mapsto \I_{\rho (d)}$
determines a bijective correspondence between the set of leaves of
$D$ and the set of minimal prime ideals of $K[M]$. In order to
prove this, we first show that the ideal $\I_{\rho (d)}$
determined by a leaf $d$ of $D$ is a prime ideal
(Theorem~\ref{liscieidpier}) and also that every prime ideal of
$K[M]$ contains such an ideal $\I_{\rho (d)}$
(Theorem~\ref{idpierzawlisc}). In particular, every minimal prime
ideal $P$ of $K[M]$ contains exactly one of the ideals of $\h$ or
of $\diam$ type (which correspond to the first level of the tree
$D$), but some of the ideals of $\h$ and $\diam$ types can be
contained in several minimal prime ideals of $K[M]$. The proof
provides us with a procedure to construct every such ideal $P$. In
particular, every minimal prime $P$ has a remarkable form
$P=\I_{\rho_P}$, where $\rho_{P}$ is the congruence on $M$ defined
by $\rho_{P}=\{ (s,t) \in M\times M \colon s-t\in P \}$.
Consequently, $K[M]/P\simeq K[M/\rho_{P}]$, so $K[M]/P$ inherits
the natural $\Z$-gradation and therefore this algebra is again
defined by a homogeneous semigroup presentation.

In Section~\ref{r-zast} we derive some important consequences
of the main result. Our construction implies that every
$M/\rho_{P}$ is contained in a product $B^{i}\times \Z^{j}$ for
some $i,j$, where $B=\langle p,q \colon qp=1\rangle $ is the
bicyclic monoid. The latter plays an important role in ring theory
and in semigroup theory, \cite{CP,TYLam}. We show that $M$ embeds
into the product $\prod_{P}K[M]/P$, where $P$ runs over the set of
all minimal primes in $K[M]$. Hence $M$ embeds into some
$B^d\times \Z^e$. This entirely new representation of the Chinese
monoid $M$ implies in particular that $M$ satisfies certain
explicitly given semigroup identity. Since the leaves of $D$
correspond to diagrams of certain special type, one can enumerate
the minimal primes of $K[M]$. It turns out that their number is
equal to the so called $n$-th Tribonacci number. Moreover, the
description of minimal primes $P$ of $K[M]$ allows us to prove
that every $K[M]/P$ is semiprimitive. In particular, the prime
radical of $K[M]$ coincides with the Jacobson radical.

\section{Special types of ideals and congruences}
\label{r-typy}

In this section, two families of ideals of $K[M]$ are defined in
Part \ref{s-idealy}. They will play a crucial role in the approach
developed in the paper. Basic properties of these ideals are then
presented in Part \ref{s-postac}. Throughout, $M$ stands for a
Chinese monoid of rank $n \geq 3$.

\subsection{Ideals of $\h$ and $\diam$ type}
\label{s-idealy}

\

We start with describing certain relations that hold in $K[M]$.
\begin{tw}
\label{rown} The following equalities hold in the Chinese algebra
$K[M]$:
\begin{align}
\label{22}
 &(a_i a_j - a_j a_i)K[M](a_k a_l - a_l a_k)&= 0&  & \textup{for} & & i > j \geq k > l  \\
\label{23}
 &(a_i a_j - a_j a_i)K[M](a_{j+1} a_l - a_l a_{j+1})a_m&= 0&  & \textup{for}  & & i \geq j+1 > j \geq m > l \\
\label{32} a_m &(a_i a_j - a_j a_i)K[M](a_{j+1} a_l - a_l
a_{j+1})&= 0&  & \textup{for} & & i > m \geq j+1 > j \geq l
\end{align}
\end{tw}

\begin{proof}
We use the canonical form (\ref{canon}) of elements of $M$. To
shorten the notation, we write only $i$ instead of $a_i$. Also, we
write each exponent as $*$ if it may be equal 0 and as $+$ if it
is positive. Thus, the canonical form of an element $w\in M$ is
\[w= (1)^{*}\ (21)^{*}(2)^{*}\ (31)^{*}(32)^{*}(3)^{*} \ldots (n1)^{*}(n2)^{*}\ldots(n)^{*}\]
and the desired equalities may be written as
\begin{align*}
&(ij - ji)\;w\;(kl - lk) & =0&  \ & \textup{for} & \   & i > j \geq k > l \tag*{(\ref{22})} \\
&(ij-ji)\;w\;(kl-lk)m & =0& \ & \textup{for} & \  & i \geq k = j+1 > j \geq m > l \tag*{(\ref{23})} \\
m &(ij-ji)\;w\;(kl-lk) & =0& \ & \textup{for} & \  & i > m \geq k
= j+1 > j \geq l \tag*{(\ref{32})}
\end{align*}
Notice that all three equalities are of the form $\alpha w
\beta=0$. We proceed by induction on the length of $w$. If $w$ has
length $0$, so it is the unity of $M$, by using the defining
relations and bringing the involved elements of $M$ to the
canonical form we get
\begin{multline*}
(ij-ji)(kl-lk) = (ijk)l - (ijl)k - j(ikl) + j(ilk) = \\ = j(ikl) -
j(il)k - j(il)k + (jk)(il) = (jk)(il) - (jk)(il) - (jk)(il) +
(jk)(il) = 0,
\end{multline*}
and similarly
\begin{multline*}
(ij-ji)(kl-lk)m = (ijk)lm - (ijl)(km) - j(ikl)m + j(il)(km) = \\
= j(ikl)m - (il)j(km) - j(km)(il) + j(km)(il) = j(km)(il) -
j(km)(il) = 0,
\end{multline*}
\begin{multline*}
m(ij-ji)(kl-lk) = m(ijk)l-(mj)(ikl)-m(ijl)k+(mj)(ilk) = \\ =
mk(ijl) - k(mj)(il) - (mj)k(il) + k(mj)(il) = k(mj)(il)- k
(mj)(il) = 0.
\end{multline*}
So, assume the length of $w$ is positive and assume the equalities
hold for all $w'\in M$ shorter than $w$. The following
regularities hold in all three cases being considered.

If $y$ is the last letter of $w$ and $y\geq k$, then $y(kl-lk) =
k(yl) - k(yl) = 0$, which completes the proof. So assume that $y <
k$. Two possibilities can occur.

1) If $$w = (1)^{*}\ (21)^{*}(2)^{*}\ (31)^{*}(32)^{*}(3)^{*}
\ldots (y1)^{*}(y2)^{*} \ldots (y)^{+}$$ then all letters of the
word $w$ are smaller than $k$, so not greater then $j$, in
particular this applies to the first letter --- we  denote it by
$x$. Therefore, $(ij-ji)x = j(ix) - j(ix) = 0$, which proves the
assertion.

2) We now assume that for some $x>y$
$$w = (1)^{*}\ (21)^{*}(2)^{*}\ (31)^{*}(32)^{*}(3)^{*} \ldots (x1)^{*}(x2)^{*} \ldots (xy)^{+}.$$
If $x < k$ then again the first letter of $w$ is smaller than $k$,
so not greater than $j$. As above, we obtain $(ij-ji)x =
j(ix)-j(ix) = 0$, as desired. Hence, assume that $x \geq k$. We
know that $k>l$ and $k>y$. Two possibilities arise.

2a) $l \geq y$, so $x \geq k > l \geq y$. Let $w'$ be the initial
subword of the word $w$ such that $w=w'(xy)$. Then $w'$ is shorter
than $w$, so by the induction hypothesis $\alpha w' \beta = 0$.
Moreover, in all three equalities $\alpha w \beta=0$ considered
above, $xy$ commutes with $\beta$, because $xy$ commutes with all
letters of $z$ such that $x>z>y$. Thus we get
\[\alpha w \beta = \alpha (w'(xy)) \beta = (\alpha w' \beta) (xy) = 0 \cdot (xy) = 0,\]
which completes the proof in this case.

2b) $l < y$, so $x \geq k > y > l$. Then
\begin{multline*}
(xy)(kl-lk) = (xyk)l - (xyl)k = (kxy)l - (yxl)k = k(xyl)- y(xlk) = \\
= k(yxl) - y(kxl) = ky(xl) - yk(xl) = (ky-yk)(xl),
\end{multline*}
and $xl$ also commutes with all $m$ such that $x>m>l$. Therefore,
$(xy) \beta  = \beta'(xl)$, where $\beta'$ is of the same form as
$\beta$, but has an $y$ instead of the $l$.

Thus the following equality holds:
\[\alpha w \beta = \alpha w' (xy) \beta = \alpha w' \beta' (xl).\]
If $y<m$, the indices $i,j,k,y,m$ in $\alpha,\beta'$ satisfy the
inequalities mentioned in hypotheses of the theorem. Then $\alpha
w \beta = \alpha w' \beta' (xl) = 0$ holds by the induction
hypothesis, because $w'$ is shorter than $w$. This completes the
proof in this case.

If $y \geq m$ (this can occur in the case of equality (\ref{23})),
we have $\beta'=(ky-yk)m= y(km) - y(km) = 0$. Therefore in this
case the desired equality holds as well. This completes the proof.
\end{proof}

\begin{ozn}
\label{boxplus} Pairs of elements $\alpha,\beta \in K[M]$
satisfying $\alpha K[M] \beta = 0$ and of a form as in
Theorem~\ref{rown} will be called  \emph{pairs of type
$\boxplus$}. We denote by $\boxplus=\{(\alpha_{i},\beta_i) \colon
i \in I\}$ the set of all such pairs; $I$ is a finite set of
indices.
\end{ozn}

If $P$ is a prime ideal of $K[M]$, then for each of the equalities
$\alpha K[M] \beta = 0$ in Theorem \ref{rown} one of the elements
$\alpha$ or $\beta$ must belong to $P$. So $P$ must contain a set
of the form $\{\gamma_i \colon \forall_{i \in I}\,
(\gamma_i=\alpha_{i} \text{ or } \gamma_i=\beta_i)\}$. In this
manner we obtain a number of different sets $X_1, X_2, \ldots$. We
shall use indices $\gamma_{i,j}$ for the elements of the set
$X_j$. By $(X_j)$ we denote the ideal generated by the set $X_j$
in $K[M]$. So, let
\[P_j = (X_j) =  \sum_{\g_{i,j} \in X_j} K[M]\g_{i,j} K[M]
\triangleleft K[M].\]

Since every element $\g_i$ is of the form $l_i-p_i$ for some
$l_i,p_i\in M$, it follows that $P_j = \mathcal{I}_{\rho_j}$,
where $\rho_j$ is a congruence generated by the pairs $(l_i,p_i)$
for $i \in I$.

\begin{df}
\label{types-def} An \textbf{ideal of $\heartsuit$ type}, for $s =
2,3, \ldots, n-1$, is the ideal of $K[M]$ generated by the
elements:
\begin{align*}
\label{1BC} \tag{$\heartsuit$}
a_m a_i &- a_i a_m &  &{\textup{for \ }}& s & \leq m, i, \notag\\
a_l a_m &- a_m a_l &  &{\textup{for \ }}& l,m & \leq s. \notag
\end{align*}
Notice that, modulo such an ideal, the corresponding element $a_s$
is central.

An \textbf{ideal of $\diamondsuit$ type}, for $s = 2,3, \ldots,
n$, is the ideal generated  by the elements:
\begin{align*}
\label{1A} \tag{$\diamondsuit$}
a_m a_i &- a_i a_m, & a_i a_{s-1} a_m &- a_m a_{s-1} a_i & &{\textup{for \ }}& s & \leq m,i, \notag\\
a_l a_m &- a_m a_l, &  a_l a_s a_m &- a_m a_s a_l & &{\textup{for
\ }}& l,m & \leq s-1. \notag
\end{align*}
Notice that, modulo such an ideal, the corresponding element $a_s
a_{s-1}$ is central.

We say a congruence $\rho$ on $M$ is of $\heartsuit$ or
$\diamondsuit$ \emph{type} if the ideal $\I_\rho$ o $K[M]$
generated by $\rho$ is of $\heartsuit$ or $\diamondsuit$ type,
respectively. We write $M_\rho = M/\rho$ in this case.
\end{df}

If $\I$ is an ideal of a ring $R$ then $\ol{w}$ denotes the image
of the element $w\in R$ in $R/\I$. Sometimes, to simplify
notation, we shall write $w$ instead of $\ol{w}$ if from the
context it is clear that we mean the image in $R/\I$.

\begin{tw}
\label{idealy} Every prime ideal $P$ in $K[M]$  contains one of
the above mentioned $2n-3$ ideals $\I_\rho$ of $\h$ or $\diam$
type.
\end{tw}

\begin{proof}
If all elements $a_i$ in $K[M]$ commute modulo $P$, then $K[M]/P$
satisfies all equalities (\ref{22})-(\ref{32}) of Theorem
\ref{rown}. Hence $P$ contains all ideals of $\h$ and $\diam$
type.

Hence, we will assume that for some $u>v$ the element
$a_ua_v-a_va_u$ does not belong to the ideal $P$. Since $P$ is
prime, for each equality of the type $\alpha K[M] \beta = 0$
either $\alpha \in P$ or $\beta \in P$ holds. In particular, for
an equality of type (\ref{22}), all elements $a_ka_l-a_la_k$ for
$v\geq k>l$ and all $a_ia_j-a_ja_i$ for $i>j\geq u$ must belong to
$P$.

Thus, in $K[M]/P$, the elements $a_1,a_2,\ldots,a_v$ commute and
the  elements $a_u,a_{u+1},\ldots,a_n$ also commute. Let $s$ be
the smallest index greater than $v$ and such that in $K[M]/P$ the
element $a_s$  does not commute with an element $a_i$ for some
$i\in \{1,2,\ldots ,s-1\}$. Such an $s$ exists and $s \leq u$,
because for $s=u$ the elements $a_s=a_u$ and $a_v$ by assumption
do not commute in $K[M]/P$. Since $s$ is minimal, the elements
$a_1,a_2,\ldots,a_{s-1}$ commute in $K[M]/P$.

Since $a_s$ and $a_i$ do not commute in $K[M]/P$, in the
equalities of (\ref{22}) type, in which $\beta = a_ia_s-a_sa_i$,
we obtain $\alpha \in P$. Therefore, the elements
$a_s,a_{s+1},\ldots,a_n$ must commute in $K[M]/P$. Thus, we have
found such an $s > 1$ that the elements $a_1,a_2,\ldots,a_{s-1}$
commute in $K[M]/P$ and the elements $a_s,a_{s+1},\ldots,a_n$ also
commute in $K[M]/P$, so $P$ contains $\alpha$ or $\beta$ for each
equality of (\ref{22}) type.

The prime ideal $P$ must also contain $\alpha$ or $\beta$ for each
equality of (\ref{23}) and (\ref{32}) type. We know that $P$ does
not contain the element $a_ua_v-a_va_u$, so elements $a_u,a_v$ do
not commute in $K[M]/P$.

Assume that in $K[M]/P$ the element $a_{s-1}$ commutes with each
of $a_s,a_{s+1},\ldots,a_n$ (so, in view of the earlier
assumptions, $a_{s-1}$ is central in $K[M]/P$) or $a_s$ commutes
in $K[M]/P$ with all elements $a_1,a_2,\ldots,a_{s-1}$ (so,
similarly $a_{s}$ is central in $K[M]/P$). If in some equality
(\ref{22})-(\ref{32}) in one of the parentheses there is a
noncommuting pair (so their difference is not in $P$), then the
pairs in the other parentheses must commute in $K[M]/P$.
Therefore, if elements $a_1,a_2,\ldots,a_{s-1}$ commute in
$K[M]/P$, the elements $a_s,a_{s+1},\ldots,a_n$ commute in
$K[M]/P$ and one of the elements $a_{s-1}$, $a_s$ is central in
$K[M]/P$, then $P$ contains an element $\alpha$ or $\beta$ from each
equality of (\ref{22})-(\ref{32}) type. The properties described
above lead to the conclusion that $P$ contains some \textbf{ideal
of \ref{1BC} type}. There are $n-2$ such ideals (because $1<s<n$).

Assume now that in $K[M]/P$ an element $a_{s-1}$ does not commute
with an element $a_i$ for some $i\in \{s,s+1,\ldots ,n\}$ and
$a_s$ does not commute with an element $a_l$ for some $l\in \{
1,2,\ldots ,s-1\}$.

Consider an equality of (\ref{23}) type. The ideal $P$ is prime,
so it  must contain $\alpha$ or $\beta$ from that equality. This
condition is of course fulfilled, if $P$ contains the expression
from one of the parentheses. However, if $j=s-1$ (so $j+1=s$), by
our assumption, there exists an $i$ such that $a_i,a_j$ do not
commute in $K[M]/P$ and there exists an $l$ such that
$a_{j+1},a_l$ do not commute in $K[M]/P$. Therefore, both
expressions in the parentheses in our equality may not belong to
$P$. Then, if $P$ is to satisfy the above condition for such $i
\geq j+1 = s > s-1 = j \geq m > l$, it must contain $(a_{j+1} a_l
- a_l a_{j+1})a_m$. This means that in $K[M]/P$ the following
equality holds:
\begin{align}
\label{mprawa} a_s a_l a_m = a_l a_s a_m
\end{align}
for every $a_l$ not commuting with $a_{j+1}=a_s$ and every $m$
such that $l < m \leq s-1$. Notice that if $a_l,a_s$ commute in
$K[M]/P$, this equality is of course also satisfied. So, we may
rewrite condition (\ref{mprawa}) in a more general form (using the
relations in $M$):
\begin{align}
\label{mP} a_m a_s a_l = a_l a_s a_m \textup{ \ for $l,m$ such
that \ }  l, m  < s.
\end{align}
Similarly, for equalities of (\ref{32}) type we obtain
\begin{align}
\label{mL} a_m a_{s-1} a_i = a_i a_{s-1} a_m \textup{ \ for $i,m$
such that \ } s-1 < m, i.
\end{align}
Therefore, in this case $P$ must contain all elements $a_m a_s a_l
- a_l a_s a_m$ for $l,m$ such that $l, m  < s$ and $a_m a_{s-1}
a_i - a_i a_{s-1} a_m$ for $i,m$ such that $s-1 < m, i$, as well
as  the previously mentioned elements $a_ma_i-a_ia_m$ for $s \leq
m, i$ and $a_la_m-a_ma_l$ for $l,m \leq s-1$. This means that $P$
contains an \textbf{ideal of \ref{1A} type}. Notice that there are
$n-1$ such ideals (because $1<s \leq n$). Moreover, the element
$a_s a_{s-1}$ is central modulo such an ideal (so also modulo
$P$). Namely, by (\ref{mP}) for $m=s-1$, it commutes in $K[M]/P$
with $a_l$ for $l<s-1$, similarly by (\ref{mL}) for $m=s$, it
commutes in $K[M]/P$ with $a_i$ for $i>s$, an finally, by the
equalities in $M$, it commutes with $a_{s-1}, a_s$.

We have considered all the possible cases. The result follows.
\end{proof}

\subsection{The form of $M_\rho$ for $\rho$ of $\h$ type or of $\diam$ type}
\label{s-postac}

\

\begin{ozn}
\label{Mkreski} For a congruence $\rho $ of type $\h$ or $\diam$,
let $\psi$ be the natural homomorphism $M \rightarrow M_{\rho} =
M/\rho$. We also denote by $\psi$ its natural extension to a map
$K[M] \rightarrow K[M_{\rho}]$. The image of an element $x\in
K[M]$ under $\psi$ will be also denoted by $\ol{x}$. Since $\rho$
is a homogeneous congruence,  in $M_\rho$ we still have a natural
$\Z$-gradation given by the lengths of words, so we may consider
the degrees of elements of $M_\rho$.

For $\rho$ of $\h$ or $\diam$ type, we denote the homomorphism
$\psi$  by $\psi_\h$ or $\psi_\diam$, respectively.

If $\rho$ is of $\heartsuit$ type with the distinguished generator
$a_s$ then by  $M_{n-1}^s$ we denote the Chinese monoid with
generators $a_1,\ldots,a_{s-1},a_{s+1},\ldots,a_n$ and by
$\ol{M_{n-1}^s}$ its image under $\psi|_{M_{n-1}^s}$. Then it is
easy to see that
\[\ol{M_{n-1}^s} = \psi|_{M_{n-1}^s}
(M_{n-1}^s) = M_{n-1}^s \big/ \left(\rho|_{M_{n-1}^s}\right) =
\raisebox{1ex}{$M_{n-1}^s$} \Big/
\raisebox{-1ex}{$\left(\substack{a_1,\ldots,a_{s-1} \text{
commute} \\ a_{s+1},\ldots,a_n \text{ commute}}\right)$}.\]

If $\rho$ is of $\diamondsuit$ type with the distinguished
generators  $a_s, a_{s-1}$, by $M_{n-2}^{s-1,s}$ we denote the
Chinese monoid with generators
$a_1,\ldots,a_{s-2},a_{s+1},\ldots,a_n$ and by
$\ol{M_{n-2}^{s-1,s}}$ its image under $\psi|_{M_{n-2}^{s-1,s}}$.
Then, using Definition~\ref{types-def}, it is easy to see that
\[\ol{M_{n-2}^{s-1,s}} = \psi|_{M_{n-2}^{s-1,s}} (M_{n-2}^{s-1,s}) = M_{n-2}^{s-1,s}
\big/ \left(\rho|_{M_{n-2}^{s-1,s}}\right) =
\raisebox{1ex}{$M_{n-2}^{s-1,s}$}
\Big/ \raisebox{-1ex}{$\left(\substack{a_1,\ldots,a_{s-2} \text{ commute} \\
a_{s+1},\ldots,a_n \text{ commute}}\right)$}.\]
\end{ozn}

\begin{uw}
\label{ascentr} By Definition \ref{types-def}, we know that
$\ol{a_s}$ is central in $K[M_{\rho}]$ for the congruence $\rho$
of $\h$ type with  the distinguished generator $a_s$.
\end{uw}

\begin{lem}
  \label{Mkier}
If the congruence $\rho$ is of $\heartsuit$ type with the
distinguished generator $a_s$,  then the element $\ol{a_s}$ is
regular in $K[M_{\rho}]$. Moreover,
  $M_{\rho} \simeq \ol{M_{n-1}^s} \times \la \ol{a_s} \ra \simeq \ol{M_{n-1}^s} \times \N$.
\end{lem}

\begin{proof}
An easy degree argument shows that the element $\ol{a_s}$ is
non-zero in $K[M_{\rho}]$.

By Remark~\ref{Mkier}, $\ol{a_s}$ is central in $M_{\rho}$ and
every element $w \in M_{\rho}$ is of the form $w=w_0 \cdot
\ol{a_s}^k$, where $w_0 \in \ol{M_{n-1}^s}$ and $k \in \N$.
Therefore $M_{\rho} = \ol{M_{n-1}^s} \cdot \la \ol{a_s} \ra$.

We now introduce in $M_{\rho}$ a new relation $\ol{a_s}=1$. Then
the corresponding image of the whole $M_{\rho}$ coincides with
$\ol{M_{n-1}^s}$ and $w_0$ is the image of $w$. If $w=w_0 \cdot
\ol{a_s}^k$ is equal to some $w'=w'_0 \cdot \ol{a_s}^{k'}$ for
$w'_0 \in \ol{M_{n-1}^s}$, then their images after introducing the
relation $\ol{a_s}=1$ are also equal. Therefore $w_0=w'_0$.

Moreover, if the elements $w,w'$ are equal in $M_{\rho}$, then the
exponents, with  which $\ol{a_s}$ appears in them, must also be
equal (by a degree argument). Therefore the equality $k=k'$ also
holds. So the product $M_{\rho} = \ol{M_{n-1}^s} \cdot \la
\ol{a_s} \ra$ is direct:
\[M_{\rho} \simeq \ol{M_{n-1}^s} \times \la \ol{a_s} \ra \simeq \ol{M_{n-1}^s} \times \N.\]
In particular $\ol{a_s}$ is a regular element in $K[M_{\rho}]$.
\end{proof}

For the future convenience, we reformulate the above lemma,
introducing an additional notation.

\begin{wn}
\label{wnKier} Let $\what{\psi}_{\h} \colon M \rightarrow
\ol{M_{n-1}^s} \times \la \ol{a_s} \ra$ be the homomorphism
defined by:
\[\begin{cases}
\what{\psi}_{\h}(a_s) = (1,\ol{a_s}) \\
\what{\psi}_{\h}(a_i) = (\ol{a_i},1)  \ \ \textup{for} \ i \neq s.
\end{cases}\]
Let $\lambda \colon M_{\rho} \rar \ol{M_{n-1}^s} \times \la
\ol{a_s} \ra$ be the isomorphism  $w \mapsto (w_0, \ol{a_s}^k)$
resulting from the proof of Lemma \ref{Mkier}. Then
$\what{\psi}_{\h}$ is an epimorphism, $ker (\what{\psi}_{\h})= ker
(\psi_{\h})$  and the following diagram commutes:
\[
\xymatrix{
M \ar@{->>}[d]_{\psi_\h} \ar@{->>}[rd]^{\what{\psi}_{\h}} \\
M_{\rho}
\ar@{}[r]|(0.25){\mbox{\normalsize{$\simeq$}}}_(0.25)\lambda &
\ol{M_{n-1}^s} \times \la \ol{a_s} \ra \simeq \ol{M_{n-1}^s}
\times \N }
\]
\end{wn}

\begin{uw}
\label{asas-1centr} By Definition \ref{types-def}, we know that
the element $\ol{a_s} \ol{a_{s-1}}$ is central in $K[M_{\rho}]$
for the congruence $\rho$ of $\diam$ type with distinguished
generators $a_{s-1}, a_s$.
\end{uw}

\begin{lem}
\label{asas-1reg} The element $\ol{a_s} \ol{a_{s-1}}$ is regular
in $K[M_{\rho}]$, where  $\rho$ is of $\diamondsuit$ type with
distinguished generators  $a_{s-1}, a_s$.
\end{lem}

\begin{proof}
A degree argument easily implies that $\ol{a_s} \ol{a_{s-1}} \neq
0$. In view of Remark~\ref{asas-1centr}, it suffices to prove that
for any elements $x,y \in M_{\rho}$, from the equality $\ol{a_s}x
\ol{a_{s-1}}y=0$ it follows that $x=y$.

As in the proof of Theorem \ref{rown}, to simplify notation we
shall write $i$ instead of $\ol{a_i}$ and we shall write $*$
instead of the exponents (they may be equal $0$). Then, by
(\ref{canon}), we know that elements $w \in M$ have the canonical
form
\begin{align*}
 w = & (1)^* \\
  & (21)^* (2)^*\\
  & (31)^* (32)^* (3)^*\\
  &  \ldots \\
 & (n1)^* (n2)^* \ldots (n)^*.
\end{align*}
Relations in $M_{\rho}$ in particular imply that the generators
$1,2, \ldots, s-1$ commute. Therefore, the element $w \in
M_{\rho}$ may be written in the form
\begin{align*}
 w = & (1)^* (2)^* \ldots (s-1)^* \\
  & (s1)^* (s2)^* \ldots (s)^*\\
  & (s+1 \ 1)^* (s+1 \  2)^* \ldots (s+1)^*\\
  &  \ldots \\
 & (n1)^* (n2)^* \ldots (n)^*.
\end{align*}
A presentation in this form is not unique, because for example
$i(sj) = j(si)$ for $i,j<s$,  where the element $j$ commutes with
all $j'$ for $j' \leq s-1$ and the element $si$ commutes with all
$si'$ for $i'<s$. We can therefore perform all such possible
changes for $j<i$, coming to the following form of $w$:
\begin{align*}
 w = & (1)^* (2)^* \ldots (i)^* \\
  & (sj)^* (s \ j+1)^* \ldots (s)^*\\
  & (s+1 \ 1)^* (s+1 \  2)^* \ldots (s+1)^*\\
  &  \ldots \\
 & (n1)^* (n2)^* \ldots (n)^*,
\end{align*}
where the first or the second row (or both) may disappear or in
the second row only  $(s)^*$ may be left. If in both of these rows
some elements with non-zero exponents are left (other than
$(s)^*$), then $w$ may be written in the above form with $i,j$
satisfying the condition  $i \leq j \leq s-1$ and the exponents of
$(i)$ and of $(sj)$ are positive.

By the commutativity of elements $s,s+1, \ldots,n$, each segment
of the form $(t1)^* (t2)^* \ldots (t)^*$ for $t>s$, can be
replaced by a segment of the form $(t1)^* (t2)^* \ldots (t \
s-1)^* (s)^* (s+1)^* \ldots (t)^*$. Moreover, $(s)^* (s+1)^*
\ldots (t)^*$ commutes with every product $(t'1)^* (t'2)^* \ldots
(t' \ s-1)^*$ for any $t'>t$, and $(s)^*$ commutes with the
product $(s+1 \ 1)^* (s+1 \ 2)^* \ldots (s+1 \ s-1)^*$. Hence, $w$
can be rewritten as
\begin{align*}
 w = & (1)^* (2)^* \ldots (i)^* \\
  & (sj)^* (s \ j+1)^* \ldots (s \ s-1)^*\\
  & (s+1 \ 1)^* (s+1 \  2)^* \ldots (s+1 \ s-1)^*\\
  &  \ldots \\
 & (n1)^* (n2)^* \ldots (n \ s-1)^*\\
 & (s)^*(s+1)^* \ldots (n)^*,
\end{align*}
where the first or second row (or both) may disappear. If in both
of these rows some elements with non-zero exponents are left, then
 $w$ may be written in the above form with $i,j$ satisfying the
condition  $i \leq j \leq s-1$ and the exponents of $(i)$ and of
$(sj)$ are positive.

Each element of the form $(t \ s-1)^*$, for $t \geq s$, commutes
with all elements of the form $(t' \ t'')^*$ for $t' > t$ and $t''
\leq s-1$. Thus, we obtain another form of $w$:
\begin{align*}
 w = & (1)^* (2)^* \ldots (i)^* \\
  & (sj)^* (s \ j+1)^* \ldots (s \ s-2)^*\\
  & (s+1 \ 1)^* (s+1 \  2)^* \ldots (s+1 \ s-2)^*\\
  &  \ldots \\
 & (n1)^* (n2)^* \ldots (n \ s-2)^*\\
 & (s \ s-1)^* (s+1 \ s-1)^* \ldots (n \ s-1)^*\\
 & (s)^*(s+1)^* \ldots (n)^*,
\end{align*}
where the first or second row (or both) may disappear. If in both
of these rows some elements with non-zero exponents are left, then
$w$ may be written in the above form with $i,j$ satisfying the
condition  $i \leq j \leq s-2$ and the exponents of $(i)$ and of
$(sj)$ are positive. If some elements remain only in the first
row, then in the above form we have $0< i \leq s-1$, if only in
the second row, then $0< j \leq s-2$.

Note that this form is not unique, for example  $(k \ s-1)l =(l \
s-1) k$ for $k,l>s-1$, where the element $(l \ s-1)$ commutes with
all $(l' \ s-1)$ for $l' > s-1$ and the element $k$ commutes with
all $k' \geq s$. Thus we can perform all such possible changes for
$k>l$, coming to the following form of $w$:
\begin{align*}
\label{postac} \tag{$\star$}
 w = & (1)^* (2)^* \ldots (i)^* \\
  & (sj)^* (s \ j+1)^* \ldots (s \ s-2)^*\\
  & (s+1 \ 1)^* (s+1 \  2)^* \ldots (s+1 \ s-2)^*\\
  &  \ldots \\
 & (n1)^* (n2)^* \ldots (n \ s-2)^*\\
 & (s \ s-1)^* (s+1 \ s-1)^* \ldots (k \ s-1)^*\\
 & (l)^*(l+1)^* \ldots (n)^*,
\end{align*}
where the first or second row (or both) may disappear. If in both
of these rows some elements with non-zero exponents are left, then
$w$ may be written in the above form with $i,j$ satisfying $i \leq
j \leq s-2$ and the exponents of $(i)$ and of $(sj)$ are positive.
If some elements remain in only one of the first two rows, in the
above form we have $i \leq s-1$ or $j \leq s-2$. Moreover, the
last or the second last row (or both of them) may also disappear.
If in both of these rows some elements remain with non-zero
exponents, $w$ may be written in the above form with $k,l$
satisfying $s \leq k \leq l$. If some elements remain in only one
of the last two rows, then in the above form we have $k \geq s$ or
$l \geq s$.

For example, for $s=2$ the above algorithm leads to the following
form of $w$:
\begin{align*}
 w = & (1)^* \\
  & (21)^* (31)^* \ldots (k1)^*\\
 & (l)^*(l+1)^* \ldots (n)^*,
\end{align*}
where the conditions concerning the last two rows and the values
of $k$ and $l$ are similar to those described above for an
arbitrary $s$.

We shall prove by induction on $n$ that the form (\ref{postac}) of
$w$ is unique. By assumption $n \geq 3$.

First, consider the case $n=s=3$. We get
\begin{align*}
 w = & (1)^* (2)^* \\
 & (3 \ 1)^*\\
 & (3 \ 2)^*\\
 & (3)^*,
\end{align*}
where at least one among the exponents of $(2)$, $(3 \ 1)$ is
zero. Let \[w = (1)^x (2)^y (3 \ 1)^z (3 \ 2)^t (3)^u, \ \ \ W =
(1)^X (2)^Y (3 \ 1)^Z (3 \ 2)^T (3)^U,\] where either $y=0$ or
$z=0$ and either $Y=0$ or $Z=0$. Assume that $w=W$.

If $\rho$ is of type $\diam$, the relations introduced by
factoring by $\rho$ are $\ol{a_1}\ol{a_2}=\ol{a_2}\ol{a_1}$ and
$\ol{a_1}\ol{a_3}\ol{a_2}=\ol{a_2}\ol{a_3}\ol{a_1}$. By
introducing in $M_\rho$ a new relation $\ol{a_1}=\ol{a_2}$, we
obtain a homomorphism of $M_\rho$ into the Chinese monoid $M_2 =
\la \ol{a_1}, \ol{a_3} \ra$. So, we can use the canonical forms in
$M_2$ of the images of $w$ and $W$, by comparing the corresponding
exponents. The image of $w$ is $(1)^{x+y} (3 \ 1)^{z+t} (3)^u$,
where either $y=0$ or $z=0$, the image of $W$ is $(1)^{X+Y} (3 \
1)^{Z+T} (3)^U$, where either $Y=0$ or $Z=0$. Therefore,
\[\begin{cases}
x+y = X+Y \\
z+t = Z+T \\
u = U.
\end{cases}\]
On the other hand, comparing the degrees of $w$ and $W$ with
respect to the generators, we obtain
\[\begin{cases}
x+z = X+Z \\
y+t = Y+T \\
z+t+u = Z+T+U.
\end{cases}\]
If $x,y,z,t,u$ are known and the equality $Y=0$ holds, we may
calculate $U=u$, $T=y+t$, $X=x+y$ oraz $Z=z-y$. Since either $y=0$
or $z=0$ and the exponents are non-negative, the equality $Z=z-y$
yields $y=0$. Therefore $Z=z$, $T=t$, $X=x$, so all exponents in
$w$ and $W$ are equal. The case of $Z=0$ is similar. Therefore,
for $n=s=3$ the form (\ref{postac}) is unique, as was claimed.

A similar proof works in the case where $n=3$ and $s=2$.

Thus, assume that $n>3$ and that the above form (\ref{postac}) is
unique for all Chinese monoids of rank $m<n$ and all congruences
of type $\diam$ defined on them.

Assume first that $s \geq 3$. Then, as we already know, $w$ is of
the form
\begin{align*}
\tag{$\star$}
 w = & (1)^{\alpha_{1}} (2)^{\alpha_{2}} \ldots (i)^{\alpha_{i}} \ldots (s-1)^{\alpha_{s-1}}\\
  & (s1)^{\alpha_{s1}} \ldots (sj)^{\alpha_{sj}} (s \ j+1)^{\alpha_{s \ j+1}} \ldots (s \ s-2)^{\alpha_{s \ s-2}}\\
  & (s+1 \ 1)^{\alpha_{s+1 \ 1}} (s+1 \  2)^{\alpha_{s+1 \ 2}} \ldots (s+1 \ s-2)^{\alpha_{s+1 \ s-2}}\\
  &  \ldots \\
 & (n1)^{\alpha_{n1}} (n2)^{\alpha_{n2}} \ldots (n \ s-2)^{\alpha_{n \ s-2}}\\
 & (s\ s-1)^{\alpha_{s\ s-1}} (s+1\ s-1)^{\alpha_{s+1\ s-1}} \ldots (k\ s-1)^{\alpha_{k \ s-1}}\ldots(n\ s-1)^{\alpha_{n \ s-1}}\\
 &  (s)^{\alpha_{s}} \ldots (l)^{\alpha_{l}} (l+1)^{\alpha_{l+1}} \ldots (n)^{\alpha_{n}} ,
\end{align*}
where the following conditions (\ref{war2}) and (\ref{war3}) hold:
\begin{align*}
\label{war2} \tag{$\star \star$}
& \alpha_{1}=\alpha_{2}=\ldots=\alpha_{s-1} =0  \text{ or}  \\
& \alpha_{s1}=\alpha_{s2}=\ldots=\alpha_{s \ s-2} =0  \text{ or} \\
& \text{in the first and in the second row there exist non-zero exponents and} \\
& \begin{cases}
\alpha_{i+1}=\alpha_{i+2}=\ldots=\alpha_{s-1} =0 \\
\alpha_{s1}=\alpha_{s2}=\ldots=\alpha_{s \ j-1} =0, \ \text{ where
} i \leq j \leq s-2
\end{cases}
\end{align*}
and
\begin{align*}
\label{war3} \tag{$\star \star \star$}
& \alpha_{s \ s-1}=\alpha_{s+1 \ s-1}= \ldots = \alpha_{n \ s-1}=0 \text{ or} \\
& \alpha_{s}=\alpha_{s+1}=\ldots=\alpha_n=0 \text{ or} \\
& \text{in the last and in the second last row there exist non-zero exponents and} \\
& \begin{cases}
\alpha_{k+1 \ s-1}=\alpha_{k+2 \ s-1}=\ldots=\alpha_{n \ s-1}=0 \\
\alpha_{s}=\alpha_{s+1}=\alpha_{l-1}=0,  \ \text{ where } s \leq k
\leq l.
\end{cases}
\end{align*}

Let us now introduce a new relation $\ol{a_1}=\ol{a_2}$ in
$M_{\rho}$. Then the relations in the new monoid
$M_{\rho}/(\ol{a_1}=\ol{a_2})$ are exactly the same as in the
Chinese monoid of rank $n-1$ with generators $\ol{a_2}, \ldots,
\ol{a_n}$ and with relations of $\diamondsuit$ type for the
distinguished  elements $\ol{a_{s-1}}, \ol{a_s}$. Using notation
of \ref{Mkreski}, we get natural isomorphisms
$M_{\rho}/(\ol{a_1}=\ol{a_2}) \simeq \ol{M_{2,\ldots,n}^{s-1, s}}
\simeq \ol{M_{n-1}^{s-2, s-1}}$.

In the new monoid $M_{\rho}/(\ol{a_1}=\ol{a_2})$, the image of $w$
is
\begin{align*}
 \tilde{w} = & (2)^{\alpha_{1}+\alpha_{2}} \ldots (i)^{\alpha_{i}} \ldots (s-1)^{\alpha_{s-1}}\\
  & (s2)^{\alpha_{s1}+\alpha_{s2}} \ldots (sj)^{\alpha_{sj}} (s \ j+1)^{\alpha_{s \ j+1}} \ldots (s \ s-2)^{\alpha_{s \ s-2}}\\
  & (s+1 \  2)^{\alpha_{s+1 \ 1}+ \alpha_{s+1 \ 2}} \ldots (s+1 \ s-2)^{\alpha_{s+1 \ s-2}}\\
  &  \ldots \\
 & (n2)^{\alpha_{n1}+\alpha_{n2}} \ldots (n \ s-2)^{\alpha_{n \ s-2}}\\
 & (s\ s-1)^{\alpha_{s\ s-1}} (s+1 \ s-1)^{\alpha_{s+1 \ s-1}} \ldots (k \ s-1)^{\alpha_{k \ s-1}}\ldots(n\ s-1)^{\alpha_{n \ s-1}}\\
 & (s)^{\alpha_{s}} \ldots (l)^{\alpha_{l}} (l+1)^{\alpha_{l+1}} \ldots (n)^{\alpha_{n}} ,
\end{align*}
where conditions (\ref{war2}) and (\ref{war3}) hold.

Assume we have an element $w'$ such that $w'=w$ in $M_\rho$, which
is also written in the form (\ref{postac}), with exponents denoted
respectively by $\beta_x$ or $\beta_{xy}$. To prove that the form
(\ref{postac}) is unique, we need to show that the corresponding
exponents $\alpha$ and $\beta$ are equal. Let the image of $w'$,
after introducing the relation $\ol{a_1}=\ol{a_2}$, be equal to
\begin{align*}
 \tilde{w}' = & (2)^{\beta_1+\beta_2} \ldots (i')^{\beta_{i'}}\ldots (s-1)^{\beta_{s-1}} \\
  & (s2)^{\beta_{s1}+\beta_{s2}} \ldots  (sj')^{\beta_{sj'}} (s \ j'+1)^{\beta_{s \ j'+1}} \ldots (s \ s-2)^{\beta_{s \ s-2}}\\
  & (s+1 \  2)^{\beta_{s+1 \ 1}+ \beta_{s+1 \ 2}} \ldots (s+1 \ s-2)^{\beta_{s+1 \ s-2}}\\
  &  \ldots \\
 & (n2)^{\beta_{n1}+\beta_{n2}} \ldots (n \ s-2)^{\beta_{n \ s-2}}\\
 & (s \ s-1)^{\beta_{s \ s-1}} (s+1 \ s-1)^{\beta_{s+1 \ s-1}} \ldots (k' \ s-1)^{\beta_{k' \ s-1}}\ldots(n\ s-1)^{\beta_{n \ s-1}}\\
 & (s)^{\beta_s} \ldots  (l')^{\beta_{l'}} (l'+1)^{\beta_{l'+1}} \ldots (n)^{\beta_n} ,
\end{align*}
where conditions analogous to (\ref{war2}) and (\ref{war3}) hold.

Since $w=w'$, also $\tilde{w}=\tilde{w}'$. Now, $\tilde{w}$ and
$\tilde{w}'$ are elements of the monoid $\ol{M_{2,\ldots,n}^{s-1,
s}}$. As noted before, this monoid is defined by relations of type
$\diam$ on a Chinese monoid on $n-1$ generators. Therefore, by the
induction hypothesis, $w$ and $w'$ are uniquely presented in the
form (\ref{postac}). Thus $i=i'$, $j=j'$, $k=k'$, $l=l'$ (if they
exist) and

{\small{
\begin{align*}
\hspace{-0cm} \alpha_{1}+\alpha_{2} = & \beta_1+\beta_2    & \ldots & & \alpha_{i} = & \beta_{i'}  & \ldots &  & \alpha_{s-1} = & \beta_{s-1} \\
\hspace{-0cm} \alpha_{s1}+\alpha_{s2} = & \beta_{s1}+\beta_{s2}  &  \ldots & & \alpha_{sj} = & \beta_{sj'} & \ldots  &  & \alpha_{s \ s-2} = & \beta_{s \ s-2}\\
\hspace{-0cm} \alpha_{s+1 \ 1}+ \alpha_{s+1 \ 2} = & \beta_{s+1 \ 1}+ \beta_{s+1 \ 2} &\ldots& &\ldots& &\ldots& & \alpha_{s+1 \ s-2} = & \beta_{s+1 \ s-2}\\
\hspace{-0cm} \ldots & & \ldots & &  \ldots & & \ldots  & & \ldots & \\
\hspace{-0cm} \alpha_{n1}+\alpha_{n2} =  & \beta_{n1}+\beta_{n2} & \ldots &            & \ldots  & & \ldots & & \alpha_{n \ s-2} = & \beta_{n \ s-2}\\
\hspace{-0cm} \alpha_{s \ s-1} = & \beta_{s \ s-1} & \ldots & & \alpha_{k \ s-1} = & \beta_{k' \ s-1} & \ldots & & \alpha_{n \ s-1} = & \beta_{n \ s-1}\\
\hspace{-0cm} \alpha_{s} = & \beta_s & \ldots & & \alpha_{l} = &
\beta_{l'} & \ldots  & & \alpha_{n} = & \beta_n,
\end{align*}
}} \noindent where conditions (\ref{war2}) and (\ref{war3}) hold.

Suppose $s>3$. Then we may introduce the relation
$\ol{a_2}=\ol{a_3}$ in $M_{\rho}$ instead of the relation
$\ol{a_1}=\ol{a_2}$. We then obtain a system of equations similar
to the one above. It can easily be checked that this system,
combined with the one above, leads to the conclusion that all
corresponding exponents $\alpha$ and $\beta$ are equal. Therefore
the form (\ref{postac}) of $w$ is unique.

The last case to consider is $s \leq 3$.  Suppose first that
$s<n-1$.

We may then impose a new relation $\ol{a_{n-1}}=\ol{a_n}$ in
$M_{\rho}$. As in the case considered above, using the induction
hypothesis and the commutativity of all elements of the form $nx$
for arbitrary $x$, we conclude that $i=i'$, $j=j'$, $k=k'$, $l=l'$
(if they exist) and that

{\small{
\begin{align*}
\hspace{-0.5cm} \alpha_{1} = & \beta_1     & \ldots &  & \alpha_{i} = & \beta_{i'} & \ldots & & \alpha_{s-1} = & \beta_{s-1} \\
\hspace{-0.5cm} \alpha_{s1} = & \beta_{s1} & \ldots  & &  \alpha_{sj} = & \beta_{sj'}  & \ldots  &  & \alpha_{s \ s-2} = & \beta_{s \ s-2}\\
\hspace{-0.5cm} \alpha_{s+1 \ 1} = & \beta_{s+1 \ 1} &\ldots& &\ldots& &\ldots& & \alpha_{s+1 \ s-2} = & \beta_{s+1 \ s-2}\\
\hspace{-0.5cm} \ldots & & \ldots & &  \ldots & & \ldots  & & \ldots & \\
\hspace{-0.5cm} \alpha_{n-1 \ 1} + \alpha_{n1} =  & \beta_{n-1 \ 1} + \beta_{n1} & \ldots & & \ldots  & & \ldots & & \alpha_{n-1 \ s-2} + \alpha_{n \ s-2} = & \beta_{n-1 \ s-2} + \beta_{n \ s-2}\\
\hspace{-0.5cm} \alpha_{s \ s-1} = & \beta_{s \ s-1} & \ldots & & \alpha_{k \ s-1} = & \beta_{k' \ s-1} & \ldots & & \alpha_{n \ s-1} = & \beta_{n \ s-1}\\
\hspace{-0.5cm} \alpha_{s} = & \beta_s & \ldots & &  \alpha_{l} =
& \beta_{l'} & \ldots  & & \alpha_{n-1} + \alpha_n = & \beta_{n-1}
+ \beta_n,
\end{align*}
}} \noindent where conditions (\ref{war2}) and (\ref{war3}) hold.

We may also impose in $M_{\rho}$ the relation $\ol{a_{n-2}} =
\ol{a_{n-1}}$ instead of $\ol{a_{n-1}} = \ol{a_n}$. We then get a
similar system of equations which, combined with the one above,
implies that all corresponding exponents $\alpha$ and $\beta$ are
equal. Thus the form (\ref{postac}) of $w$ is indeed unique.

It remains to consider the case where $n-1 \leq s$ and $s \leq 3$.
Then $n \leq 4$. For $n=3$ we know by the induction hypothesis
that the claim is true. For $n=4$ we have to consider only the
case $4-1 \leq s \leq 3$, so $s=3$.

Introducing in $M_4/\rho$ the relation $\ol{a_1}=\ol{a_2}$
leads to a system of equations as above. Also, independently, we
may introduce the relation $\ol{a_3} = \ol{a_4}$, which leads to
another system of a similar type. These two systems of equations
combined easily lead to the conclusion that
 \[\begin{cases}
   \alpha_{31} + \alpha_{32} = \beta_{31} + \beta_{32} \\
   \alpha_{41} + \alpha_{42} = \beta_{41} + \beta_{42} \\
   \alpha_{31} + \alpha_{41} = \beta_{31} + \beta_{41} \\
   \alpha_{32} + \alpha_{42} = \beta_{32} + \beta_{42}
 \end{cases}\]
and all other exponents $\alpha, \beta$ are, respectively, equal.
Furthermore, either $\alpha_{2}=0$ or $\alpha_{31}=0$ and either
$\alpha_{3}=0$ or $\alpha_{42}=0$, similarly for $\beta$.

Now we introduce in $M_4/\rho$ the relations $\ol{a_2} \ol{a_x} =
\ol{a_x} \ol{a_2}$ and $\ol{a_3} \ol{a_x} = \ol{a_x} \ol{a_3}$ for
$x = 1,2,3,4$. We obtain
\[\tilde{M} =
\raisebox{1ex}{$M_4/\rho$} \Big/
\raisebox{-1ex}{$(\substack{\ol{a_2} \text{ central}
\\ \ol{a_3} \text{ central}})$}
\simeq M_2 \times \la \ol{a_2} \ra \times \la \ol{a_3} \ra,\]
where $\ol{a_1}$ and $\ol{a_4}$ are the generators of $M_2$. In
$M_2$ we have the canonical form of elements, so the elements of
$M_2 \times \la \ol{a_2} \ra \times \la \ol{a_3} \ra$ may be
written in the canonical form $(1)^*(41)^*(4)^*(2)^*(3)^*$.
Therefore, for the any elements $w=w'$ in $M_4/\rho$ and their
images $\tilde{w}=\tilde{w}' \in \tilde{M}$, comparing the
exponents in the canonical forms $\tilde{w}$ and $\tilde{w}'$ we
obtain in particular $\alpha_{41} = \beta_{41}$. Combined with the
system of equations obtained above for exponents $\alpha$ and
$\beta$, this equality leads to the conclusion that all respective
exponents are equal. Therefore the forms $w$ and $w'$ are
identical, so the form (\ref{postac}) is indeed unique also in
this case.

Therefore, we have a unique form (\ref{postac}) of elements in
$M_{\rho}$. We shall now prove that the element $s \ (s-1)$ is
regular. Suppose that for some elements $w, w'$ the equality $s \
(s-1) \ w = s \ (s-1) \ w'$ holds and that the exponents of the
element $s \ (s-1)$ in elements $w, w'$ written in the form
(\ref{postac}) are $\alpha_{s \ s-1}$ and $\beta_{s \ s-1}$,
respectively. By Definition \ref{types-def}, the element $\ol{a_s
a_{s-1}}$, denoted here by $s \ (s-1)$, is central in $M_{\rho}$.
Therefore the exponents of the element $s \ (s-1)$ in  $s \ (s-1)
\ w$ and $s \ (s-1) \ w'$, written in the form (\ref{postac}), are
equal to $\alpha_{s \ s-1}+1$ and $\beta_{s \ s-1}+1$,
respectively. Since, by assumption, $s \ (s-1) \ w = s \ (s-1) \
w'$ and the form (\ref{postac}) is unique, the equality $\alpha_{s
\ s-1}+1 = \beta_{s \ s-1}+1$ holds, so $\alpha_{s \ s-1} =
\beta_{s \ s-1}$ also holds. All other exponents in the canonical
forms of the elements $s \ (s-1) \ w$ and $s \ (s-1) \ w'$ are the
same as in the elements $w$ and $w'$, so they are equal. Hence
$w=w'$, which means that the element $s \ (s-1)$ is left regular.
Since it is central, the assertion follows.
\end{proof}

Note that Remark~\ref{asas-1centr} and Lemma~\ref{asas-1reg} imply
that for $\rho$ of $\diam$ type we can consider the central
localization $M_{\rho} \la (\ol{a_s} \ol{a_{s-1}})^{-1} \ra$.

\begin{lem}
  \label{Mkaro}
If $\rho$ is a congruence of  $\diamondsuit$ type with
distinguished generators  $a_{s-1},a_s$, then there is an
isomorphism
\[M_{\rho} \la (\ol{a_s} \ol{a_{s-1}})^{-1} \ra \simeq \ol{M_{n-2}^{s-1,s}} \times B \times \Z\]
which is the natural extension of the map $\lambda \colon M_{\rho}
\rar \ol{M_{n-2}^{s-1,s}} \times B \times \Z$ defined by
\[\begin{cases}
\lambda(\ol{a_{s-1}}) = (1,p,g)  \\
\lambda(\ol{a_s}) = (1,q,1) \\
\lambda(\ol{a_i}) = (\ol{a_i},p,1)  \ \ \textup{for} \ i < s-1 \\
\lambda(\ol{a_j}) = (\ol{a_j},q,1)  \ \ \textup{for} \  j > s.
\end{cases}\]
\end{lem}

\begin{proof}
Define the transformation $\what{\psi}_{\diam} \colon M
\rightarrow \ol{M_{n-2}^{s-1,s}}\times B \times \Z$ by
\[\begin{cases}
\what{\psi}_{\diam}(a_{s-1}) = (1,p,g)  \\
\what{\psi}_{\diam}(a_s) = (1,q,1) \\
\what{\psi}_{\diam}(a_i) = (\ol{a_i},p,1)  \ \ \textup{for} \ i < s-1 \\
\what{\psi}_{\diam}(a_j) = (\ol{a_j},q,1)  \ \ \textup{for} \  j >
s.
\end{cases}\]
We will prove that $\what{\psi}_{\diam}$ is a homomorphism by
checking each  coordinate separately.

In the first coordinate we have the transformation $a_{s-1}
\mapsto 1$, $a_s \mapsto 1$, $a_l \mapsto \ol{a_l} =
\psi_\diam(a_l)$ for $l \neq s-1, s$. We shall check that under
this map the images of generators satisfy all the relations
satisfied by the generators, i.e. the defining relations of $M$:
$a_ia_ja_k = a_ia_ka_j=a_ja_ia_k$ for $i \geq j \geq k$.

If all three indices $i, j, k$ are equal to $s-1$ or $s$, then all
the images of generators  are equal 1 and satisfy all the
relations.

If exactly two among the indices $i, j, k$ are equal $s-1$ or $s$,
then on both sides of each of the relations there is only the
image of the generator with the third index. Therefore all the
relations are satisfied.

If only the index $i$ is equal $s-1$ or $s$, then
\begin{align*}
a_ia_ja_k & \mapsto  1 \ol{a_j} \ol{a_k},\\
a_ia_ka_j & \mapsto  1 \ol{a_k} \ol{a_j},\\
a_ja_ia_k & \mapsto  \ol{a_j} 1 \ol{a_k}
\end{align*}
and all the images are equal, because, by the definition of
$\psi_\diam$,  $\ol{a_j}$ and $\ol{a_k}$ commute. Similarly, one
verifies the case where only the index $k$ is equal to $s-1$ or
$s$.

If only the index $j$ is equal $s-1$ or $s$, then
\begin{align*}
a_ia_ja_k & \mapsto  \ol{a_i} 1 \ol{a_k},\\
a_ia_ka_j & \mapsto  \ol{a_i} \ol{a_k} 1,\\
a_ja_ia_k & \mapsto  1 \ol{a_i} \ol{a_k}
\end{align*}
and in this case also all the images are equal.

If none of the indices $i, j, k$ is equal to $s-1$ or $s$, then
the images of the  generators satisfy the respective relations,
because $\psi_\diam$ is a homomorphism. This completes the proof
of the fact that the first coordinate of $\what{\psi}_{\diam}$ is
a homomorphism.

In the second coordinate we have  $a_i \mapsto p$ for $i \leq
s-1$,  $a_j \mapsto q$ for $j \geq s$. As above, we verify that
under this map, the images of the generators satisfy all the
defining relations of $M$. Let $i\geq j\geq k$. If $k \geq s$,
then the image of each of the elements $a_ia_ja_k$, $a_ia_ka_j$,
$a_ja_ia_k$ is equal to $q^3$, thus these images are equal.
Similarly, if  $i \leq s-1$, then the images are equal to $p^3$.
If $i \geq s > s-1 \geq j \geq k$, then
\begin{align*}
a_ia_ja_k \mapsto qpp &= p,\\
a_ia_ka_j \mapsto qpp &= p,\\
a_ja_ia_k \mapsto pqp &= p,
\end{align*}
thus also all the images are equal. Similarly, for $i \geq j \geq
s> s-1 \geq k$ all the images are equal to $q$. Therefore the
second coordinate of $\what{\psi}_{\diam}$ is indeed a
homomorphism.

The third coordinate of $\what{\psi}_{\diam}$ is a homomorphism as
well, because all the relations in $M$ are homogeneous with
respect to $a_{s-1}$. This completes the proof of the fact that
$\what{\psi}_{\diam}$ is a homomorphism.

We shall now prove that $ker (\psi_{\diam}) \ssq ker
(\what{\psi}_{\diam})$. So, assume that
$\psi_\diam(x)=\psi_\diam(y)$ for some $x,y \in M$. We will show
that also $\what{\psi}_{\diam}(x)=\what{\psi}_{\diam}(y)$. It
suffices to prove that $\what{\psi}_{\diam}(x) =
\what{\psi}_{\diam}(y)$ for pairs $(x,y)$ generating $\rho$.

For $i,k<s-1$ and the pair $(a_ia_k, a_ka_i)$, we obtain
\[\what{\psi}_{\diam}(a_ia_k) = (\ol{a_i},p,1) \cdot (\ol{a_k},p,1) = (\ol{a_ia_k},p^2,1)\]
and similarly $\what{\psi}_{\diam}(a_ka_i) = (\ol{a_ka_i},p^2,1).$
Therefore $\what{\psi}_{\diam}(a_ia_k) =
\what{\psi}_{\diam}(a_ka_i)$, as claimed. The proof is similar for
$j,l>s$ and the pair $(a_ja_l, a_la_j)$.

For $i<s-1$ and the pair $(a_ia_{s-1}, a_{s-1}a_i)$, we obtain
\[\what{\psi}_{\diam}(a_ia_{s-1})=(\ol{a_i},p,1) \cdot (1,p,g) = (\ol{a_i},p^2,g) =
\what{\psi}_{\diam}(a_{s-1}a_i).\] Similarly for $j>s$ and the
pair $(a_ja_s, a_sa_j)$.

For $i,k<s-1$ and the pair $(a_ia_sa_k, a_ka_sa_i)$, we obtain
\[\what{\psi}_{\diam}(a_ia_sa_k) = (\ol{a_i},p,1) \cdot (1,q,1) \cdot (\ol{a_k},p,1) =
(\ol{a_ia_k},pqp,1) = (\ol{a_ia_k},p,1) =
\what{\psi}_{\diam}(a_ka_sa_i).\]  Similarly for $j,l>s$ and the
pair $(a_ja_{s-1}a_l, a_la_{s-1}a_j)$. If in the above cases
$i=s-1$, $k=s-1$, $j=s$ or $l=s$, the proof is analogous.

We have thus completed the proof of the fact that $ker
(\psi_{\diam}) \ssq ker (\what{\psi}_{\diam})$.

Therefore,  $\what{\psi}_{\diam}$ can be presented as the
composition of the epimorphism $\psi_\diam$ and some homomorphism
$\lambda$:
\[
\xymatrix{
M \ar@{->>}[d]_{\psi_\diam} \ar[r]^(0.3){\what{\psi}_{\diam}} & \ol{M_{n-2}^{s-1,s}}\times B \times \Z \\
M_{\rho} \ar@{-->}[ur]_{\lambda} }
\]
The element $\ol{a_s}\ol{a_{s-1}}$ is central in $M_{\rho}$ and
from Lemma \ref{asas-1reg}  we know it is regular, therefore we
can consider the localization $M_{\rho}
\la(\ol{a_s}\ol{a_{s-1}})^{-1}\ra$. The image
$\lambda(\ol{a_sa_{s-1}}) = (1,1,g)$ is an invertible element in
$\ol{M_{n-2}^{s-1,s}}\times B \times \Z$. Hence, we may consider
the natural extension $\lambda'$ of $\lambda $ to the localization
$M_{\rho} \la (\ol{a_s} \ol{a_{s-1}})^{-1} \ra$.

We shall check that $\lambda'$ is an epimorphism. We have
\[\lambda'(\ol{a_sa_{s-1}}) = (1,1,g) \ \ \textup{and} \ \
\lambda'((\ol{a_sa_{s-1}})^{-1})=(1,1,g^{-1}),\] so also
\[\lambda'(\ol{a_{s-1}}(\ol{a_sa_{s-1}})^{-1})=(1,p,g)(1,1,g^{-1})=(1,p,1).\]
Moreover  \[\lambda'(\ol{a_s})=(1,q,1),\] so also
\[\lambda'(\ol{a_s}\ol{a_i})=(1,q,1)(\ol{a_i},p,1)=(\ol{a_i},1,1) \ \ \textup{for} \ \ i<s-1\] and \[\lambda'(\ol{a_j}(\ol{a_{s-1}}(\ol{a_sa_{s-1}})^{-1}))=(\ol{a_j},q,1)(1,p,1)=(\ol{a_j},1,1) \ \ \textup{for}\ \  j>s.\]
Therefore, in the image of $\lambda'$ we can obtain any value in
each of the three coordinates separately. Thus $\lambda'$ is an
epimorphism.

Next, we prove that $\lambda'$ determines an isomorphism $\la
\ol{a_{s-1}}, \ol{a_s}, (\ol{a_s}\ol{a_{s-1}})^{-1} \ra \simeq B
\times \Z$. Notice that by the definition of $M_n$ and
$\psi_\diam$, we have $\la \ol{a_{s-1}}, \ol{a_s} \ra \simeq \la
a_{s-1}, a_s \ra \simeq M_2$, because $\psi_\diam|_{\la  a_{s-1},
a_s \ra}$ is trivial. Let $M'_2 = \la \ol{a_{s-1}}, \ol{a_s} \ra$,
so that $M'_2 \simeq M_2$. The localization of $\la \ol{a_{s-1}},
\ol{a_s} \ra$ with respect to $\la (\ol{a_s}\ol{a_{s-1}})^{-1}
\ra$ is $\la \ol{a_{s-1}}, \ol{a_s}, (\ol{a_s}\ol{a_{s-1}})^{-1}
\ra$, so
\[\la \ol{a_{s-1}}, \ol{a_s}, (\ol{a_s}\ol{a_{s-1}})^{-1} \ra =
M'_2 \la (\ol{a_s}\ol{a_{s-1}})^{-1} \ra.\] Consider the
restriction of $\lambda'$ to $\la \ol{a_{s-1}}, \ol{a_s},
(\ol{a_s}\ol{a_{s-1}})^{-1} \ra$. We know that
\[\lambda'(\ol{a_{s-1}})=(1,p,g), \ \ \lambda'(\ol{a_s})=(1,q,1),
\ \ \lambda'((\ol{a_sa_{s-1}})^{-1})=(1,1,g^{-1}),\] therefore, a
proof similar to the  one above shows that this restriction of
$\lambda'$ is an epimorphism onto $\{1\} \times B \times \Z$.

We shall check that it is also an injection. Each element $w \in
\la \ol{a_{s-1}}, \ol{a_s}, (\ol{a_s}\ol{a_{s-1}})^{-1} \ra = M'_2
\la (\ol{a_s}\ol{a_{s-1}})^{-1} \ra$ can be written in the form
$w=\ol{a_{s-1}}^k \ol{a_s}^l (\ol{a_s}\ol{a_{s-1}})^{-m}$, where
$k,l \in \N$, $m \in \Z$. We then obtain $\lambda'(w) = (1,
p^kq^l, g^{k-m})$. If for some element $v=\ol{a_{s-1}}^{k'}
\ol{a_s}^{l'} (\ol{a_s}\ol{a_{s-1}})^{-m'}$ the equality
$\lambda'(v) = \lambda'(w) =  (1, p^kq^l, g^{k-m})$ holds, then
from the uniqueness of the canonical forms of elements of $B$ and
$\Z$ it follows that $k=k'$, $l=l'$, $k-m=k'-m'$, so also $m=m'$
and thus $w=v$.

Therefore the considered restriction of $\lambda'$ is indeed an
injection. Since we know it is a surjection, it is an isomorphism
and thus
\[M'_2 \la (\ol{a_s}\ol{a_{s-1}})^{-1} \ra = \la \ol{a_{s-1}}, \ol{a_s},
(\ol{a_s}\ol{a_{s-1}})^{-1} \ra  \nad{\lambda'}{\simeq} \{1\}
\times B \times \Z \simeq B \times \Z.
\]

We shall now prove that \[M_{\rho} \la (\ol{a_s}
\ol{a_{s-1}})^{-1} \ra \simeq \ol{C} \times B \times \Z,\] where
\[C \nadrow{def} \la a_sa_i, a_ja_{s-1} \colon  i<s-1; s<j \ra
\subseteq M.\] First we check that
\[M_{\rho} \la(\ol{a_s}\ol{a_{s-1}})^{-1}\ra = \ol{C} \cdot \la \ol{a_{s-1}}, \ol{a_s}, (\ol{a_s}\ol{a_{s-1}})^{-1} \ra.\]
By the relations in the Chinese monoid, the following equalities
hold in $M$:
\begin{align*}
  (a_sa_i)a_{s-1} = a_{s-1}(a_sa_i), \ (a_sa_i)a_s = a_s(a_sa_i)
\end{align*}
for $i \leq s-1$ and similarly
\begin{align*}
  (a_ja_{s-1})a_s =  a_s(a_ja_{s-1}), \  (a_ja_{s-1})a_{s-1} = a_{s-1}(a_ja_{s-1})
\end{align*}
for $j \geq s$. Analogous equalities hold in $M_{\rho}$.
Therefore, each element of the set $\ol{C} = \la \ol{a_s}\ol{a_i},
\ol{a_j}\ol{a_{s-1}} \colon i < s-1; \; s<j \ra$ commutes with all
elements of the set $\la \ol{a_{s-1}}, \ol{a_s}\ra$.

In the localization $M_{\rho} \la (\ol{a_s}\ol{a_{s-1}})^{-1} \ra$
the following  equalities hold:
\begin{align*}
\ol{a_i} &=  (\ol{a_s}\ol{a_{s-1}})^{-1} (\ol{a_s}\ol{a_i})\ol{a_{s-1}}  \ \ \textup{for} \ i < s-1,    \\
\ol{a_j} &=  (\ol{a_s}\ol{a_{s-1}})^{-1}
(\ol{a_j}\ol{a_{s-1}})\ol{a_s}  \ \ \textup{for} \ j > s.
\end{align*}
Hence $M_{\rho} \subseteq \la \ol{C}, \ol{a_{s-1}}, \ol{a_s},
(\ol{a_s}\ol{a_{s-1}})^{-1} \ra$, thus also $M_{\rho} \la
(\ol{a_s}\ol{a_{s-1}})^{-1} \ra \subseteq \ol{C}~\cdot~\la
\ol{a_{s-1}}, \ol{a_s}, (\ol{a_s}\ol{a_{s-1}})^{-1} \ra$. The
opposite inclusion holds by the definition of $C$, thus we obtain
\begin{equation}
\label{(1)} M_{\rho} \la (\ol{a_s}\ol{a_{s-1}})^{-1} \ra = \ol{C}
\cdot \la \ol{a_{s-1}}, \ol{a_s}, (\ol{a_s}\ol{a_{s-1}})^{-1} \ra.
\end{equation}
Let $\beta \colon \ol{M_{n-2}^{s-1,s}} \rightarrow \ol{M}$ be the
map defined  by
\[\begin{cases}
\beta(\ol{a_i}) =  \ol{a_sa_i} \ \ \textup{for} \ i < s-1    \\
\beta(\ol{a_j}) =  \ol{a_ja_{s-1}}  \ \ \textup{for} \ j > s.
\end{cases}\]
We shall check that $\beta$ is a homomorphism. It suffices to
check that all relations of $\ol{M_{n-2}}$ hold also for the
images of elements in $\ol{M}$. The relations in $\ol{M_{n-2}}$
are the relations of the Chinese monoid $M_{n-2}$ and the
relations introduced by $\rho$. To simplify notation, instead of
$\ol{a_x}$ we shall write only $x$. Then in $\ol{M_{n-2}}$ we
have:

1) relations from $\rho$: commutativity of elements $1,\ldots,
(s-2)$,

2) relations from $\rho$: commutativity of elements $(s+1),
\ldots, n$,

3) relations from the Chinese monoid: $zyx=zxy=yzx$ for $x \leq y
\leq z$.

Notice that if $z <s-1$ or $x >s$, then the relations listed in
(3) follow from the relations from (1) and (2). Therefore instead
of (3) we can consider only:

3.1) if $x \leq y<s<z$, then $zyx=yzx$,

3.2) if $x<s<y \leq z$, then $zyx=zxy$.

In $\ol{M}$ we have relations of the Chinese monoid $M$ and the
relations introduced by $\rho$. We shall now check that the images
of elements in $\ol{M}$ satisfy relations stated in (1), (2),
(3.1) and (3.2).

1) by the definition of $\beta$ for $i, k < s-1$ we have
$\beta(i)= si$, $\beta(k)=sk$, so using the relations in $\ol{M}$
we obtain $\beta(i)\beta(k) = sisk = s(isk) = s(ksi) = sksi =
\beta(k) \beta(i)$; therefore the images of elements $1,\ldots,
(s-2)$ commute.

2) by an analogous argument.

3.1) for $x \leq y \leq s-1 <s<z$, using the Chinese relations in
$\ol{M}$, we obtain
\begin{multline*}
  \beta(z) \beta(y) \beta(x) = z(s-1) sy sx = z (sx) (s-1) sy = (zx) s (s-1) (sy) =  \\
= (sy) zsx(s-1) =  sy z(s-1) sx = \beta(y) \beta(z) \beta(x).
\end{multline*}

3.2) by an analogous argument.

Thus, we have verified that $\beta$ is indeed a homomorphism.

From the definitions of $\beta$ and $C$ we obtain that $\beta
\colon \ol{M_{n-2}^{s-1,s}} \na \ol{C}$. Therefore $\ol{C}$ is the
homomorphic image of the monoid $\ol{M_{n-2}^{s-1,s}}$.

We may now define the natural homomorphism
\[\beta' \colon \ol{M_{n-2}^{s-1,s}} \times B \times \Z \rightarrow \ol{M} \times B \times \Z\]
as $\beta$ on the first coordinate $\ol{M_{n-2}^{s-1,s}}$ and
identity on $B \times \Z$. Therefore, $\beta'(\ol{M_{n-2}^{s-1,s}}
\times B \times \Z) = \ol{C} \times B \times \Z$. Earlier we have
shown that $\la \ol{a_{s-1}}, \ol{a_s},
(\ol{a_s}\ol{a_{s-1}})^{-1} \ra \simeq B \times \Z$. Therefore,
\begin{equation}
\label{(2)} \ol{C} \times  B \times \Z \simeq \ol{C} \times \la
\ol{a_{s-1}}, \ol{a_s}, (\ol{a_s}\ol{a_{s-1}})^{-1} \ra.
\end{equation}
The composition of the epimorphisms $\lambda'$ and $\beta'$ gives
a natural epimorphism
\begin{equation}
\label{(3)} \beta'\lambda' \colon M_{\rho} \la
(\ol{a_s}\ol{a_{s-1}})^{-1} \ra \na \ol{C} \times B \times \Z.
\end{equation}
We also have a natural epimorphism
\begin{equation}
\label{(4)} \ol{C} \times \la  \ol{a_{s-1}}, \ol{a_s},
(\ol{a_s}\ol{a_{s-1}})^{-1} \ra \na \ol{C} \cdot \la \ol{a_{s-1}},
\ol{a_s}, (\ol{a_s}\ol{a_{s-1}})^{-1} \ra.
\end{equation}
Using (\ref{(1)})-(\ref{(4)}) we obtain the commutative diagram
\[ \xymatrix{
M_{\rho} \la (\ol{a_s}\ol{a_{s-1}})^{-1} \ra
\ar@{}[r]|(0.4){\mbox{\normalsize{$=$}}}
\ar@{->>}[d]_{\beta'\lambda'} &
\ol{C} \cdot \la  \ol{a_{s-1}}, \ol{a_s}, (\ol{a_s}\ol{a_{s-1}})^{-1} \ra \\
\ol{C} \times B \times \Z
\ar@{}[r]|(0.4){\mbox{\normalsize{$\simeq$}}} & \ol{C} \times \la
\ol{a_{s-1}}, \ol{a_s}, (\ol{a_s}\ol{a_{s-1}})^{-1} \ra
\ar@{->>}[u] } \] Therefore both maps in (\ref{(3)}) and
(\ref{(4)}) must be isomorphisms. Thus in particular
\[M_{\rho} \la (\ol{a_s} \ol{a_{s-1}})^{-1} \ra \simeq \ol{C} \times B \times \Z.\]
Denote this isomorphism by $\alpha$, so $\alpha =\beta'\lambda'$.
Then we have the commutative diagram
\[ \xymatrix{
M \ar@{->>}[d]_{\psi_\diam} \ar[r]^(0.3){\what{\psi}_{\diam}} & \ol{M_{n-2}^{s-1,s}}\times B \times \Z \ar@{->>}[rd]^{\beta'} \\
M_{\rho} \ar[ur]^{\lambda}
\ar@{}[r]|(0.32){\mbox{\normalsize{$\ssq$}}} & M_{\rho} \la
(\ol{a_s} \ol{a_{s-1}})^{-1}\ra  \ar@{->>}[u]_{\lambda'}
\ar@{}[r]|(0.55){\mbox{\normalsize{$\simeq$}}}_(0.55)\alpha &
\ol{C} \times B \times \Z } \] This means that $\beta'$ and
$\lambda'$ are isomorphism, which in particular leads to the
conclusion that $M_{\rho} \la (\ol{a_s} \ol{a_{s-1}})^{-1} \ra
\simeq \ol{M_{n-2}^{s-1,s}} \times B \times \Z$. This completes
the proof.
\end{proof}

Notice that  $\ol{M_{n-2}^{s-1,s}} \ssq M_{\rho}$, but the factor
$\ol{M_{n-2}^{s-1,s}}$ in the image of $\lambda'$, i.e. in
$\ol{M_{n-2}^{s-1,s}} \times B \times \Z$, is not the same object.

\section{Minimal prime ideals in $K[M]$}
\label{r-min-id-pier}

In this section, a bijection between the set of minimal prime
ideals of $K[M]$ and the set of leaves of a certain tree $D$ is
established. More precisely, the elements $d\in D$ are defined as
diagrams of some special type in Definition~\ref{drzewo}. These
diagrams correspond in a constructive way to certain homogeneous
congruences $\rho (d)$ on $M$ (Construction~\ref{interpr}) and
therefore to the ideals $\I_{\rho (d)}$ of the algebra $K[M]$. The
ideals of $\h$ type and of $\diam$ type will correspond to the
first level of the tree. In particular, it will follow that every
minimal prime ideal $P$ of $K[M]$ is of the form $P=\I_{\rho_P}$,
where $\rho_{P}$ is the congruence on $M$ defined by $\rho_{P}=\{
(s,t) \in M\times M \colon s-t\in P \}$. Therefore $K[M]/P\simeq
K[M/\rho_{P}]$. The construction implies also that $M/\rho_{P}$
embeds into the monoid $\N^{c_P} \times (B \times \Z)^{d_P}$,
where $c_P+2d_P=n$. In Part \ref{s-ddD} the tree $D$ is
introduced. In Parts \ref{s-interpr} and \ref{s-minim} some
intermediate steps are proved. In particular,
Theorem~\ref{liscieidpier} shows that $\I_{\rho (d)}$ is a prime
ideal if $d$ is a leaf of $D$ and Theorem~\ref{idpierzawlisc}
shows that every prime ideal $P$ of $K[M]$ satisfies $\I_{\rho
(d)}\subseteq P$ for a leaf $d$ of $D$. The main result is then
derived in Theorem~\ref{bijekcja}.

Recall that if the rank $n$ of the monoid $M$ is equal to $1$ or
$2$ then the algebra $K[M]$ is prime and semiprimitive,
\cite{jofc}. Hence, as before, we shall assume that $n\geq 3$.

\subsection{Diagrams and the tree $D$}
\label{s-ddD}

\

\begin{ozn}
\label{diagr1} We start with defining certain auxiliary
\emph{diagrams}, built on the set of $n$ generators $a_1, a_2,
\ldots, a_n$ of $M$. Let $\podpis{\circ}{i}$ denote the $i$-th
generator. The simplest diagram is of the form
\[ \diag{ \mathop{\circ}\limits_{1}^{\,^{\,^{\,}}} & \mathop{\circ}\limits_{2}^{\,^{\,^{\,}}}
& \ldots & \mathop{\circ}\limits_{n\vphantom{1}}^{\,^{\,^{\,}}}  }
\] If unambiguous, we omit the indices,
denoting the above diagram also by
\[ \diag{ \circ & \ldots & \circ } \]
The next simple diagrams are of the form
\[ \diag{ \circ & \ldots & \circ & \bullet & \circ & \ldots & \circ
} \] with a distinguished generator $a_s$. A diagram of this type
will be called a \emph{dot $a_s$} or simply a \emph{dot}. We
consider such diagrams only for $s=2,3,\ldots,n-1$. If the number
of generators is $k<n$, such a diagram is called \emph{dot$_k$}.

A diagram of the form
\[ \diag{ \circ &  \ldots & \circ & \bul \ar@/^0.4pc/@{-}[r] & \bul & \circ & \ldots & \circ
} \] with an arc joining generators $a_{s-1}$ and $a_{s}$ is
called an \emph{arc $a_sa_{s-1}$} or simply an \emph{arc}. Here
$s$ can be any of the numbers $2, \ldots, n$. If the number of
generators is $k<n$, such a diagram is called an \emph{arc$_k$}.
\end{ozn}

Next we construct more complicated diagrams. It turns out that all
considered diagrams can be organized in a tree $D$, which
indicates the order and the way these diagrams are constructed.

\begin{df}
\label{drzewo} We construct a finite tree $D$ whose vertices are
diagrams. The construction is performed in several steps. We start
with defining the root of $D$, then in the first step we connect
it by edges with certain new diagrams, which treated as vertices
of $D$ form the \emph{first level of $D$}. In the next steps we
build the subsequent levels of $D$.
\begin{itemize}
  \item{We start with the vertex corresponding to the first of the diagrams described in
  \ref{diagr1}; this vertex is called the \emph{root} of $D$,}
  \item{in the first step we connect the root with $2n-3$ vertices:
  $n-2$ diagrams which are dots and $n-1$ diagrams which are arcs (in the sense of
  \ref{diagr1}); for example, if $n=4$, then we get the first level of $D$:
\newbox{\pu}
\newbox{\pud}
\newbox{\pude}
\newbox{\pudel}
\newbox{\pudelk}
\newbox{\pudelko}
 \savebox{\pu}{\begin{varwidth}[t][0.7cm][c]{\textwidth} $\diag{ \circ & \circ & \circ & \circ}$\end{varwidth}}
 \savebox{\pud}{\begin{varwidth}[t][0.7cm][c]{\textwidth}  $\diag{ \circ & \bul & \circ & \circ}$\end{varwidth}}
 \savebox{\pude}{\begin{varwidth}[t][0.7cm][c]{\textwidth}  $\diag{ \circ & \circ & \bul & \circ}$\end{varwidth}}
 \savebox{\pudel}{\begin{varwidth}[t][0.7cm][c]{\textwidth}  $\diag{ \bul \ar@/^0.4pc/@{-}[r] & \bul & \circ & \circ}$\end{varwidth}}
 \savebox{\pudelk}{\begin{varwidth}[t][0.7cm][c]{\textwidth}  $\diag{ \circ & \bul \ar@/^0.4pc/@{-}[r]& \bul & \circ}$\end{varwidth}}
 \savebox{\pudelko}{\begin{varwidth}[t][0.7cm][c]{\textwidth} $\diag{ \circ & \circ & \bul \ar@/^0.4pc/@{-}[r] & \bul}$\end{varwidth}}
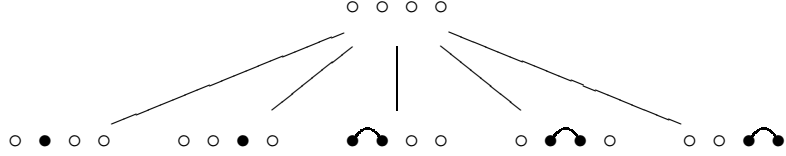
\[ \xymatrix{
& & \usebox{\pu} \ar@{-}[lld] \ar@{-}[ld]  \ar@{-}[d] \ar@{-}[rd] \ar@{-}[rrd] \\
\usebox{\pud}  &
\usebox{\pude} &
\usebox{\pudel} &
\usebox{\pudelk} &
\usebox{\pudelko}
} \] }
\item generators involved in the construction of the appropriate dots or arcs are marked in black
and are called the \emph{used} generators, while the other
generators are called \emph{unused},
  \item{in the next steps we construct the subsequent \emph{levels} of
  $D$, in each step adding, as vertices of $D$, more complicated diagrams constructed according
  to the following rules $\bigoplus$ and $\bigodot$.}
\end{itemize}

\noindent Rules $\bigoplus$:

\begin{itemize}
  \item{in every diagram each generator can be used at most once,}
  \item{if a diagram has $k$ unused generators, we connect to it, as vertices of $D$, all
  diagrams obtained by adding a dot$_k$ or an arc$_k$
   (for these $k$ generators), in a way allowed by the remaining rules,}
  \item{if in a diagram there is an arc using one of the extreme generators $a_1$ or $a_n$, then
  we do not connect any new vertices of $D$ to this vertex and we call such an arc an \emph{extreme arc},
and the corresponding vertex -- a \emph{leaf} of $D$, }
  \item an arc$_k$ (for some $k$ unused generators) can be added only \emph{above}, which means that
  this arc connects the two generators that are neighbors of some used generators.
  (As we shall see in Remark~\ref{oD}, in every step of the construction used generators
  have indices ranging from $j$ to $j+i$ for some $j>0$, $i \geq 0$, so that the two neighboring
  generators are well defined). We get a diagram of the form

  \[ \diag{
\circ & \ldots & \circ & \bul \ar@/^1.3pc/@{-}[rr] & \
{\substack{n-k \text{ used}\\ \text{generators}}} \ & \bul & \circ
& \ldots & \circ } \] We denote $i$ subsequent used generators by
$<i>$, so that the above diagram is simply written as
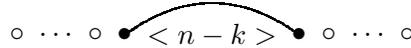
\[ \diag{
\circ & \ldots & \circ & \bul \ar@/^1pc/@{-}[rr] & <n-k> & \bul &
\circ & \ldots & \circ } \]
  \item if in a given step we do not add to some diagram an arc above, and this diagram is not
  a leaf of $D$ then we have to add a dot obeying rules $\bigodot$.
\end{itemize}

\noindent Rules $\bigodot$:

\begin{itemize}
\item{after an arc$_k$ a dot$_{k-2}$ can only follow next to this arc,
  in other words, after the diagram whose last step of construction was an arc$_k$
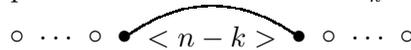
\[ \diag{ \circ & \ldots & \circ & \bul \ar@/^1pc/@{-}[rr] & <n-k> & \bul & \circ & \ldots &
\circ } \] we can either have the diagram

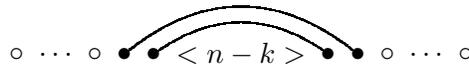
\[ \diag{ \circ & \ldots & \circ & \bul \ar@/^1.5pc/@{-}[rrrr] & \bul \ar@/^1pc/@{-}[rr] &
<n-k> & \bul & \bul & \circ & \ldots & \circ } \] or one of the
following two diagrams
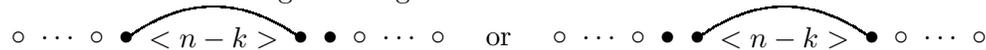
\[ \diag{ \circ & \ldots & \circ & \bul \ar@/^1pc/@{-}[rr] & <n-k> & \bul & \bullet & \circ &
\ldots & \circ } \ \ \ \ \ \text{or} \ \ \ \ \
 \diag{ \circ & \ldots & \circ & \bul & \bul \ar@/^1pc/@{-}[rr] & <n-k> & \bullet & \circ &
 \ldots & \circ } \]
}
\item{after a dot$_k$, for $k<n$, the next dot$_{k-1}$ can occur only as a neighbor
of the former dot; in other words, after a diagram
\[ \diag{ \circ & \ldots & \circ & <n-k> & \bullet & \circ & \ldots & \circ} \]
whose last step of construction was the indicated dot$_k$, either
a diagram of the following form can follow
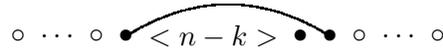
\[ \diag{ \circ & \ldots & \circ & \bul
\ar@/^1pc/@{-}[rrr] & <n-k> & \bullet & \bul & \circ & \ldots &
\circ} \] or the following diagram can follow
\[ \diag{ \circ & \ldots & \circ & <n-k> & \bullet & \bullet & \circ & \ldots & \circ} \]
}
\item immediately after a dot in the first level of $D$ only an arc$_{n-1}$ above can be
added, so after a diagram
\[ \diag{ \circ & \ldots & \circ & \bullet & \circ & \ldots & \circ } \]
 the following diagram can only follow
\[ \diag{ \circ & \ldots & \circ & \bul \ar@/^0.4pc/@{-}[rr] & \bullet & \bul & \circ
& \ldots & \circ } \]
\end{itemize}
\end{df}

\begin{prz}
The following diagrams are vertices of some trees $D$ (for $n=15$
and $9$, respectively)
\[\diag{ \circ & \circ & \circ & \circ &
\circ & \circ & \circ & \bullet \ar@/^0.8pc/@{-}[rrrrrr] & \bullet
\ar@/^0.4pc/@{-}[rr] & \bullet & \bullet & \bullet & \bullet &
\bullet & \circ } \ \ \ \ \ \ \ \ \ \ \ \ \ \ \diag{ \circ & \circ
& \circ & \bul & \bul & \bullet \ar@/^0.4pc/@{-}[r] &  \bullet &
\circ & \circ }
\]
\end{prz}

\begin{uw}
\label{oD} The tree $D$ is finite. In every step of the above
construction the used generators have indices $j, \ldots, j+i$ for
some $i \geq 0$, $j> 0$. The order in which all dots and arcs were
added can be uniquely determined from the form of a given diagram.
The generators $a_1$ or $a_n$ can only be used as elements of an
arc, and such an arc is an extreme arc. A leaf of $D$ is a vertex
in which an extreme arc has appeared.
\end{uw}

\begin{df}
\label{defl} \emph{A branch} in $D$ is a chain of connected
vertices, leading from the root to some vertex $d$. If $d$ is a
leaf then such a branch is called \emph{maximal}.

If a vertex $d_2$ was connected to a vertex $d_1$ in the process
of construction of $D$, then $d_2$ is called a \emph{descendant}
of the vertex $d_1$.
\end{df}

\begin{przy}
\label{przykD} The following diagrams are leaves of $D$

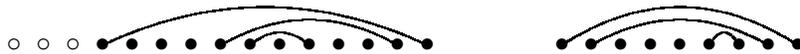
\[\diag{
\circ & \circ & \circ & \bullet \ar@/^1.2pc/@{-}[rrrrrrrrrrr] &
\bullet & \bullet & \bullet &
 \bullet \ar@/^0.8pc/@{-}[rrrrrr] & \bullet \ar@/^0.4pc/@{-}[rr] & \bullet & \bullet &
 \bullet & \bullet & \bullet & \bullet } \ \ \ \ \ \ \ \ \ \ \ \ \ \
\diag{ \bullet \ar@/^1.2pc/@{-}[rrrrrrrr] & \bullet
\ar@/^0.8pc/@{-}[rrrrrr] & \bul & \bul & \bul & \bullet
\ar@/^0.4pc/@{-}[r] &  \bullet & \bullet & \bul } \] For $n=3$ the
tree $D$ has the form
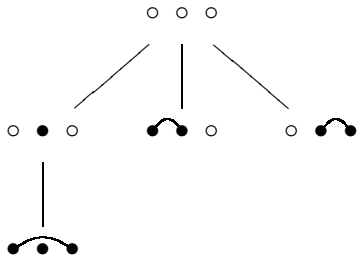
\newbox{\sk}
\newbox{\skr}
\newbox{\skrz}
\newbox{\skrzy}
\savebox{\pu}{\begin{varwidth}[t][0.5cm][c]{\textwidth} $\diag{
\circ & \circ & \circ }$\end{varwidth}}
\savebox{\pud}{\begin{varwidth}[t][0.5cm][c]{\textwidth}  $\diag{
\circ & \bul & \circ }$\end{varwidth}}
\savebox{\pude}{\begin{varwidth}[t][0.5cm][c]{\textwidth}  $\diag{
\bul \ar@/^0.4pc/@{-}[r] & \bul & \circ }$\end{varwidth}}
\savebox{\pudel}{\begin{varwidth}[t][0.5cm][c]{\textwidth}
$\diag{ \circ & \bul \ar@/^0.4pc/@{-}[r]& \bul }$\end{varwidth}}
\savebox{\pudelk}{\begin{varwidth}[t][0.5cm][c]{\textwidth}
$\diag{ \bul \ar@/^0.4pc/@{-}[rr] & \bul & \bul }$\end{varwidth}}
\[ \xymatrix{
 & \usebox{\pu}  \ar@{-}[ld]  \ar@{-}[d] \ar@{-}[rd] \\
\usebox{\pud} \ar@{-}[d] &
\usebox{\pude} &
\usebox{\pudel} \\
\usebox{\pudelk}  } \]
while for $n=4$ the tree $D$ has the form
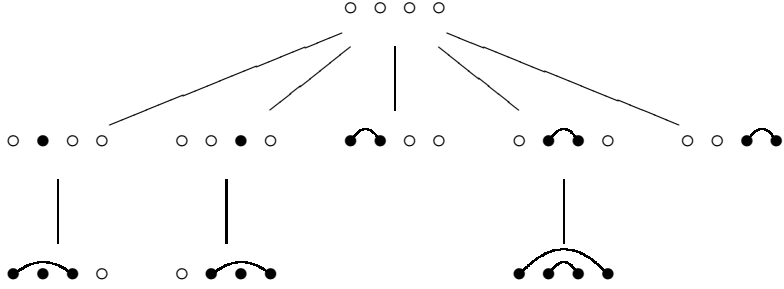
\newbox{\pudl}
\newbox{\pudlo}
\newbox{\skrzynia}
 \savebox{\pu}{\begin{varwidth}[t][0.7cm][c]{\textwidth} $\diag{ \circ & \circ & \circ & \circ}$\end{varwidth}}
 \savebox{\pud}{\begin{varwidth}[t][0.7cm][c]{\textwidth}  $\diag{ \circ & \bul & \circ & \circ}$\end{varwidth}}
 \savebox{\pude}{\begin{varwidth}[t][0.7cm][c]{\textwidth}  $\diag{ \circ & \circ & \bul & \circ}$\end{varwidth}}
 \savebox{\pudel}{\begin{varwidth}[t][0.7cm][c]{\textwidth}  $\diag{ \bul \ar@/^0.4pc/@{-}[r] & \bul & \circ & \circ}$\end{varwidth}}
 \savebox{\pudelk}{\begin{varwidth}[t][0.7cm][c]{\textwidth}  $\diag{ \circ & \bul \ar@/^0.4pc/@{-}[r]& \bul & \circ}$\end{varwidth}}
 \savebox{\pudelko}{\begin{varwidth}[t][0.7cm][c]{\textwidth} $\diag{ \circ & \circ & \bul \ar@/^0.4pc/@{-}[r] & \bul}$\end{varwidth}}
 \savebox{\pudl}{\begin{varwidth}[t][0.7cm][c]{\textwidth} $\diag{ \bul \ar@/^0.4pc/@{-}[rr] & \bul & \bul & \circ}$\end{varwidth}}
 \savebox{\pudlo}{\begin{varwidth}[t][0.7cm][c]{\textwidth} $\diag{ \circ & \bul \ar@/^0.4pc/@{-}[rr] & \bul & \bul } $ \end{varwidth}}
 \savebox{\skrzynia}{\begin{varwidth}[t][0.7cm][c]{\textwidth} $\diag{ \bul \ar@/^0.8pc/@{-}[rrr] & \bul \ar@/^0.4pc/@{-}[r] & \bul & \bul } $\end{varwidth}}
\[ \xymatrix{
& & \usebox{\pu} \ar@{-}[lld] \ar@{-}[ld]  \ar@{-}[d] \ar@{-}[rd] \ar@{-}[rrd] \\
\usebox{\pud} \ar@{-}[d] &
\usebox{\pude} \ar@{-}[d] &
\usebox{\pudel} &
\usebox{\pudelk} \ar@{-}[d] &
\usebox{\pudelko} \\
\usebox{\pudl} &
\usebox{\pudlo} & &
\usebox{\skrzynia}
} \]
\end{przy}

\subsection{Diagrams as congruences on $M$}
\label{s-interpr}

\

In this part we associate to every diagram $d\in D$ a
congruence $\rho (d)$ on $M$ and we show that if $d$ is a leaf of
$D$ then $\I_{\rho (d)}$ is a minimal prime ideal of $K[M]$.

\begin{ozn}
If $u<v$, by $M_{i_j}^{u,v}$ we denote the Chinese monoid with
$i_j$ generators $a_1,\ldots, a_{u-1}, a_{v+1}, \ldots, a_n$; so
that $i_j=n-v+u-1$. Sometimes we denote this monoid simply by
$M_{i_j}$, if it is clear from the context or inessential which of
the generators $a_1, \ldots, a_n$ are skipped. This generalizes
the notation used earlier: $M_{n-1}^s$ and $M_{n-2}^{s-1,s}$.
Indices $i_j$ will be helpful because we shall build sequences of
congruences $\rho_j$ for $j=0,1,2,\ldots$ and monoids $M_{i_j}$
corresponding to these congruences.

Recall that $\rho_0$ denotes  the trivial congruence on $M$. For a
congruence $\rho$ on $M$, by $M_{i_j}/\rho$ we mean
$M_{i_j}/(\rho|_{M_{i_j}})$.

For a given congruence $\rho_j$, let $\psi_j \colon M \rightarrow
M/\rho_j$ be the natural epimorphism. For every $x \in M$ we write
$\psi_j(x) = \wh{j}{x}$. In particular, for $x \in M_{i_j}$ by
$\wh{j}{x}$ we mean the image of $x$ in $M_{i_j}/\rho_j =
M_{i_j}/({\rho_j|_{M_{i_j}}})$. With this notation, $M/\rho_0=M$,
$\psi_0 = id$, $\wh{0}{x} = x$. If $\rho_1$ is a congruence of
type $\h$ or $\diam$ on $M$, then $\wh{1}{x} = \ol{x} = \psi(x)$,
where $\psi = \psi_1 \colon M \rar M/{\rho_1}$ is the natural
homomorphism.
\end{ozn}

\begin{df}
\label{St} We define inductively the following sequences of pairs
$(S_t, i_t)$ for $t \geq 1$, $i_t>0$. Let $i_0=n$ and
\[ \begin{cases}
S_1 = \N \\
i_1 = n-1
\end{cases}
\ \ \ \ \text{or} \ \ \ \
\begin{cases}
S_1 = B \times \Z \\
i_1 = n-2
\end{cases} \]
and for every $t>1$ let
\[ \begin{cases}
S_t = S_{t-1} \times \N \\
i_t = i_{t-1}-1
\end{cases}
\ \ \ \ \text{or} \ \ \ \
\begin{cases}
S_t = S_{t-1} \times B \times \Z \\
i_t = i_{t-1}-2.
\end{cases} \]
Every such sequence $(S_t, i_t)$ is clearly finite. In each of the
pairs, $S_t$ is a direct product of $n-i_t$ factors.

For example, $(\N, n-1)$, $(\N \times B \times \Z, n-3)$, $(\N
\times B \times \Z \times \N, n-4)$, $\ldots$ are initial elements
of a sequence of pairs.
\end{df}

\begin{konstr}
\label{interpr} With each of the diagrams $d\in D$ defined in
Part~\ref{s-ddD} we associate in a natural way a congruence $\rho
(d)$ on $M$ such that if a vertex $d'$ is a descendant of a vertex
$d$ in the tree $D$, then $\rho (d)\subseteq \rho (d')$.
\end{konstr}

\begin{proof}
We proceed by induction. We adopt an induction hypothesis
consisting of five parts and we immediately verify the validity of
the first inductive step.

\textbf{Part} \textbf{(I)}. Consider the diagrams described in
\ref{diagr1}. With the diagram $\ \diag{\circ & \ldots & \circ}\ $
we associate the trivial congruence $\rho_0$. With each
diagram $d$ from the first level of the tree $D$ (so a dot or an
arc) we associate a congruence $\rho_1$ of type $\h$ and $\diam$,
respectively, with an appropriate value of the distinguished index
$s$. We define $\rho (d)=\rho_{1}$. Clearly, all such $\rho_1$
satisfy $\rho_0 \ssq \rho_1$.

Hence, assume inductively that we already know congruences
corresponding to all diagrams up to the $j$-th level of the tree
$D$. For every diagram $d'$ from level $j+1$ we wish to
define a congruence $\rho (d')$. This diagram was constructed in
step $j+1$ of the construction of the tree $D$ from a diagram $d$
in level $j$, with which a congruence $\rho_j=\rho (d)$ is
associated, by adding an arc above or adding a dot on one of the
sides. So we assume that \textbf{(I)}: a diagram $d$ from level
$j$ is not a leaf of $D$ and that we already constructed a chain
of congruences $\rho_0 \ssq \rho_1 \ssq \rho_2 \ssq \ldots \ssq
\rho_j$ on $M$, corresponding to the branch of $D$ which leads to
the diagram $d'$ from level $j+1$.

\textbf{Part} \textbf{(II)}. Let $a_u, a_{u+1}, \ldots, a_v$ be
the generators used in our diagram $d$ from level $j$ (we know
that $u \neq 1$, $v \neq n$ and in view of \ref{oD} the used
generators form a connected segment). Consider
$M_{i_j}^{u,v}/\rho_j$.

For $j=0$ this is $M_n/\rho_0=M_n$. For $j=1$ we know from
\ref{Mkreski} that in case $\h$
\[M_{n-1}^s / \rho_1 = \ol{M_{n-1}^s} =
\raisebox{1ex}{$M_{n-1}^s$} \Big/
\raisebox{-1ex}{$(\substack{a_1,\ldots,a_{s-1}
\text{ commute} \\
a_{s+1},\ldots,a_n \text{ commute}})$},\] while in case $\diam$
\[M_{n-2}^{s-1,s} / \rho_1 =  \ol{M_{n-2}^{s-1,s}} = \raisebox{1ex}{$M_{n-2}^{s-1,s}$}
\Big/ \raisebox{-1ex}{$(\substack{a_1,\ldots,a_{s-2} \text{
commute} \\ a_{s+1},\ldots,a_n \text{ commute}})$}.\]

Hence, assume inductively that \textbf{(II)}: the congruence
$\rho_j$ is chosen in such a way that $M_{i_j}^{u,v}/\rho_j$ is a
Chinese monoid of rank $i_j$ with generators $\wh{j}{a}_1, \ldots,
\wh{j}{a}_{u-1}, \wh{j}{a}_{v+1}, \ldots, \wh{j}{a}_n$, and with
additional relations making the monoid $\langle \wh{j}{a}_1,
\ldots, \wh{j}{a}_{u-1}\rangle $ free commutative and making
$\langle \wh{j}{a}_{v+1}, \ldots, \wh{j}{a}_n\rangle $ free
commutative.

\textbf{Part} \textbf{(III)}. Assume that for every diagram from
any level $t \leq j$, which has some number $i_t$ of unused
generators, there is an associated pair $(S_t, i_t)$, in
accordance with Definition~\ref{St}. For $t=0$, so for the root of
$D$, we have $i_0=0$, while $S_0$ is not defined. With diagrams of
the first level of the tree $D$, so for $t=1$, we associate the
pairs $(S_1, i_1)$ as in Definition~\ref{St}.

\textbf{Part} \textbf{(IV)}. By Corollary~\ref{wnKier} and
Lemma~\ref{Mkaro} we have, for $\rho_1$ of type $\h$ and $\diam$,
respectively, an epimorphism $\what{\psi}_{\h}$ or a homomorphism
$\what{\psi}_{\diam}$, such that
\[(\h) \
\ \  M \rightarrow M/{\rho_1} \simeq \ol{M_{n-1}^s} \times \la
\ol{a_s} \ra \simeq \ol{M_{n-1}^s} \times \N \simeq
M_{n-1}^s/\rho_1 \times \N,\]
\[(\diam) \ \ \ M \rightarrow M/{\rho_1} \hookrightarrow M/{\rho_1} \,
\la(\ol{a_sa_{s-1}})^{-1}\ra \simeq \ol{M_{n-2}^{s-1,s}} \times B
\times \Z \simeq M_{n-2}^{s-1,s}/ \rho_1 \times B \times \Z,\]
with the embedding accomplished by the central localization with
respect to $\la \ol{a_s}\ol{a_{s-1}} \ra$.

More precisely, in $M_n$, for $n \geq 2$, we have $n-2$ possible
congruences $\rho_1$ of type $\h$, so also $n-2$ possible
epimorphisms $\what{\psi}_{\h}$, and also we have $n-1$ possible
congruences $\rho_1$ of type $\diam$, so also $n-1$ possible
homomorphisms $\what{\psi}_{\diam}$. These homomorphisms can be
associated with the corresponding branches in $D$, depending on
the value of $s$. For example, if $n=3$, we get the following
first level of the tree $D$:
\savebox{\pu}{\begin{varwidth}[t][0.5cm][c]{\textwidth} $\diag{
\circ & \circ & \circ }$\end{varwidth}}
\savebox{\pud}{\begin{varwidth}[t][0.5cm][c]{\textwidth}  $\diag{
\circ & \bul & \circ }$\end{varwidth}}
\savebox{\pude}{\begin{varwidth}[t][0.5cm][c]{\textwidth} \
\hskip3cm $\diag{ \bul \ar@/^0.4pc/@{-}[r] & \bul & \circ }$
\hskip3cm \ \end{varwidth}}
\savebox{\pudel}{\begin{varwidth}[t][0.5cm][c]{\textwidth}
$\diag{ \circ & \bul \ar@/^0.4pc/@{-}[r]& \bul }$\end{varwidth}}
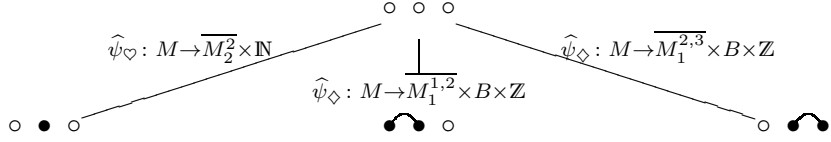
\[ \xymatrix{
 & \usebox{\pu}  \ar@{-}_{\what{\psi}_{\h} \colon M \rightarrow \ol{M_2^2} \times \N \ \ }[ld]
 \ar@{-}|(0.65){\what{\psi}_{\diam} \colon M \rightarrow \ol{M_1^{1,2}}\times B \times \Z}[d]
 \ar@{-}^{\ \ \ \ \what{\psi}_{\diam} \colon M \rightarrow \ol{M_1^{2,3}}\times B \times \Z}[rd] \\
\usebox{\pud} &
\usebox{\pude} &
\usebox{\pudel} } \]

Let $f$ be a diagram from the branch of $D$ leading to the
considered diagram $d$ from level $j+1$ (for which we want to
construct $\rho_{j+1}$). Assume that $f$ is from level $t + 1 \leq
j$ in $D$ and it was created from  a diagram from level $t<j$, in
which there are $0 \leq i_t \leq n$ unused generators, and the
used generators have indices $u_t,\ldots ,v_t$. The value of $i_t$
can be different for different diagrams from level $t$, see
Definition~\ref{St}.

If the diagram $f$ was created by adding a dot then assume
inductively that \textbf{(IV~$\h$)}: there exists an epimorphism
\[\wh{t}{\psi}_{\h} \colon M_{i_t}^{u_t,v_t}/\rho_t \rar
M_{i_t-1}/{\rho_{t+1}} \times \N\] given by
\[\begin{cases}
\wh{t}{\psi}_{\h}(\wh{t}{a}_{s}) = (1, g_s) \\
\wh{t}{\psi}_{\h}(\wh{t}{a}_{l}) = (\wh{t+1}{a}_{l}, 1) \ \ \ \ \
\textup{for $l \neq s$,}
\end{cases}\]
where $\la g_s \ra \simeq \N$ and  $s=u_t-1$ or $s=v_t+1$,
depending on which of the two possible dots was added. We
associate this epimorphism with the edge of the tree $D$ which is
used to add the diagram $f$.

If $f$ was created by adding an arc, then assume inductively that
\textbf{(IV~$\diam$)}: there exists a homomorphism
\[\wh{t}{\psi}_{\diam} \colon M_{i_t}^{u_t,v_t}/\rho_t
\rar M_{i_t-2}^{u_t-1, v_t+1}/{\rho_{t+1}} \times B \times \Z\]
given by
\[\begin{cases}
\wh{t}{\psi}_{\diam}(\wh{t}{a}_{u-1}) = (1, p, g) \\
\wh{t}{\psi}_{\diam}(\wh{t}{a}_{v+1}) = (1, q,1)  \\
\wh{t}{\psi}_{\diam}(\wh{t}{a}_{l}) = (\wh{t+1}{a}_{l}, p, 1) \ \ \ \ \ \textup{for $l < u_t-1$}\\
\wh{t}{\psi}_{\diam}(\wh{t}{a}_{l}) = (\wh{t+1}{a}_{l}, q, 1) \ \
\ \ \ \textup{for $l > v_t+1$},
\end{cases}\]
where $\Z \simeq \la g, g^{-1} \ra$. We associate this
homomorphism with the edge of the tree $D$, which was used to add
the diagram $f$.

Notice that for $t=0$ the corresponding homomorphisms are
$\wh{0}{\psi}_{\h}=\what{\psi}_{\h}$ and $\wh{0}{\psi}_{\diam} =
\what{\psi}_{\diam}$.

\textbf{Part} \textbf{(V)}. Define for $t<j$ and for
$\triangle=\h$ or $\triangle=\diam$ the map
$\wh{t}{\kappa}_{\triangle}$ by $\wh{0}{\kappa}_\triangle =
\what{\psi}_\triangle$ and for  $0<t<j$
\[\wh{t}{\kappa}_{\triangle} \colon M_{i_t}/{\rho_t} \times S_t \rar M_{i_{t+1}}/{\rho_{t+1}}
\times S_{t+1}, \ \ \ \ \ \wh{t}{\kappa}_\triangle =
(\wh{t}{\psi}_\triangle, id).\] By the induction hypothesis (IV)
applied to $\wh{t}{\psi}_\triangle$ we know that for every edge in
$D$ one of the maps $\wh{t}{\kappa}_{\h}$ or
$\wh{t}{\kappa}_{\diam}$ exists and $\wh{t}{\kappa}_{\h}$ is an
epimorphism. Each map $\wh{t}{\kappa}_\triangle =
(\wh{t}{\psi}_\triangle, id)$ we associate with the edge in $D$
with which the corresponding $\wh{t}{\psi}_\triangle$ is
associated.

Consider the branch of $D$ leading from the root to a diagram $f$
from level $t+1$. The subsequent edges of this branch correspond
to some homomorphisms $\wh{0}{\kappa}_*, \wh{1}{\kappa}_*, \ldots,
\wh{t}{\kappa}_*$, where each $*$ denotes $\h$ or $\diam$. For
$\triangle$ equal to $\h$ or $\diam$, we define a homomorphism
\[\wh{t}{\kappa'}_\triangle \colon M \rar M_{i_{t+1}}/{\rho_{t+1}}
\times S_{t+1}\] as the composition
\[\wh{t}{\kappa'}_\triangle = \wh{t}{\kappa}_\triangle \circ \wh{t-1}{\kappa}_* \circ
\ldots \circ \wh{1}{\kappa}_* \circ \wh{0}{\kappa}_*.\] For $t=0$
we have $\wh{0}{\kappa'}_\triangle = \wh{0}{\kappa}_\triangle$,
while for $0<t<j$ we have
\[\wh{t}{\kappa'}_\triangle =
\wh{t}{\kappa}_\triangle \circ (\wh{t-1}{\kappa}_\nabla \circ
\wh{t-2}{\kappa}_*  \circ \ldots \wh{1}{\kappa}_* \circ
\wh{0}{\kappa}_*) = \wh{t}{\kappa}_\triangle \circ
\wh{t-1}{\kappa'}_\nabla,\]
\[\wh{t}{\kappa'}_\triangle \colon M \nad{\wh{t-1}{\kappa'}_\nabla}{\rightarrow}
M_{i_t}^{u,v}/\rho_t \times S_t
\nad{\wh{t}{\kappa}_\triangle}{\rightarrow} M_{i_{t+1}}/\rho_t
\times S_{t+1},\] where each $*$ denotes $\h$ or $\diam$ and
$\nabla = \h$ or $\nabla = \diam$. The map
$\wh{t}{\kappa'}_\triangle$ is an epimorphism if and only if all
$*$ are equal to $\h$ and $\nabla = \h$. However, by the
construction of $D$ we cannot simultaneously have
$\wh{0}{\kappa}_* = \wh{0}{\kappa}_\h$ and $\wh{1}{\kappa}_* =
\wh{1}{\kappa}_\h$. Hence $\wh{t}{\kappa'}_\triangle$ is an
epimorphism only for $t=0$ and $\triangle =\h$.

In cases where the index $\triangle$ is not important we simply
write $\wh{t}{\psi}$ or $\wh{t}{\kappa}$ or $\wh{t}{\kappa'}$,
respectively.

Assume inductively that \textbf{(V)}: for $t+1 \leq j$ we have
$\rho_{t+1}= ker(\wh{t}{\kappa'})$. For $t=0$, since
$\wh{0}{\kappa'} = \wh{0}{\kappa} = \wh{0}{\psi}$, we have
$\rho_1=ker (\psi)$, which agrees with the definition of $\rho_1$.

Next, we define a congruence $\rho_{j+1}$ and we verify that it
satisfies the inductive claim, so conditions (I)-(V) are
satisfied. We may assume that the considered diagram $d'$ from
level $j+1$ was constructed from a diagram $d$ from level $j$ by
adding a dot $a_{u-1}$ (the proof for a dot $a_{v+1}$ is similar)
or an arc $a_{v+1}a_{u-1}$. In the former case, we define a map
\[\wh{j}{\psi}_\h \colon
M_{i_j}^{u,v}/\rho_j \rightarrow M_{i_j-1}/\rho_j \times \N\] as
the natural extension of the homomorphism defined on generators as
follows:
\[\begin{cases}
\wh{j}{\psi}_{\h}(\wh{j}{a}_{u-1})=(1,g_{u-1}) \\
\wh{j}{\psi}_{\h}(\wh{j}{a}_l) = (\wh{j}{a}_l,1) \ \ \text{ for }
l \neq u-1,
\end{cases}\]
where $\la g_{u-1}\ra \simeq \N$, and in the latter case as the
homomorphism
\[\wh{j}{\psi}_\diam \colon M_{i_j}^{u,v}/\rho_j \rightarrow M_{i_{j-2}}/\rho_j \times B \times \Z\]
naturally extending:
\[\begin{cases}
\wh{j}{\psi}_{\diam}(\wh{j}{a}_{u-1})=(1,p,g) \\
\wh{j}{\psi}_{\diam}(\wh{j}{a}_{v+1})=(1,q,1) \\
\wh{j}{\psi}_{\diam}(\wh{j}{a}_l) = (\wh{j}{a}_l,p,1) \ \ \textup{for} \ l < u-1   \\
\wh{j}{\psi}_{\diam}(\wh{j}{a}_l) = (\wh{j}{a}_l,q,1) \ \
\textup{for} \ l > v+1.
\end{cases}\]
Both maps are homomorphisms because they are homomorphisms on each
of the components. Moreover $\wh{j}{\psi}_{\h}$ is an epimorphism
and $\wh{j}{\psi}_{\diam}$ is not an epimorphism. This is verified
in the same way as for $\widehat{\psi}_\h$ and
$\widehat{\psi}_\diam$.

Let
\[ \begin{cases}
S_{j+1} = S_j \times \la g_{u-1} \ra \simeq S_j \times \N \\
i_{j+1}=i_j-1
\end{cases} \]
in case $\h$ and
\[ \begin{cases}
S_{j+1}= S_j \times B \times \Z \\
i_{j+1}=i_j-2
\end{cases} \]
in case $\diam$. Then $i_{j+1}$ so defined coincides with the
number of used generators in the diagram. Moreover, the pairs
$(S_{j+1}, i_j)$ defined in this way satisfy conditions of
Definition~\ref{St}. This completes \textbf{the proof of part
(III)} of the inductive claim.

Define the homomorphism
\[\wh{j}{\kappa}_\triangle \colon
M_{i_j}^{u,v}/\rho_j \times S_j \rightarrow M_{i_{j+1}}/\rho_j
\times S_{j+1} \ \ \ \ \textup{by} \ \ \ \
\wh{j}{\kappa}_\triangle = (\wh{j}{\psi}_\triangle, id),\] so
$\wh{j}{\kappa'}_\triangle$ are defined in the same way as
$\wh{t}{\kappa'}_\triangle$ for $t<j$.

Similarly as for $\wh{t}{\kappa'}$, let $\wh{j}{\kappa'}_\triangle
\colon M \rightarrow M_{i_{j+1}}/\rho_{j+1} \times S_{j+1}$ be the
homomorphism defined by $\wh{j}{\kappa'}_\triangle =
\wh{j}{\kappa}_\triangle \circ \wh{j-1}{\kappa}_*$. The
homomorphism $\wh{j}{\kappa'}$ corresponds to a congruence
${ker(\wh{j}{\kappa'})}$ on $M$. We define
$\rho_{j+1}={ker(\wh{j}{\kappa'})}$. We will show that
$\rho_{j+1}$ satisfies the inductive claim.

By the inductive hypothesis we know that $\rho_t \ssq \rho_{t+1}$
for $0<t<j$. Also, for $0<t \leq j$ we have $\wh{t}{\kappa'} =
\wh{t}{\kappa} \circ \wh{t-1}{\kappa'}$, so that
$ker(\wh{t-1}{\kappa'}) \ssq ker(\wh{t}{\kappa'})$. Thus, by the
inductive hypothesis (V) and by the definition of $\rho_{j+1}$ we
get
\[\rho_j = ker(\wh{j-1}{\kappa'}) \ssq ker(\wh{j}{\kappa'}) = \rho_{j+1},\]
so that $\rho_j \ssq \rho_{j+1}$. This completes \textbf{the proof
of part (I)} of the inductive claim.

Now we will show that $\rho_j|_{M_{i_{j+1}}} \supseteq
\rho_{j+1}|_{M_{i_{j+1}}}$. Assume that for some $x,y \in
M_{i_{j+1}}$ we have $(x,y) \in \rho_{j+1}|_{M_{i_{j+1}}}$.
Similarly as above, by the definition of $\rho_{j+1}$ this means
that $(x,y) \in {ker(\wh{j}{\kappa'})}$, so that
$\wh{j}{\kappa'}(x) = \wh{j}{\kappa'}(y)$. By the definition of
$\wh{j}{\kappa'} \colon M \rar M_{i_{j+1}}/\rho_j \times S_{j+1}$,
the latter implies that the first components (belonging to
$M_{i_{j+1}}/\rho_j$) of elements $\wh{j}{\kappa'}(x)$ and
$\wh{j}{\kappa'}(y)$ are equal, so that the images of $x,y$ in
$M_{i_{j+1}}/\rho_j$ are equal. Hence $(x,y) \in
\rho_j|_{M_{i_{j+1}}}$. So indeed we have $\rho_j|_{M_{i_{j+1}}}
\supseteq \rho_{j+1}|_{M_{i_{j+1}}}$, as desired.

Therefore, in view of the opposite inclusion proved before, we get
$\rho_j|_{M_{i_{j+1}}} = \rho_{j+1}|_{M_{i_{j+1}}}$, whence also
\[M_{i_{j+1}}/\rho_j = M_{i_{j+1}}/\rho_{j+1}.\] Thus
$M_{i_{j+1}}/\rho_{j+1}$ is the Chinese monoid on $i_{j+1}$
generators, with the same additional relations as
$M_{i_{j+1}}/\rho_j$, so with commutativity of the sets of
generators that are on the same side of the generators used
earlier. This completes \textbf{the proof of part (II)} of the
inductive claim.

Since $M_{i_{j+1}}/\rho_j = M_{i_{j+1}}/\rho_{j+1}$, it follows
that
\[\wh{j}{\psi}_\h \colon M_{i_j}^{u,v}/\rho_j \rightarrow M_{i_j-1}/\rho_j \times \N = M_{i_{j+1}}/\rho_j \times \N =
M_{i_{j+1}}/\rho_{j+1} \times \N,\]
\[\wh{j}{\psi}_\diam \colon M_{i_j}^{u,v}/\rho_j \rightarrow M_{i_j-2}/\rho_j \times B \times \Z =
M_{i_{j+1}}/\rho_j \times B \times \Z = M_{i_{j+1}}/\rho_{j+1}
\times B \times \Z.\] Moreover, for every generator $a_l$ of
$M_{i_{j+1}}$ we therefore have $\wh{j}{a}_l = \wh{j+1}{a}_l$. By
the above, and in view of the definition, $\wh{j}{\psi}$ satisfies
all the conditions of the inductive hypothesis for $\wh{t}{\psi}$
with $t<j$. This completes  \textbf{the proof of part (IV)} of the
inductive claim.

Hence, also $\wh{j}{\kappa'}$ satisfies all conditions that hold
by the assumption for $\wh{t}{\kappa'}$ with $t<j$. Moreover, by
the definition, $\rho_{j+1} = ker(\wh{j}{\kappa'})$. This
completes \textbf{the proof of part (V)} of the inductive claim.
Therefore, defining $\rho (d')=\rho_{j+1}$, we accomplish the
aims stated at the beginning of the inductive construction.
\end{proof}

\begin{ozn}
\label{oznphi} From now on we adopt the notation used in
Construction~\ref{interpr}. We know that $\rho_j \ssq \rho_{j+1}$
and $\psi_j$ and $\psi_{j+1}$ are epimorphisms. Hence there exists
a natural epimorphism $\varphi_j$ such that the diagram
\[ \xymatrix{
M \ar@{->>}^{\psi_j}[r]  \ar@{->>}[rd]_{\psi_{j+1}} & M/\rho_j  \ar^{\varphi_j}[d] \\
 & M/\rho_{j+1}
}\] commutes, that is $\psi_{j+1} = \varphi_j \circ \psi_j$.
\end{ozn}

\begin{lem}
\label{diagr-przem} For every $j$ there exists a natural embedding
\[\lambda_{j+1} \colon M/\rho_{j+1} \hookrightarrow M_{i_{j+1}} /
\rho_{j+1} \times S_{j+1}.\] Moreover, the following diagram
commutes
\[ \xymatrix{
M/\rho_j \ar@{^{(}->}^(0.4){\lambda_j}[r] \ar_{\varphi_j}[d] &  M_{i_j} / \rho_j \times S_j \ar^{\wh{j}{\kappa}}[d]\\
M/\rho_{j+1} \ar@{^{(}->}^(0.35){\lambda_{j+1}}[r] &  M_{i_{j+1}}
/ \rho_{j+1} \times S_{j+1} }\] \end{lem}

\begin{proof}
First, consider the case $j=1$. If $\rho_1$ is of type $\h$ then
Lemma~\ref{Mkier} yields an isomorphism $M/\rho_1 \simeq
M_{i_1}/\rho_1 \times S_1$, which we denote by $\lambda_1$. We
know that the diagram
\[ \xymatrix{
M \ar_{\psi}[d] \ar^{\wh{0}{\kappa'}=\what{\psi}}[rd]\\
M/\rho_1
\ar@{}|(0.4){\mbox{\normalsize{$\simeq$}}}_(0.4){\lambda_1}[r]  &
M_{i_1} / \rho_1 \times S_1 }\] commutes.

If $\rho_1$ is of type $\diam$, then by the proof of
Lemma~\ref{Mkaro} we have an embedding $\lambda \colon M/\rho_1
\hrar M_{i_1}/\rho_1 \times S_1$, which we denote by $\lambda_1$.
As in case $\h$ we know that the diagram
\[ \xymatrix{
M \ar_{\psi}[d] \ar^{\wh{0}{\kappa'}=\what{\psi}}[rd]\\
M/\rho_1 \ar@{^{(}->}^(0.4){\lambda_1}[r]  &  M_{i_1} / \rho_1
\times S_1 }\] commutes.

For $j>1$, by the inductive construction of $\wh{j}{\kappa'}$ in
\ref{interpr}, we get $Im(\wh{j}{\kappa'}) \ssq
M_{i_{j+1}}/\rho_{j+1} \times S_{j+1}$. Since
$\rho_{j+1}={ker(\wh{j}{\kappa'})}$, we thus get the desired
natural embedding \[\lambda_{j+1} \colon M/\rho_{j+1}
\hookrightarrow  M_{i_{j+1}}/\rho_{j+1} \times S_{j+1}.\]
Therefore the diagram
\[ \xymatrix{
M \ar_{\psi_{j+1}}[d] \ar^{\wh{j}{\kappa'}}[rd]\\
M/\rho_{j+1} \ar@{^{(}->}^(0.35){\lambda_{j+1}}[r]  &  M_{i_{j+1}}
/ \rho_{j+1} \times S_{j+1} }\] commutes.

Since congruences from higher levels of the tree $D$ satisfy
analogous conditions, for every $j>0$ we get a commuting diagram
\[ \xymatrix{
M \ar_{\psi_j}[d] \ar^{\wh{j-1}{\kappa'}}[rd]\\
M/\rho_j \ar@{^{(}->}^(0.4){\lambda_j}[r]  &  M_{i_j} / \rho_j
\times S_j. }\] Hence $\wh{j-1}{\kappa'} = \lambda_j \circ
\psi_j$. Thus $\wh{j-1}{\kappa'} \circ \psi_j^{-1}(x) =
\lambda_j(x)$ for every $x \in M/\rho_j$.

By the definition of $\wh{j}{\kappa'}$ we have $\wh{j}{\kappa'} =
\wh{j}{\kappa} \circ \wh{j-1}{\kappa'}$. Moreover, by the
definition of $\varphi_j$, we have $\psi_{j+1} = \varphi_j \circ
\psi_j$. Hence, for every $x \in M/\rho_j$ and its preimage
$\psi_j^{-1}(x) \ssq M$, it follows that $\psi_{j+1} (
\psi_j^{-1}(x)) = \varphi_j(x).$ All the above easily leads to
\begin{align*}
\lambda_{j+1} \circ \varphi_j (x)= (\lambda_{j+1} \circ
\psi_{j+1}) ( \psi_j^{-1} (x))=
\wh{j}{\kappa'} ( \psi_j^{-1} (x))= \\
= (\wh{j}{\kappa} \circ \wh{j-1}{\kappa'}) ( \psi_j^{-1} (x))=
\wh{j}{\kappa} \circ \lambda_j (x),
\end{align*}
which establishes the assertion.
\end{proof}

\begin{ozn}
\label{Al} For a fixed diagram $d$ in $D$, let $A_l$ be the
submonoid of $M$ generated by all products $a_ya_x$, corresponding
to arcs built in this diagram up to the $l$-th step (inclusive) of
the construction of $d$. In case $\h$ for $l=1$ (where there are
no arcs), we define $A_1 = \{1\}$.
\end{ozn}

\begin{ozn}
\label{utozs}
  For simplicity, we sometimes identify
  $M/\rho_l$ with $\lambda_l(M/\rho_l)$ and we identify
  $(M/\rho_l) \cdot (\wh{l}{A}_l)^{-1}$ with
  $\lambda_l(M/\rho_l) \cdot (\lambda_l(\wh{l}{A}_l))^{-1}$.
\end{ozn}

\begin{lem}
\label{mozna-lokal} With notation as in \ref{Al}, for every $l>0$
the elements $\wh{l}{a}_y\wh{l}{a}_x$ are central and regular in
$M/\rho_l$. Moreover $(M/\rho_l) \cdot (\wh{l}{A}_l)^{-1} \ssq
M_{i_l}/\rho_l \times S_l$ (identifying $M/\rho_l$ with
$\lambda_l(M/\rho_l)$).
\end{lem}

\begin{proof}
We know that $\wh{l}{A}_l = \psi_l (A_l) \ssq M/\rho_l$. From
Lemma~\ref{diagr-przem} we have an embedding $\lambda_l \colon
M/\rho_l \hrar M_{i_l}/\rho_l \times S_l$. We will consider the
images of elements of $A_l$ in $M_{i_l}/\rho_l \times S_l$ under
the map $\lambda_l \circ \psi_l = \wh{l-1}{\kappa'} \colon M \rar
M_{i_l}/\rho_l \times S_l$.

Assume that in some step $k+1<l$ of the construction, an arc
$a_ya_x$ is built, where $k \geq 0$. We study the images of the
generators up to this step $k+1$.

Consider all the steps of the construction, from step one till
step $k$. By the last part of the proof of part (IV) of the
inductive claim in Construction~\ref{interpr} we know that for
every $t$, in step $t+1$ we have (depending on the case,
respectively)
\[ \begin{cases}
  \wh{t}{\psi}_{\h}(\wh{t}{a}_x) = (\wh{t+1}{a}_x,1) \\
  \wh{t}{\psi}_{\h}(\wh{t}{a}_y)  = (\wh{t+1}{a}_y,1)
\end{cases}
\ \ \ \ \ \text{or} \ \ \ \ \
\begin{cases}
  \wh{t}{\psi}_{\diam}(\wh{t}{a}_x) = (\wh{t+1}{a}_x,p,1) \\
  \wh{t}{\psi}_{\diam}(\wh{t}{a}_y) = (\wh{t+1}{a}_y,q,1).
\end{cases}\]
This follows from the definition of the maps
$\wh{t}{\psi}_\triangle$ and from the fact that in step $k$ an arc
$a_ya_x$ is built, so in the previous steps generators with
indices between $x$ and $y$ are used.

This implies that in step $k$ the images of the generators $a_x,
a_y \in M$ under $\wh{k-1}{\kappa'}$ have the form
\[\begin{cases}
  \wh{k-1}{\kappa'}(a_x) = (\wh{k}{a}_x,[1,p]) \\
  \wh{k-1}{\kappa'}(a_y) = (\wh{k}{a}_y,[1,q]),
\end{cases} \]
where $\wh{k}{a}_x, \wh{k}{a}_y \in M_{i_k}^{x+1,y-1}/\rho_k$, and
$[1,p],[1,q] \in S_k=\N^c \times (B \times \Z)^d$ denote sequences
of length $c+2d$ consisting of $(c+d)$ elements $1$ and $d$
elements $p$ or (respectively) $(c+d)$ elements $1$ and $d$
elements $q$, and $p$ in $[1,p]$ occurs in exactly the same places
as $q$ in $[1,q]$.

In step $(k+1)$ an arc $a_ya_x$ is built, so according to the
definition of $\wh{k}{\psi}_{\diam}$ we get
\[\begin{cases}
\wh{k}{\psi}_{\diam}(\wh{k}{a}_x) = (1,p,g) \\
\wh{k}{\psi}_{\diam}(\wh{k}{a}_y)= (1,q,1),
\end{cases}\]
so that
\[\begin{cases}
  \wh{k}{\kappa'}_{\diam}(a_x) = \wh{k}{\kappa}_{\diam} \circ \wh{k-1}{\kappa'}(a_x) =
  \wh{k}{\kappa}_{\diam}(\wh{k}{a}_x,[1,p]) =
  (\wh{k}{\psi}_{\diam}(\wh{k}{a}_x), [1,p]) = (1,p,g,[1,p]) \\
  \wh{k}{\kappa'}_{\diam}(a_y) = (1,q,1,[1,q]) \ \ \ \ \ \textup{(analogously)}.
\end{cases}\]
In the next steps of the construction, from step $k+2$ till step
$j$, elements $1$ occurring as the first components of the above
images of $a_x$ and $a_y$ yield in the image $(1,1)$ in case $\h$
and $(1,1,1)$ in case $\diam$, respectively. More precisely, since
$\wh{t}{\psi}_{\h}$ and $\wh{t}{\psi}_{\diam}$, are homomorphisms,
for the element $1 \in M_{i_t}/\rho_t$ we get the equalities
\[\begin{cases}
  \wh{t}{\psi}_{\h}(1) = (1,1) \in M_{i_t-1}/\rho_{t+1} \times \N\\
  \wh{t}{\psi}_{\diam}(1) = (1,1,1)\in M_{i_t-2}/\rho_{t+1} \times B \times
  \Z ,
\end{cases}\]
respectively. Hence, for every $z \in S_t$ we get (respectively)
\[\begin{cases}
  \wh{t}{\kappa}_{\h}(1,z) = (1,1,z) \\
  \wh{t}{\kappa}_{\diam}(1,z) = (1,1,1,z) ,
\end{cases}\]
respectively. Therefore, in step $l$ of the construction, with
$l>k+1$ we get
\[\begin{cases}
  \wh{l-1}{\kappa'}(a_x) = (\wh{l}{a}_x,[1,p])=(1,\ldots,1,1,p,g,[1,p]) \\
  \wh{l-1}{\kappa'}(a_y) = (\wh{l}{a}_y,[1,q])=(1,\ldots,1,1,q,1,[1,q]).
\end{cases} \]
Since $p$  occurs in $[1,p]$ in the same components as $q$ occurs
in $[1,q]$, it follows that
\begin{multline*}
 \wh{l-1}{\kappa'}(a_ya_x) = \wh{l-1}{\kappa'}(a_y)\wh{l-1}{\kappa'}(a_x) =
 (\wh{l}{a}_y\wh{l}{a}_x,[1,qp])= \\
 = (1, \ldots,1,1,qp,g,[1,qp]) = (1,\ldots,1,g,1,\ldots,1) \in M_{i_l}/\rho_l \times S_l.
\end{multline*}
Moreover, the above implies that $g$ occurs in the components $\N$
occurring in $S_l$.

Thus in step $l$ for $l>k+1$ the image of the element $a_ya_x$
(corresponding to any previously built arc) in $M_{i_l}/{\rho_l}
\times S_l$ is of the form $(\wh{l}{a}_y\wh{l}{a}_x,
[1,1])=(1,\ldots,1,g, 1, \ldots, 1)$. This is a central element.
It is also invertible in $M_{i_l}/{\rho_l} \times S_l$. So it is
also central and regular in $M/\rho_l \ssq M_{i_l}/{\rho_l} \times
S_l$ (where $M/\rho_l$ is identified with $\lambda_l(M/\rho_l)$).

In particular, we may consider the localization $(M/\rho_l)
\cdot(\wh{l}A_l)^{-1}$ with respect to the submonoid generated by
all such elements. Moreover, with identifications as in
\ref{utozs}, since we have inclusions $M/\rho_l \ssq
M_{i_l}/{\rho_l} \times S_l$ and $(\wh{l}A_l)^{-1} \ssq
M_{i_l}/{\rho_l} \times S_l$, we also get $(M/\rho_l) \cdot
(\wh{l}{A}_l)^{-1} \ssq M_{i_l}/\rho_l \times S_l$. This completes
the proof.
\end{proof}

\begin{df}
By the \emph{middle of a diagram} $d\in D$ we mean

  -- the first generator used in $d$ as a dot, if the construction of $d$ starts with a dot,

  -- the middle of the first arc, if the construction of $d$ starts with an arc.

We say that a generator $a_{j}$ is on the \emph{left} (\emph{right},
respectively) \emph{of the middle of $d$} if it is located so in the
graphical presentation of the diagram $d$.
\end{df}

For further convenience, we state some basic observations
used in the proof of Lemma~\ref{mozna-lokal} and some related facts
 easily obtained in a similar way.

\begin{uw}
\label{lewoprawo} Let $\rho = \rho (d)$ be the congruence on $M$
corresponding to a diagram $d\in D$. Say, $d$ is at level $l$ of
the tree $D$ and $\rho =\rho_{l}$, where $\rho_{1}\subseteq \cdots
\subseteq \rho_{l}$ is a chain of congruences used in the process
of building $\rho$, described in Construction~\ref{interpr}. As in
the proof of Lemma~\ref{mozna-lokal}, $M/\rho_l $ is viewed as a
submonoid of $M_{i_l}/\rho \times S_l,$ where $S_l$ is the product
of some copies of  $B\times \Z$ and some copies of $\N$. We
consider the images of the generators of $M$ in $M/\rho \subseteq
M_{i_l}/\rho \times S_{l}$. Then:

-- for every component $B$ in $S_l$ the image of a generator
$a_{j}$ used in $d$ has one of the elements $p,q$ in this
component,

-- for every component $B$ in $S_l$ there exist generators $a_i$,
$a_j$, whose images in this component are equal to $p$ and $q$,
respectively; hence, if the image in $M/\rho$ of some $w\in M$ has
$p^t$ or $q^t$ for some $t>0$ in this component, then it is not
central in $M/\rho$,

-- if $a_{j}$ is the middle of $d$ then the image of the generator
$a_{j}$ is of the form $(1, \ldots ,1, g_{j})$,

-- if $a_{j}$ is a used generator in $d$ that is on the left from
the middle of $d$ then the image of $a_{j}$ does not have
components equal to $q\in B$ and it has at most one component
equal to $g$ (recall that $\la g, g^{-1} \ra \simeq \Z$; this
happens if $a_{j}$ was used in an arc in $d$) or $g_j$ (if $a_{j}$
was used as a dot; in this case $\la g_j \ra \simeq \N$),

-- if $a_{j}$ is a used generator in $d$ that is on the right of
the middle of $d$ then the image of $a_{j}$ in $M/\rho$ does not
have components equal to $p\in B$ and at most one component $g_j$
can occur (if the generator $a_{j}$ was used as a dot).
\end{uw}

The following is an extension of Lemma~\ref{Mkier} and
Lemma~\ref{Mkaro}.

\begin{stw}
\label{centrlok} With identifications as in \ref{utozs}, we have
\[(M/\rho_l) (\wh{l}{A}_l)^{-1} =  M_{i_l} / \rho_l \times S_l.\]
\end{stw}

\begin{proof}
By Lemma~\ref{mozna-lokal} we know that $(M/\rho_l)
(\wh{l}{A}_l)^{-1} \ssq M_{i_l} / \rho_l \times S_l.$

For $\rho_1$ of type $\h$ we have $A_1=\emptyset$ by the
definition, and $i_1=n-1$, $S_1=\N$, so that the claim takes the
form $M/\rho_1 \simeq M_{n-1} / \rho_1 \times \N$, which holds by
Lemma~\ref{Mkier}. For $\rho_1$ of type $\diam$, the set $A_1$
consists of one element $a_sa_{s-1}$, and $i_1=n-2$, $S_1=B \times
\Z$, so the claim takes the form $(M/\rho_1) \cdot \la
(\ol{a_s}\ol{a_{s-1}})\ra^{-1} = M_{n-2} / \rho_1 \times B \times
\Z$, which follows from Lemma~\ref{Mkaro}.

Assume by induction that the claim holds for all congruences
$\rho$ corresponding to diagrams in level $j\geq 1$ of the tree
$D$. Let $d$ be an arbitrary diagram from level $j+1$. Let
$\rho_{j+1}=\rho (d)$. Let $e\in D$ be the diagram of level $j$
which was used in the construction of $d$ and let $\rho_{j}=\rho
(e)$. Then the inductive hypothesis says that
\[(M/\rho_j) \cdot (\wh{j}{A}_j)^{-1}  =  M_{i_j} / \rho_j \times
S_j\] and the inductive claim takes the form $(M/\rho_{j+1}) \cdot
(\wh{j+1}{A}_{j+1})^{-1} = M_{i_{j+1}} / \rho_{j+1} \times
S_{j+1}.$

Consider the commuting diagram (see Lemma~\ref{diagr-przem})
\[ \xymatrix{
M/\rho_j \ar@{^{(}->}^(0.4){\lambda_j}[r] \ar_{\varphi_j}[d] &
M_{i_j} /
\rho_j \times S_j \ar^{\wh{j}{\kappa}}[d]\\
M/\rho_{j+1} \ar@{^{(}->}^(0.35){\lambda_{j+1}}[r] &  M_{i_{j+1}}
/ \rho_{j+1} \times S_{j+1} }  \tag{\#}\] By
Lemma~\ref{mozna-lokal} the map $\lambda_{j+1}$ yields an
embedding $(M/\rho_{j+1})(\wh{j+1}{A}_{j+1})^{-1} \hrar
M_{i_{j+1}} / \rho_{j+1} \times S_{j+1}$. All elements
$\lambda_j(\wh{j}{A}_j)^{-1}$ are invertible in $M_{i_j} / \rho_j
\times S_j$, by the last part of the proof of
Lemma~\ref{mozna-lokal}. Let for any $i = 1,2,\ldots$, the
homomorphism $\lambda'_i$ be the unique extension of $\lambda_i$
to the localization $(M/\rho_{j+1})(\wh{j+1}{A}_{j+1})^{-1}$.
Consider the diagram
\[ \xymatrix{
(M/\rho_j)(\wh{j}{A}_j)^{-1}
\ar@{}|{\mbox{\normalsize{$=$}}}^{\lambda'_j}[r]
\ar_{\varphi'_j}[d] &  M_{i_j} / \rho_j \times S_j \ar^{\wh{j}{\kappa}}[d]\\
(M/\rho_{j+1})(\wh{j+1}{A}_{j+1})^{-1}
\ar@{^{(}->}^{\lambda'_{j+1}}[r] & M_{i_{j+1}} / \rho_{j+1} \times
S_{j+1} } \tag{\#\#}\] where $\varphi'_j$ is the natural extension
of $\varphi_j$. We know that $\varphi_j$ is an epimorphism.

We will show that for $\rho_j$ of type $\h$ the map $\varphi'_j$
also is an epimorphism. We have $A_j=A_{j+1}$, because in case
$\h$ there is no new arc added. Consider the image $\wh{j}{A}_j
\ssq M/\rho_j$ under the map $\varphi_j \colon M/\rho_j
\rightarrow M/\rho_{j+1}$. We get $\wh{j+1}{A}_{j} \ssq
M/\rho_{j+1}$, and the elements of $\wh{j+1}{A}_{j}$ are central
in $M/\rho_{j+1}$ (because so are the elements of $\wh{j}{A}_j$ in
$M/\rho_j$). Moreover, we know that $\wh{j+1}{A}_j =
\wh{j+1}{A}_{j+1}$, because $A_j=A_{j+1}$, so that the image of
$\wh{j}{A}_j \ssq M/\rho_j$ under $\varphi_j$ is equal to
$\wh{j+1}{A}_{j+1}$. The image of $M/{\rho_j}$ is equal to
$M/{\rho_{j+1}}$, whence indeed $\varphi'_j$ is an epimorphism.

We know that the diagram $(\#)$ commutes. Thus, we have
\[(\lam'_{j+1} \varphi'_j)|_{M/\rho_j} = \lam_{j+1} \varphi_j = \wh{j}{\ka} \lam_j
= (\wh{j}{\ka} \lam'_j)|_{M/\rho_j}.\] Under this map, the images
of elements of $\wh{j}{A}_j$ are invertible in $M_{i_{j+1}} /
\rho_{j+1} \times S_{j+1}$, because $\varphi_j(\wh{j}{A}_j) =
\wh{j+1}{A}_j \ssq \wh{j+1}{A}_{j+1}$ and we know from the last
part of the proof of Lemma~\ref{mozna-lokal} that the elements
$\lambda_{j+1}((\wh{j+1}{A}_{j+1})^{-1})$ are invertible in
$M_{i_{j+1}} / \rho_{j+1} \times S_{j+1}$. Hence, there exists a
unique extension to the localization
$(M/\rho_j)(\wh{j}{A}_j)^{-1}$. It is equal to  $\lam'_{j+1}
\varphi'_j$ and also equal to $\wh{j}{\ka} \lam'_j$, so that
$\lam'_{j+1} \varphi'_j = \wh{j}{\ka} \lam'_j$. In other words,
the diagram $(\#\#)$ commutes.

We know that in case $\h$ the maps $\wh{j}{\kappa}$, $\lambda'_j$
and $\varphi'_j$ are epimorphisms. Hence, $\lam'_{j+1} \varphi'_j
= \wh{j}{\ka} \lam'_j$ is an epimorphism, so that the embedding
$\lam'_j$ is an epimorphism. Thus, in case $\h$, $\lam'_j$ is an
isomorphism. Then we get $(M/\rho_{j+1})(\wh{j+1}{A}_{j+1})^{-1} =
M_{i_{j+1}} / \rho_{j+1} \times S_{j+1}$, as desired. This
completes the inductive step in case $\h$.

In case $\diam$, $\wh{j}{\kappa}$ and $\varphi'_j$ are not
epimorphisms and $A_j \varsubsetneq A_{j+1}$. We have
$Im(\varphi'_j) = (M/\rho_{j+1})(\wh{j+1}{A}_j)^{-1}$. It also
follows that $Im(\wh{j}{\kappa}) \ssq Im(\lambda'_{j+1})$, because
\begin{multline*}
Im(\wh{j}{\kappa}) = \lambda'_{j+1}(Im(\varphi'_j)) =
\lambda'_{j+1}((M/\rho_{j+1})(\wh{j+1}{A}_j)^{-1}) \ssq \\
\ssq \lambda'_{j+1}((M/\rho_{j+1})(\wh{j+1}{A}_{j+1})^{-1}) =
Im(\lambda'_{j+1}).
\end{multline*}
Since $\wh{j}{\kappa}|_{S_j}=id$, we also have
\[S_j = \wh{j}{\kappa}(S_j) \ssq Im(\wh{j}{\kappa}) \ssq Im(\lambda'_{j+1}) =
\lambda'_{j+1}((M/\rho_{j+1})(\wh{j+1}{A}_{j+1})^{-1}).\]

Consider $M^{u, v}_{i_j}/\rho_{j+1}=M_{i_j}/\rho_{j+1}$. This is a
Chinese monoid $M^{u, v}_{i_j}$ with the additional relations of
type $\diam$, so with relations corresponding to the arc $a_v
a_u$. Hence, we are in a case as in Lemma~\ref{Mkaro}, where
relations of type $\diam$ are imposed on the Chinese monoid $M$.
Moreover, notice that $\lambda'_{j+1}|_{M_{i_j}/\rho_{j+1}}$
coincides with the map $\lambda$ from Lemma~\ref{Mkaro}. Hence we
may apply Lemma~\ref{Mkaro} to $M_{i_j}/\rho_{j+1}$. We then get
$i_{j+1}=i_j-2$, and $M_{i_{j+1}}/\rho_{j+1}$ corresponds to
$\ol{M_{n-2}}$, and more generally
\[\lambda'_{j+1}\left((M_{i_j}/\rho_{j+1}) \la
(\wh{j}{a}_v \wh{j}{a}_u)^{-1} \ra \right) =
 M_{i_{j+1}}/\rho_{j+1} \times B \times \Z.\]
Since $M_{i_j}/\rho_{j+1} \ssq M/\rho_{j+1}$ and $\la (\wh{j}{a}_v
\wh{j}{a}_u)^{-1} \ra \ssq (\wh{j+1}{A}_{j+1})^{-1}$, we thus get
\begin{multline*}
M_{i_{j+1}}/\rho_{j+1} \times B \times \Z =
\lam'_{j+1}((M_{i_j}/\rho_{j+1}) \la (\wh{j}{a}_v \wh{j}{a}_u)^{-1} \ra ) \ssq \\
\ssq \lam'_{j+1}((M/\rho_{j+1}) (\wh{j+1}{A}_{j+1})^{-1}) =
Im(\lam'_{j+1}).
\end{multline*}
Moreover we know that $S_j \ssq Im(\lam'_{j+1})$. This leads to
\[M_{i_{j+1}}/\rho_{j+1} \times S_{j+1} =  M_{i_{j+1}}/\rho_{j+1} \times B
\times \Z \times S_j \ssq Im(\lam'_{j+1}).\] The opposite
inclusion holds by the assumption. Hence the embedding
$\lam'_{j+1}$ is an epimorphism, which implies that
\[\lam'_{j+1}((M/\rho_{j+1}) (\wh{j+1}{A}_{j+1})^{-1}) = M_{i_{j+1}}/\rho_{j+1} \times S_{j+1}.\]
Then, as in case $\h$, $\lam'_{j+1}$ is an isomorphism, so that
$(M/\rho_{j+1}) (\wh{j+1}{A}_{j+1})^{-1} = M_{i_{j+1}}/\rho_{j+1}
\times S_{j+1}.$ This completes the inductive step in case
$\diam$, proving the assertion.
\end{proof}

Notice that Construction~\ref{interpr} assigns ideals of the
form $\mc{I}_{\rho (d)} \lhd K[M]$ to the vertices $d$ of the tree
$D$.

\begin{tw}
\label{liscieidpier} Under the correspondence described in
Construction~\ref{interpr}, every ideal $\mc{I}_{\rho (d)}$
of the algebra $K[M]$ corresponding to a leaf $d\in D$ is prime.
\end{tw}

\begin{proof}
By Proposition~\ref{centrlok}, for every leaf $d$ from level
$r$ of $D$ and for the corresponding congruence $\rho (d) =\rho_r$
on $M$, we have
\[(M/\rho_r) (\wh{r}{A}_r)^{-1} \simeq M_{i_r}/\rho_r \times S_r,\]
and $S_r$ is of the form $\N^k \times (B\times \Z)^m$ for some
exponents $k,m$.

If the extreme arc occurring in the diagram $d$ does not join
generators $a_{1},a_{n}$ then there are $i_r>0$ unused generators.
The congruence $\rho_r$ introduces the commutativity of these
remaining generators, so that $M_{i_{r}}/\rho_{r} \simeq
\N^{i_r}$. Hence
\[M_{i_r}/\rho_r \times S_r \simeq \N^{i_r} \times \N^k \times (B \times \Z)^m \simeq \N^k \times (B \times \Z)^m.\]

On the other hand, if the extreme arc joins $a_{1}$ and $a_{n}$
then $i_r=0$ and
\[M_{i_r}/\rho_r \times S_r \simeq M_0/\rho_r \times S_r
\simeq S_r \simeq \N^k \times (B \times \Z)^m.\] Therefore, we
have
\[K[M/\rho_r] (\wh{r}{A}_r)^{-1} \simeq K[(M/\rho_r)
(\wh{r}{A}_r)^{-1}] \simeq K[\N^k \times (B\times \Z)^m].\] It is
well known that, for every field $L$, the algebra $L[B]$ is
primitive, see \cite{TYLam}. From \cite{Krempa} it then follows
that $R[B]$ is prime for every prime algebra $R$. This easily
implies that $K[\N^k \times (B\times \Z)^m]$ is prime. Thus
$K[M/\rho_r](\wh{r}{A}_r)^{-1}$ is prime. Since it is a central
localization of $K[M/\rho_r]$, also $K[M/\rho_r] \simeq
K[M]/\I_{\rho_r}$ is prime. The assertion follows.
\end{proof}

\begin{lem}
\label{osobnegal} If $\rho = \rho (d)$ and $\rho'=\rho (d')$ are
congruences corresponding to diagrams $d$ and $d'$, and $d, d'$
are in different branches of $D$, then $\I_\rho \nsubseteq
\I_{\rho'}$.
\end{lem}

\begin{proof}
  First, consider the case where the root of $D$ is the only vertex of $D$ that
  is contained in both branches leading from the root to $d$ and from the root to
  $d'$. This means that these diagrams start with

  (1a) two different dots: $d$ with $a_s$, and $d'$ with $a_t$, where $t \neq
  s$, or

  (1b) two different arcs: $d$ with $a_sa_{s-1}$, and $d'$ with $a_ta_{t-1}$, where $t \neq s$,
  or

  (1c) one of them, say $d$, starts with an arc $a_sa_{s-1}$, and $d'$ starts with a dot $a_t$.

  Clearly, the middle of $d$ is different than the middle of $d'$.

  In case (1a) the image of $a_s$ in $M/\rho$ is central.
 If $d'$ consists of a single dot $a_t$, then the image of $a_s$ in $M/\rho'$ is not central because
 it does not commute with the images of generators lying on the other side of $a_t$.
 Otherwise, in $d'$, directly after the initial dot $a_t$, according to the rules,
 the arc $a_{t+1}a_{t-1}$ was built. Hence, Remark~\ref{lewoprawo} implies that in $M/\rho'$
 in the images of all generators on the left of $a_t$ there is a component $p$, while on
 the right there is $q$. Since $a_s \neq  a_t$, some component of the image of $a_s$ is equal
 to $p$ or $q$, hence also in this case this image is not central in $M/\rho'$.
 Therefore $\rho \nsubseteq \rho'$.

 In the same way we see that the image of $a_t$ is central in $M/\rho'$, but it is not central in $M/\rho$,
  whence $\rho' \nsubseteq \rho$.
  This proves the assertion in case (1a).

 Similarly, in case (1b), the image of $a_sa_{s-1}$ is central in $M/\rho$.
 Assume, with no loss of generality, that $s>t$. Then, by
 Remark~\ref{lewoprawo}, in the images of $a_s$ and $a_{s-1}$ in $M/\rho'$ there are
 components equal to $q$ and there are no components
 equal to $p$, so the image of $a_sa_{s-1}$ is not central in $M/\rho'$.
 In the same way we see that the image of $a_ta_{t-1}$ is central
 in $M/\rho'$, but is not central in $M/\rho$.
 This yields the assertion in case (1b).

 In case (1c), similarly, assume that $s>t$.
 The image of $a_sa_{s-1}$ is central in $M/\rho$.
If $d'$ consists of the single dot $a_t$, then in $M/\rho'$ we
have the same relations as in $M$ and additionally the images of
$a_1, \ldots, a_t$ commute and the images of $a_t, \ldots, a_n$
commute. Therefore, the image of $a_sa_{s-1}$ is not central in
$M/\rho'$, because it does not commute with the image of $a_{t-1}$
(since $a_t$ is a dot, we must have $t>1$).

On the other hand, if $d'$ is not a single dot, then as in case
(1a), in the diagram $d'$ directly after the initial dot $a_t$ the
arc $a_{t+1}a_{t-1}$ must have been built. Hence, by
Remark~\ref{lewoprawo}, in the image of $a_s$ in $M/\rho'$ one of
the components is equal to $q$, while in the image of $a_{s-1}$
there are no components equal to $p$. Hence, the image of
$a_sa_{s-1}$ is not central in $M/\rho'$.

Similarly, the  image of $a_t$ is central in $M/\rho'$, but it is
not central in $M/\rho$, because one of its components is equal to
$p$. Hence, again $\I_\rho \nsubseteq \I_{\rho'}$ and $\I_{\rho'}
\nsubseteq \I_\rho$.

This completes the proof in the case where the root of $D$ is the
only common vertex of the branches containing $d$ and $d'$.

Now, consider the opposite case. So, up to a certain step in the
construction of $D$ the diagrams $d$ and $d'$ are equal. Assume
that generators $a_{x+1}, \ldots, a_{y-1}$, where $1 < x+1 < y-1 <
n$, were used in this common part of the construction of $d$ and
$d'$ (so the diagram obtained in this step is not a leaf of $D$).
We may assume that $x+1 < y-1$, so the number of generators used
in the common part of the construction of $d$ and $d'$ exceeds
$1$, because if two diagrams start with the same dot then in both
of them the same arc must follow.

Hence, one of the following cases must occur.

(2a) In one of the diagrams, say in $d$, an arc $a_ya_x$ was
built, while in $d'$ a dot was built (say, to the right of the
previously used generators, so $a_y$, and then $y<n$).

(2b) In one of the diagrams, say in $d$, a dot $a_x$ was built on
one side, while in $d'$ a dot $a_{y}$ was built on the other side
 (then $x>1$ and $y<n$).

First, consider case (2a). Recall that \mbox{$<i>$} denotes $i$
consecutive used generators, so an initial step in the
construction of $d$ in this case looks like
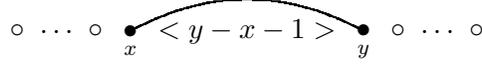
\[ \diag{
\circ & \ldots & \circ & \podpis{\bul}{x} \ar@/^1pc/@{-}[rr] &
<y-x-1> & \podpis{\bul}{y} & \circ & \ldots & \circ } \] while an
initial step in the construction of $d'$ is of the form
 \[ \diag{
 \circ & \ldots & \circ & \podpis{\circ}{x} & <y-x-1> & \podpis{\bul}{y}
 & \circ & \ldots & \circ } \]
In $d'$, before the dot $a_y$, some dots might have been added on
the same side of the previously used generators, and before this
an arc  $a_za_{x+1}$ must had been added, for some $x+1< z < y$.
So the following diagram is an initial step in the construction of
$d'$
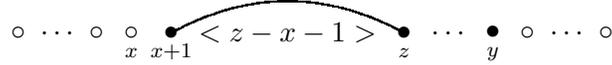
\[ \diag{
 \circ & \ldots & \circ & \podpis{\circ}{x} &  \podpis{\bul}{x+1} \ar@/^1pc/@{-}[rr] & <z-x-1> &
\podpis{\bul}{z} & \ldots & \podpis{\bul}{y} & \circ & \ldots &
\circ } \] In the next steps of the construction of $d'$, after
the dot $a_y$ some dots might have been added on the same side as
previously used generators, which was followed by one of the
following two steps.

(2a.1) An arc $a_wa_x$ was added for some $w >y$ (and perhaps the
construction of $d'$ was not complete yet). So an initial step of
the construction of $d'$ is of the form

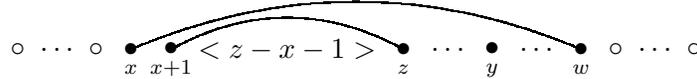
\[ \diag{
 \circ & \ldots & \circ & \podpis{\bul}{x} \ar@/^1.5pc/@{-}[rrrrrrr]
 & \podpis{\bul}{x+1} \ar@/^1pc/@{-}[rr] & <z-x-1> & \podpis{\bul}{z} & \ldots & \podpis{\bul}{y} & \ldots & \podpis{\bul}{w} &\circ & \ldots & \circ } \]

(2a.2) The construction of $d'$ was completed, so $d'$ is of the
form
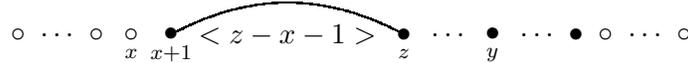
\[ \diag{
 \circ & \ldots & \circ & \podpis{\circ}{x}
 & \podpis{\bul}{x+1} \ar@/^1pc/@{-}[rr] & <z-x-1> & \podpis{\bul}{z}& \ldots & \podpis{\bul}{y} & \ldots & \bul &\circ & \ldots & \circ } \]

Since the generators $a_{x+1}, \ldots, a_{y-1}$ are in both
diagrams $d$ and $d'$ used in the same way, an initial step of the
construction of $d$ must be of the form

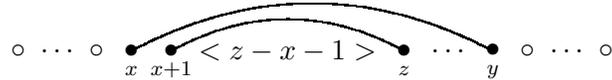
\[ \diag{
 \circ & \ldots & \circ & \podpis{\bul}{x} \ar@/^1.5pc/@{-}[rrrrr]
 & \podpis{\bul}{x+1}\ar@/^1pc/@{-}[rr] & <z-x-1> & \podpis{\bul}{z} & \ldots & \podpis{\bul}{y} & \circ & \ldots & \circ } \]
In $d$ the generator $a_{x+1}$ is not the middle, because if it
were the initial dot then an arc $a_{x+2}a_x$ would follow.
However, we know that $y \neq x+2$, because we assume that
$y-x-1>1$.

In view of Lemma~\ref{mozna-lokal}, the image of $a_ya_x$ is not
central in $M/\rho$.

First, consider case (2a.1). Since the construction of $\rho'$
involves a dot $a_y$, and at a later stage an arc $a_wa_x$, where
$x<y<w$, it follows that the image of $a_y$ in $M/{\rho'}$ is of
the form $(\ldots, 1,1, \ldots)$, while the image of $a_x$ is of
the form $(\ldots,p,g,\ldots)$, where the distinguished two
components $B \times \Z$ result from adding the arc $a_wa_x$ in
the construction of $M/\rho'$, and $\ldots$ denote the values of
the remaining components. The forms of these sequences are derived
as in the proof of Lemma~\ref{mozna-lokal}, by representing
$M/\rho'$ as a submonoid of an appropriate monoid $M_{i_l}/\rho'
\times S_l$, where $S_l$ is a direct product of some copies of $B
\times \Z$ and some copies of $\N$. Therefore, the image of
$a_ya_x$ is of the form $(\ldots,p,g,\ldots)$, whence -- because
of the component $p$ -- it is not central in $M/{\rho'}$. Hence,
in case (2a.1), $\rho \nsubseteq \rho'$.

Next, consider case (2a.2). Using an argument and notation as in
the proof of Lemma~\ref{mozna-lokal}, and applying
Remark~\ref{lewoprawo} and the fact that $a_y$ appears in $d'$ to
the right of the middle of $d'$, one can get a more detailed
description of the image of $a_y$ in $M/\rho'$ as
$(1,\ldots,1,g_y,[1,q])$, where $g_y$ appears in a component
corresponding to $\N$ in $M/{\rho'}$. Similarly, the image of
$a_x$ has the form $(\wh{l}{a}_x, 1, \ldots, 1, [1,p])$, where
elements $p$ in $[1,p]$ occur in the same components as elements
$q$ occur in $[1,q]$ in the image of $a_y$. Hence, the image of
$a_ya_x$ has the form $(\wh{l}{a}_x, 1, \ldots, 1, g_y, [1,1])$.
Moreover, $\wh{l}{a}_x \in M_{i_l}/\rho'$ does not commute with
$\wh{l}{a}_{y+1} \in M_{i_l}/\rho'$, so the image of $a_ya_x$ is
not central in $M/\rho'$. Therefore, in case (2a.2) we also have
$\rho \nsubseteq \rho'$.

Hence, in both subcases of case (2a) we get $\rho \nsubseteq
\rho'$ and in the rest of the proof we treat both these cases
together.

We claim that the image of $a_ya_{x+1}$ in $M/{\rho'}$ is central,
but its image in $M/\rho$ is not central, which will yield $\rho'
\nsubseteq \rho$. As above, we know that the image of $a_y$ in
$M/\rho'$ looks like $(1,\ldots,1,g_y,[1,q])$, where $g_y$ occurs
in a component corresponding to $\N$ in $M/{\rho'}$. On the other
hand, $a_{x+1}$ appears in $d'$ to the left of the middle of $d'$,
so in the image of $a_{x+1}$ the components are $1$, $p$ and a
single $g$ (in the component corresponding to an arc with the left
end in $a_{x+1}$). Moreover, the elements $p$ occur in the same
components in which the elements $q$ occur in the image of $a_y$.
Therefore all components of the image of $a_ya_{x+1}$ are equal to
$1$ except for a single component $g$ and a single component
$g_y$. Hence, this element is central in $M/{\rho'}$.

Similarly, the element $a_y$ in the diagram $d$ is to the right of
the middle of $d$, so the image of $a_{y}$ in $M/\rho$ has the
form $(1,\ldots,1,q,1,[1,q])$. As noticed before, the element
$a_{x+1}$ is in $d$ to the left of the middle of $d$, whence in
the image of $a_{x+1}$ in $M/\rho$ there are components $1$, $p$
and either a single component $g$, if the generator $a_{x+1}$ was
used in an arc, or a single $g_{x+1}$, if it was used as a dot.
Moreover, all components different than $1$ occur in the
components corresponding to the part of $d$, which was built
before the construction of the arc $a_ya_x$, so in some of the
components covered by $[1,q]$ in the image of $a_y$. The remaining
components of the image of $a_{x+1}$ are all equal to $1$. This
implies that the image of $a_ya_{x+1}$ is not central in $M/\rho$,
because $q$ is not central. Hence, we indeed get $\rho' \nsubseteq
\rho$, as desired.

It follows that $\I_\rho \nsubseteq \I_{\rho'}$ and $\I_{\rho'}
\nsubseteq \I_\rho$, which completes the proof in case (2a).

Finally, we deal with case (2b). This case can occur only if the
common initial part of $d$ and $d'$ is ,,covered" with an arc
(only in this case a dot can be added both on the left and on the
right of the arc). Hence, an initial part of the construction of
$d'$ has the form
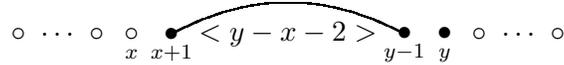
\[ \diag{
 \circ & \ldots & \circ & \podpis{\circ}{x} &  \podpis{\bul}{x+1}
 \ar@/^1pc/@{-}[rr] & <y-x-2> &
 \podpis{\bul}{y-1} & \podpis{\bul}{y} & \circ & \ldots & \circ } \]
which is a special case of the diagram $d'$ described in case
(2a). So we know that the image of $a_ya_{x+1}$ is central in
$M/\rho'$. In this case
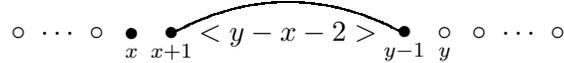
\[ \diag{
 \circ & \ldots & \circ & \podpis{\bul}{x} &  \podpis{\bul}{x+1}
 \ar@/^1pc/@{-}[rr] & <y-x-2> & \podpis{\bul}{y-1} & \podpis{\circ}{y} &
  \circ & \ldots & \circ } \]
is an initial step in the construction of the diagram $d$, whence
$a_{x+1}$ is on the left of the middle of $d$. Hence, as in case
(2a), one can show that the image of $a_ya_{x+1}$ is not central
in $M/\rho$. By symmetry, the image of $a_{y-1}a_x$ is central in
$M/\rho$, but is not central in $M/\rho'$. Then, again $\I_\rho
\nsubseteq \I_{\rho'}$ and $\I_{\rho'} \nsubseteq \I_\rho$, which
completes the proof in case (2b) and therefore the proof of the
lemma.
\end{proof}

\subsection{Minimal prime ideals as the leaves of $D$}
\label{s-minim}

\

In this part we first prove that every prime ideal of $K[M]$
contains an ideal $\I_{\rho (d)}$ corresponding to a leaf $d$ of
the tree $D$. In view of Construction~\ref{interpr}, this
strengthens the assertion of Theorem~\ref{idealy} (since $\I_\rho
\subseteq \I_{\rho (d)}$ for an ideal $\I_{\rho}$ of $\h$ or
$\diam$ type). This is then used to prove the main result of the
paper, Theorem~\ref{bijekcja}, which shows that all the above
ideals $\I_{\rho (d)}$ are actually minimal prime ideals of $K[M]$
and establishes a bijection between the set of leaves of $D$ and
the set of minimal prime ideals of $K[M]$.

\begin{tw}
  \label{idpierzawlisc}
Every prime ideal of $K[M]$ contains a prime ideal of the
form $\I_{\rho (d)}$, where $\rho (d)$ is the congruence
corresponding to a leaf $d$ of $D$.
\end{tw}

\begin{proof}
Let $P$ be a fixed minimal prime ideal of $K[M]$. By
Theorem~\ref{idealy}, $P$ contains an ideal of the form
$\mc{I}_{\rho}$, where $\rho = \rho (e)$ is a congruence
corresponding to a diagram $e$ in the first level of the tree $D$,
so it is of one of the following types:

$\h$) $\rho$ corresponds to a diagram $\ \diag{ \circ & \ldots &
\circ & \bullet & \circ & \ldots & \circ }\ $ where the indicated
dot is neither $a_{1}$ nor $a_{n}$,

$\diam$) $\rho$ corresponds to a diagram $\ \diag{ \circ & \ldots
& \circ & \bullet \ar@/^0.4pc/@{-}[r] & \bullet & \circ & \ldots &
\circ } \ $.

 \noindent Moreover, in cases $\h$ and $\diam$
respectively, the homomorphisms from Lemmas~\ref{Mkier} and
\ref{Mkaro} can be extended to the corresponding semigroup
algebras. This leads to homomorphisms
\[ \xymatrix{
K[M] \ar[r] & K[M_{\rho}] \ar@{->>}[d] \\ & K[M]/P }
\hspace{-1ex}\simeq K[\ol{M_{n-1}^s} \times \la \ol{a_s} \ra]
\simeq K[M_{n-1}^s/\rho][\N]\]
\[\hspace{-3cm} \xymatrix{
K[M] \ar[r] & K[M_{\rho}] \ar@{->>}[d] \ar@{^{(}->}[r] &
K[M_{\rho} \, \la(\ol{a_sa_{s-1}})^{-1}\ra]
    \simeq \mbox{\makebox[10em][l]{$K[\ol{M_{n-2}^{s-1,s}} \times B \times \Z]
    \simeq K[M_{n-2}^{s-1,s}/ \rho][B \times \Z]$}}
  \\
{} & K[M]/P } \] in cases $\h$ and $\diam$, respectively, where
the embedding is accomplished via the central localization with
respect to $\la \ol{a_s}\ol{a_{s-1}}\ra$.

Consider all chains of congruences $\rho_0 \varsubsetneq \rho_1
\varsubsetneq \rho_2 \varsubsetneq \ldots \varsubsetneq \rho_j$ on
$M$, corresponding to a fragment of a branch of the tree $D$, such
that $\mc{I}_{\rho_j} \ssq P$. By Lemma~\ref{diagr-przem} we then
have natural homomorphisms
\[ \xymatrix{
K[M] \ar[r] & K[M/\rho_j] \ar@{->>}[d] \ar@{^{(}->}[r]  &  K[M_{i_j} / \rho_j][S_j] \\
& K[M]/P} \] where $S_j$ and $i_j$ are defined as in \ref{St}.

In the set of all such chains we choose a chain for which $i_j$ is
minimal. We will show that $\rho_j$ is the congruence
corresponding to a leaf $d$ of $D$; in other words
$\rho_{j}=\rho (d)$. Suppose otherwise. Then $i_j
>0$ and $M_{i_j}/\rho_j$ is not a free abelian monoid of rank
$i_j$, because none of the generators $a_{1},a_{n}$ have been
used. More precisely,
\begin{multline}
\label{otimes} M_{i_j}/\rho_j =
M^{u,v}_{i_j}/(\rho_j|_{M^{u,v}_{i_j}}) =
\la a_1, \ldots, a_{u-1}, a_{v+1}, \ldots, a_n \colon \\
\underbrace{a_1, \ldots, a_{u-1}}_{\textup{commute}},
\underbrace{a_{v+1}, \ldots, a_n}_{\textup{commute}} \textup{, and
the relations of a Chinese monoid hold}\ra.
\end{multline}

Consider an equality $\alpha_{1} K[M_{i_j}] \beta_1 = 0$ of type
$\boxplus$ (see Notation~\ref{boxplus}), where $\alpha_{1},
\beta_1 \in K[M_{i_j}]$. Then, in $K[M_{i_j}/\rho_j][S_j]$ we get
$\wh{j}{\alpha}_1 K[M_{i_j} \big/ \rho_j][S_j] \wh{j}{\beta}_1 =
0$, where $\wh{j}{x}$ denotes the image of $x$ in
$K[M_{i_j}/\rho_j][S_j]$. Notice that $K[M/\rho_j]$ embeds into
$K[M_{i_j}/\rho_j][S_j]$, because $M/\rho_j \hookrightarrow
M_{i_j} / \rho_j \times S_j$. We can identify $M/\rho_j$ with its
image under this embedding. Let
\begin{align*}
\mc{I}_{\alpha_{1}} &= (K[M_{i_j} / \rho_j]\, \wh{j}\alpha_{1}\, K[M_{i_j} / \rho_j])[S_j] \cap K[M/\rho_j], \\
\mc{I}_{\beta_1} &= (K[M_{i_j} / \rho_j]\, \wh{j}\beta_1\,
K[M_{i_j} / \rho_j])[S_j] \cap K[M/\rho_j].
\end{align*}
Then  \[\mc{I}_{\alpha_{1}} \cdot \mc{I}_{\beta_1} \ssq
\left(K[M_{i_j} / \rho_j]\, \wh{j}\alpha_{1}\, K[M_{i_j} / \rho_j]
\, \wh{j}\beta_1\, K[M_{i_j} / \rho_j]\right) [S_j] \cap
K[M/\rho_j]= 0.\] Moreover $\mc{I}_{\alpha_{1}}, \mc{I}_{\beta_1}
\lhd K[M/\rho_j]$, because $\mc{I}_{\alpha_{1}}, \mc{I}_{\beta_1}
\lhd K[M_{i_j}/\rho_j][S_j]$.

Let $\wh{j}{P}$ be the image of $P$ in $K[M/\rho_j]$. Since
$\I_{\rho_j} \ssq P$, there exists a natural map $K[M/\rho_j] \rar
K[M]/P$ whose kernel is $\wh{j}{P}$. Moreover, $K[M/\rho_j] \big/
\wh{j}{P} \simeq K[M]/P$. In particular, $\wh{j}{P}$ is a prime
ideal in $K[M/\rho_j]$. So, for every pair of ideals
$\mc{I}_{\alpha_{1}}, \mc{I}_{\beta_1}$, since
$\mc{I}_{\alpha_{1}} \cdot \mc{I}_{\beta_1} = 0$, we get
$\mc{I}_{\alpha_{1}} \ssq \wh{j}{P}$ or $\mc{I}_{\beta_1} \ssq
\wh{j}{P}$. Let $\g_1=\alpha_{1}$ if $\mc{I}_{\alpha_{1}} \ssq
\wh{j}{P}$ and let $\g_1=\beta_1$ otherwise (then we must have
$\mc{I}_{\beta_1} \ssq \wh{j}{P}$). Since $\mc{I}_{\g_1} \ssq
\wh{j}{P}$, there exists a natural homomorphism $K[M/\rho_j]/
\mc{I}_{\g_1}\rightarrow K[M]/P$.

Now, consider another pair $\alpha_{2}, \beta_2 \in K[M_{i_j}]$ of
type $\boxplus$. The equalities $\alpha_{2} K[M_{i_j}] \beta_2 =
0$ hold in $\left(\raisebox{0.5ex}{$K[M_{i_j}/\rho_j]$} \big/
\raisebox{-0.5ex}{$\wh{j}{\g}_1=0$}\right)[S_j]$. It follows that
\[\wh{j,1}{\alpha}_2 \, \left(\raisebox{0.5ex}{$K[M_{i_j}/\rho_j]$}
\big/ \raisebox{-0.5ex}{$\wh{j}{\g}_1=0$}\right)[S_j] \,
\wh{j,1}{\beta}_2=0,\] where $\wh{j,1}{x}$ denotes the image of
$x$ in $\left(\raisebox{0.5ex}{$K[M_{i_j}/\rho_j]$} \big/
\raisebox{-0.5ex}{$\wh{j}{\g}_1=0$}\right)[S_j]$. Let
$\wh{j,1}{P}$ be the image of $P$ in the algebra
$\raisebox{0.5ex}{$K[M_{i_j}/\rho_j]$} \big/
\raisebox{-0.5ex}{$\wh{j}{\g}_1=0$}$. We define
$\mc{I}_{\alpha_{2}}$ and $\mc{I}_{\beta_2}$ by
\begin{align*}
\mc{I}_{\alpha_{2}} &= \left(\left(\raisebox{0.5ex}{$K[M_{i_j}/\rho_j]$} \big/ \raisebox{-0.5ex}{$\wh{j}{\g}_1=0$}\right)\, \wh{j,1}\alpha_{2}\, \left(\raisebox{0.5ex}{$K[M_{i_j}/\rho_j]$} \big/ \raisebox{-0.5ex}{$\wh{j}{\g}_1=0$}\right)\right)[S_j] \cap \raisebox{0.5ex}{$K[M/\rho_j]$} \big/ \raisebox{-0.5ex}{$\mc{I}_{\g_1}$}, \\
\mc{I}_{\beta_2} &=
\left(\left(\raisebox{0.5ex}{$K[M_{i_j}/\rho_j]$} \big/
\raisebox{-0.5ex}{$\wh{j}{\g}_1=0$}\right)\, \wh{j,1}\beta_2\,
\left(\raisebox{0.5ex}{$K[M_{i_j}/\rho_j]$} \big/
\raisebox{-0.5ex}{$\wh{j}{\g}_1=0$}\right)\right)[S_j] \cap
\raisebox{0.5ex}{$K[M/\rho_j]$} \big/
\raisebox{-0.5ex}{$\mc{I}_{\g_1}$}.
\end{align*}
As above, we see that $\mc{I}_{\alpha_{2}}, \mc{I}_{\beta_2}$ are
ideals in $\raisebox{0.5ex}{$K[M/\rho_j]$} \big/
\raisebox{-0.5ex}{$\mc{I}_{\g_1}$}$, the ideal $\wh{j,1}{P}$ is
prime in $\raisebox{0.5ex}{$K[M_{i_j}/\rho_j]$} \big/
\raisebox{-0.5ex}{$\wh{j}{\g}_1=0$}$ and either
$\mc{I}_{\alpha_{2}} \ssq \wh{j,1}{P}$ or $\mc{I}_{\beta_2} \ssq
\wh{j,1}{P}$. Let $\g_2=\alpha_{2}$, if $\mc{I}_{\alpha_{2}} \ssq
\wh{j,1}{P}$ and let $\g_2=\beta_2$ otherwise (then we must have
$\mc{I}_{\beta_2} \ssq \wh{j,1}{P}$). Since $\mc{I}_{\g_2} \ssq
\wh{j,1}{P}$, there exists a homomorphism
$\raisebox{0.5ex}{\raisebox{0.5ex}{$K[M/\rho_j]$} \big/
\raisebox{-0.5ex}{$\mc{I}_{\g_1}$}} \Big/
\raisebox{-0.5ex}{$\mc{I}_{\g_2}$} \rightarrow K[M]/P$.

Similarly one shows that the image $\wh{j,2}{P}$ of $P$ in
$\raisebox{0.5ex}{\raisebox{0.5ex}{$K[M/\rho_j]$} \big/
\raisebox{-0.5ex}{$\mc{I}_{\g_1}$}} \Big/
\raisebox{-0.5ex}{$\mc{I}_{\g_2}$}$ is a prime ideal and
\[\raisebox{0.5ex}{\raisebox{0.5ex}{$K[M/\rho_j]$} \big/ \raisebox{-0.5ex}{$\mc{I}_{\g_1}$}}
\Big/ \raisebox{-0.5ex}{$\mc{I}_{\g_2}$}  \hookrightarrow
\left(\raisebox{0.5ex}{\raisebox{0.5ex}{$K[M_{i_j}/\rho_j]$} \big/
\raisebox{-0.5ex}{$\wh{j}{\g}_1=0$}} \Big/
\raisebox{-0.5ex}{$\wh{j,1}{\g}_2=0$}\right)[S_j].\]

By the hypothesis, the above construction yields
\[ \xymatrix{
P \ar@{}[r]|(0.3){\mbox{\normalsize{$\lhd$}}} \ar[d] & K[M] \ar[d] \\
\wh{j}{P} \ar@{}[r]|(0.3){\mbox{\normalsize{$\lhd$}}} \ar[d] & K[M/\rho_j] \ar[d] \ar@{^{(}->}[r] & K[M_{i_j} / \rho_j][S_j] \ar[d]\\
\wh{j,1}{P} \ar@{}[r]|(0.3){\mbox{\normalsize{$\lhd$}}} \ar[d] & \raisebox{0.5ex}{$K[M/\rho_j]$} \big/ \raisebox{-0.5ex}{$\mc{I}_{\g_1}$} \ar[d] \ar@{^{(}->}[r] & \left(\raisebox{0.5ex}{$K[M_{i_j}/\rho_j]$} \big/ \raisebox{-0.5ex}{$\wh{j}{\g}_1=0$}\right)[S_j] \ar[d]\\
\wh{j,2}{P} \ar@{}[r]|(0.3){\mbox{\normalsize{$\lhd$}}} &
\raisebox{0.5ex}{\raisebox{0.5ex}{$K[M/\rho_j]$} \big/
\raisebox{-0.5ex}{$\mc{I}_{\g_1}$}} \Big/
\raisebox{-0.5ex}{$\mc{I}_{\g_2}$} \ar[d] \ar@{^{(}->}[r] &
\left(\raisebox{0.5ex}{\raisebox{0.5ex}{$K[M_{i_j}/\rho_j]$} \big/ \raisebox{-0.5ex}{$\wh{j}{\g}_1=0$}} \Big/ \raisebox{-0.5ex}{$\wh{j,1}{\g}_2=0$}\right)[S_j] \\
& K[M]/P} \] where the ideals in the first column are prime and
the kernels of the three homomorphisms from $K[M]$ to the
subsequent three algebras in the second column are contained in
$P$, because $\mc{I}_{\rho_j} \ssq P$, and also $\mc{I}_{\g_1}
\ssq \wh{j}{P}$ and $\mc{I}_{\g_2} \ssq \wh{j,1}{P}$.

Let $\wh{j,2}{x}$ denote the image of $x$ in
$\left(\raisebox{0.5ex}{\raisebox{0.5ex}{$K[M_{i_j}/\rho_j]$}
\big/ \raisebox{-0.5ex}{$\wh{j}{\g}_1=0$}} \Big/
\raisebox{-0.5ex}{$\wh{j,1}{\g}_2=0$}\right)[S_j]$. Similarly, we
define also $\wh{j,m}{x}$ for $m \geq 3$, using other pairs of
elements $\alpha, \beta \in K[M_{i_j}]$ of type $\boxplus$. Let
$\wh{j,0}{x}$ denote $\wh{j}{x}$.

Notice that each of the elements $\alpha,\beta \in K[M_{i_j}]$ of
type $\boxplus$  is a difference of two elements of $M_{i_j}$, see
Theorem~\ref{rown}. Hence, all considered elements $\g$ are also
of this type. Put $\g_k=l_k-p_k$, where $l_k, p_k \in M_{i_j}$.
Then it is clear that
\[\left(\raisebox{0.5ex}{$K[M_{i_j}/\rho_j]$} \big/ \raisebox{-0.5ex}{$\wh{j}{\g}_1=0$}\right)[S_j] =
K\left[\raisebox{0.5ex}{$M_{i_j}/\rho_j$} \big/
\raisebox{-0.5ex}{$\wh{j}{l}_1=\wh{j}{p}_1$}\right][S_j].\] We
also get
\[\left(\raisebox{0.5ex}{\raisebox{0.5ex}{$K[M_{i_j}/\rho_j]$}
\big/ \raisebox{-0.5ex}{$\wh{j}{\g}_1=0$}} \Big/
\raisebox{-0.5ex}{$\wh{j,1}{\g}_2=0$}\right)=
\raisebox{0.5ex}{$K[M_{i_j}/\rho_j]$} \big/
\raisebox{-1ex}{$(\wh{j}{\g}_1, \wh{j}{\g}_2)$} =
\raisebox{0.5ex}{$K[M_{i_j}/\rho_j]$} \big/
\raisebox{-1ex}{$\begin{cases}\substack{\wh{j}{\g}_1=0 \\
\wh{j}{\g}_2=0}\end{cases}$},\] which, as above, leads to
\[\left(\raisebox{0.5ex}{\raisebox{0.5ex}{$K[M_{i_j}/\rho_j]$} \big/ \raisebox{-0.5ex}{$\wh{j}{\g}_1=0$}} \Big/ \raisebox{-0.5ex}{$\wh{j,1}{\g}_2=0$}\right)[S_j]=
K\left[\raisebox{0.5ex}{\raisebox{0.5ex}{$M_{i_j}/\rho_j$} \big/
\raisebox{-1ex}{$\begin{cases}\substack{\wh{j}{l}_1=\wh{j}{p}_1 \\
\wh{j}{l}_2=\wh{j}{p}_2} \end{cases}$}}\right][S_j].\] Proceeding
in this way, until all $t$ pairs $\alpha, \beta \in K[M_{i_j}]$ of
type $\boxplus$ are used, we extend the above diagram by adding
more rows. As above, we get the following form of the last two
rows of this diagram:
\[ \xymatrix{ \wh{j,t}{P} \ar@{}[r]|(0.2){\mbox{\normalsize{$\lhd$}}} &
\raisebox{0.5ex}{\raisebox{0.5ex}{\raisebox{0.5ex}{\raisebox{0.5ex}{$K[M/\rho_j]$}
\big/ \raisebox{-0.5ex}{$\mc{I}_{\g_1}$}} \big/
\raisebox{-0.5ex}{$\mc{I}_{\g_2}$}} \big/
\raisebox{-0.5ex}{$\ldots$}}  \big/
\raisebox{-0.5ex}{$\mc{I}_{\g_t}$}  \ar[d] \ar[d] \ar@{^{(}->}[r]
& K\left[\raisebox{1.5ex}{\raisebox{0.5ex}{$M_{i_j}/\rho_j$} \Big/
\raisebox{-2ex}{$\begin{cases}\substack{\wh{j}{l}_1=\wh{j}{p}_1 \\ \wh{j}{l}_2=\wh{j}{p}_2 \\
\, \vdots \\ \wh{j}{l}_t = \wh{j}{p}_t} \end{cases}$}}\right][S_j] \\
& K[M]/P} \]

Let $\eta$ be the congruence on $M_{i_j}$ generated by the set
$\{({l}_1, {p}_1) ,({l}_2, {p}_2), \ldots , ({l}_t, {p}_t)\}$.
Then
\begin{multline*}
K\left[\raisebox{1.5ex}{\raisebox{0.5ex}{$M_{i_j}/\rho_j$} \Big/
\raisebox{-2ex}{$\begin{cases}\substack{\wh{j}{l}_1=\wh{j}{p}_1 \\
\wh{j}{l}_2=\wh{j}{p}_2 \\ \, \vdots \\ \wh{j}{l}_t = \wh{j}{p}_t}
\end{cases}$}}\right][S_j] =
\left(\raisebox{0.5ex}{\raisebox{0.5ex}{$K[M_{i_j}]$} \big/ \raisebox{-0.5ex}{$\I_{\rho_j}$}} \Big/ \raisebox{-0.5ex}{$(\wh{j}{\g}_1, \wh{j}{\g}_2, \ldots, \wh{j}{\g}_t$)}\right) [S_j] = \\
= \left(K[M_{i_j}] \big/ \left(\I_{\rho_j} \cup (\g_1, \g_2,
\ldots, \g_t)\right)\right) [S_j] = K[M_{i_j}/(\rho_j \vee
\eta)][S_j],
\end{multline*}
where $\lambda_1 \vee \lambda_2$ denotes the congruence generated
by $\lambda_1$ and $\lambda_2$.

The congruence $\eta$ is defined by a set containing one element
from each pair $(\alpha,\beta)$ of type $\boxplus$ for
$K[M_{i_j}]$, so by Theorem~\ref{idealy} it contains a congruence
$\eta_0$ of type $\h_{i_j}$ or $\diam_{i_j}$ on $M_{i_j}$.
Therefore, $\eta_0 \nsubseteq \rho_j|_{M_{i_j}}$, and so
$\rho_j|_{M_{i_j}} \varsubsetneq \rho_j|_{M_{i_j}} \vee \eta_0$
(see the description of $M_{i_j}/\rho_j$ in (\ref{otimes})).

We know that $M_{i_j}/\rho_j \times S_j = M^{u, v}_{i_j}/\rho_j
\times S_j$, so the generators $a_u, \ldots, a_v$ have been used,
for some $1 < u \leq v < n$. Let $\omega$ be the kernel of the map
$M \rar \left(M_{i_j}/(\rho_j \vee \eta_0)\right) \times S_j$. The
above construction implies that $\omega$ satisfies $\I_\omega \ssq
P$.

Let $\rho_{j-1}$ be the congruence corresponding to the diagram
$d_{j-1}$ of level $j-1$ in the tree $D$, which is connected to
the diagram $d_{j}$ corresponding to $\rho_j$. We will show that
one of the following cases holds.

(A) There exists a congruence $\rho_{j+1}$ on $M$ such that
$\rho_j \varsubsetneq \rho_{j+1} \ssq \omega$, and $\rho_{j+1}$
corresponds to a diagram $d_{j+1}$ in $D$, which is connected to
the diagram $d_j$. In this case, since $\I_{\omega} \ssq P$, we
get $\I_{\rho_{j+1}} \ssq \I_\omega \ssq P$ and $i_{j+1} < i_j$.

For this, we will find a congruence $\chi$ on $M_{i_j}$, of type
$\h_{i_j}$ or $\diam_{i_j}$, such that $\rho_j|_{M_{i_j}}
\varsubsetneq \chi \ssq \rho_j|_{M_{i_j}} \vee \eta_0$ and the
congruence $\ddot{\chi}$ on $M$ which is the kernel of the natural
homomorphism $M \rar M_{i_j}/\rho_j \times S_j \rar M_{i_j}/\chi
\times S_j$ corresponds to a diagram in $D$, lying below the
diagram corresponding to $\rho_j$; then we will put $\rho_{j+1} =
\ddot{\chi}$.

(B) There exists a congruence $\rho'_j$ such that $\rho_{j-1}
\varsubsetneq \rho'_j$, where $i'_j<i_j$, and $i'_j$ is the number
of unused generators in $\rho'_j$. Moreover, the congruence
$\rho'_j$ corresponds to a diagram in $D$, which is connected to
the diagram $d_{j-1}$ and $\rho'_j \ssq \omega$. In this case,
since $\I_{\omega} \ssq P$, we get $\I_{\rho'_j} \ssq \I_\omega
\ssq P$.

Since both cases contradict the choice of $i_j$, this will
complete the proof of the fact that $\rho_j$ corresponds to a leaf
of $D$.

We know that $\eta_0$ is a congruence of type $\h$ or $\diam$ on
$M_{i_j} = M^{u,v}_{i_j}$. If $\eta_0$ corresponds to a diagram
(on $M_{i_j}$) which is an arc $a_{v+1}a_{u-1}$, then $\omega$
corresponds to a diagram in $D$ (on $M$) and we put
$\chi=\rho_j|_{M_{i_j}} \vee \eta_0$. Then $\omega = \ddot{\chi}$
and we define $\rho_{j+1} = \ddot{\chi}$. We thus get case (A).

We consider the remaining possibilities.

If $\eta_0$ corresponds to a diagram (on $M_{i_j}$) which is an
arc $a_sa_{s-1}$ for some \mbox{$s>v+1$}, then $a_{v+1}$ becomes
central in the monoid $M_{i_j}/(\rho_j \vee \eta_0)$ (it commutes
with $a_{j}$ for $j<v+1$ because $s>v+1$ and commutes with all
$a_{j}$ for $j>v+1$ because of the congruence $\rho_j$). A
symmetric argument shows that, if $s < u$, then $a_{u-1}$ becomes
central.

Similarly, if $\eta_0$ corresponds to a diagram (on $M_{i_j}$)
which is a dot $a_s$ for some $s> v$, then the elementu $a_{v+1}$
becomes central in $M_{i_j}/(\rho_j \vee \eta_0)$. A symmetric
argument shows that, if $s < u$, then $a_{u-1}$ becomes central.

Therefore, in both considered cases, the congruence
$\rho_j|_{M_{i_j}} \vee \eta_0$ induces centrality of $a_{u-1}$ or
of $a_{v+1}$ in the image of $M_{i_j}$, so $\omega$ is a
congruence corresponding to the diagram containing a dot
neighboring the previously used generators. If these used
generators are covered with an arc, then the new diagram obtained
by adding the dot neighboring this arc is allowed by
Definition~\ref{drzewo}, so it is an element of the tree defined
for $M_{i_j}$. Let $\chi$ denote the congruence on $M_{i_j}$
corresponding to this new diagram; so $\chi \ssq \rho_j|_{M_{i_j}}
\vee \eta_0$. The second case is when the new dot is on the same
side as some recently added dot. Then it is also easy to see that
we get a congruence $\chi \ssq \rho_j|_{M_{i_j}} \vee \eta_0$ that
corresponds to a diagram on $M_{i_j}$.

In both cases we may thus define $\rho_{j+1} = \ddot{\chi}
\varsupsetneq \rho_j$ and conditions in case (A) are satisfied, in
particular $i_{j+1} < i_j$.

It remains to consider the case where the diagram $d_{j}$
corresponding to $\rho_j$ contains dots on one side and the
considered ,,new dot'' (coming from $\eta_0$) is on the other
side. In this case, we construct a congruence $\rho'_j$ that
satisfies conditions in (B).

Assume that the last step in the construction of $d_{j}$ was the
dot $a_u$, while the new dot is the dot $a_{v+1}$. Then, by
Lemma~\ref{Mkier}, we get $M_{i_{j-1}}/\rho_{j-1} \simeq
M_{i_j}/\rho_j \times \la a_u \ra$. This corresponds to replacing
$M_{i_{j-1}}/\rho_{j-1} \times S_{j-1}$ by $M_{i_j}/\rho_j \times
\la a_u \ra \times S_{j-1}$ (in the process of constructing
$\rho_j$ from $\rho_{j-1}$; see Construction~\ref{interpr}).

Let $e$ be the diagram in $D$ obtained from $d_{j-1}$ by adding
the arc $a_{v+1}a_u$. We will show that the congruence $\rho'_j$,
corresponding to $e$, is contained in $\omega$.

Let $M'$ be the image of $M_{i_j}/\rho_j \times \la a_u \ra$
obtained by making the generator $a_{v+1}$ central in the first
component, in other words
\[M'= \left(\raisebox{0.5ex}{$M_{i_j}/\rho_j$}
\big/ \raisebox{-0.5ex}{($a_{v+1}$ {\scriptsize{central}})}\right)
\times \la a_u \ra.\] We have to check that the following
relations hold in $M'$:

-- the image of $a_{v+1}a_u$ is central,

-- the images of $a_wa_{v+1}a_z$ and $a_za_{v+1}a_w$ are equal for
$w,z<u$,

-- the images of $a_wa_ua_z$ and $a_za_ua_w$ are equal for
$w,z>v+1$.

\noindent These are the relations that are imposed on
$M_{i_j}/\rho_{j-1}$ in the process of constructing $\rho'_j$ from
$\rho_{j-1}$ by adding the arc $a_{v+1}a_u$ (see the definition of
an ideal of type $\diam$ in Definition~\ref{types-def} and
Construction~\ref{interpr}).

The image of $a_{v+1}a_u$ in $M_{i_j}/\rho_j \times \la a_u \ra$
is equal to $(\wh{j}{a}_{v+1}, a_u)$. In $M'$, the element
$a_{v+1}$ becomes central in the first component. Hence the image
of $a_{v+1}a_u$ is central in $M'$.

The image of $a_wa_{v+1}a_z$ in $M_{i_j}/\rho_j \times \la a_u
\ra$ is equal to $(\wh{j}{a}_w \wh{j}{a}_{v+1} \wh{j}{a}_z, 1)$,
while $(\wh{j}{a}_z \wh{j}{a}_{v+1} \wh{j}{a}_w, 1)$ is the image
of $a_za_{v+1}a_w$. Since $a_{v+1}$ is central in the first
component and the images of $a_z$ and $a_w$ commute for $w,z<u$,
we get $\wh{j}{a}_w \wh{j}{a}_{v+1} \wh{j}{a}_z = \wh{j}{a}_{v+1}
(\wh{j}{a}_w \wh{j}{a}_z) =
 \wh{j}{a}_{v+1} (\wh{j}{a}_z \wh{j}{a}_w) =  \wh{j}{a}_z \wh{j}{a}_{v+1}
 \wh{j}{a}_w$.
 So in $M'$ the images of
$a_wa_{v+1}a_z$ and $a_za_{v+1}a_w$ are equal.

Similarly, the image of $a_wa_ua_z$ in $M_{i_j}/\rho_j \times \la
a_u \ra$ is equal to $(\wh{j}{a}_w \wh{j}{a}_z, \wh{j}{a}_u)$, and
the image of $a_za_ua_w$ is equal to $(\wh{j}{a}_z \wh{j}{a}_w,
\wh{j}{a}_u)$. In $M_{i_j}/\rho_j \times \la a_u \ra$ we get
$(\wh{j}{a}_w \wh{j}{a}_z, \wh{j}{a}_u) = (\wh{j}{a}_z
\wh{j}{a}_w, \wh{j}{a}_u)$ for $w,z>v+1$, because $\wh{j}{a}_w
\wh{j}{a}_z=\wh{j}{a}_z \wh{j}{a}_w$ in $M_{i_j}/\rho_j$ for
$w,z>v+1$. Thus, also the images of $a_wa_ua_z$ and $a_za_ua_w$
are equal in $M'$.

Hence, all the relations corresponding to adding the arc
$a_{v+1}a_u$ are satisfied. It follows that the congruence
$\rho'_j$, corresponding to $e$ is contained in $\omega$. Since
the diagram $e$ has $i'_j = i_j-1$ unused generators, case (B)
holds.

This completes the proof of the fact that $\rho_j$ corresponds to
a leaf $d$ of $D$. In other words, $\rho_{j}=\rho (d)$. The
ideal $\I_{\rho (d)}$ is prime by Theorem~\ref{liscieidpier}. This
proves the assertion.
\end{proof}

We are now ready for the main result of the paper.

\begin{tw}
  \label{bijekcja}
There exists a bijection between the set of leaves of the tree $D$
and the set of minimal prime ideals of $K[M]$. Namely, if $d$ is a
leaf of $D$ and $\rho (d)$ is the congruence corresponding to $d$,
then $\I_{\rho (d)}$ is the minimal prime ideal assigned to $d$.
\end{tw}

\begin{proof}
Let $P$ be a minimal prime ideal of $K[M]$. By
Theorem~\ref{idpierzawlisc}, $P$ contains a prime ideal  of
the form $\I_{\rho}$, where $\rho=\rho (d)$ is the congruence
corresponding to a leaf $d$ of $D$. Therefore $\I_{\rho}=P$. Let
$f(P) = \rho$.

Let $e$ be a leaf of $D$ and let $\eta = \rho (e)$ be the
corresponding congruence on $M$. Then, by
Theorem~\ref{liscieidpier}, $\I_{\eta}$ is a prime ideal of
$K[M]$. Hence, there exists a minimal prime ideal $Q$ of $K[M]$
contained in $\I_{\eta}$. Then, again by
Theorem~\ref{idpierzawlisc}, $\I_{\eta'}\ssq Q$ for a congruence
$\eta'=\rho (e')$ corresponding to a leaf $e'$ of $D$.
 Then $\I_{\eta'} \ssq Q \ssq \I_{\eta}$, while by Lemma~\ref{osobnegal} we have
$\I_{\eta'} \ssq \I_{\eta}$ if and only if the vertices of $D$
corresponding to congruences $\eta$ and $\eta'$ are in the same
branch of $D$. Since $e,e'$ are leaves, we get $\eta=\eta'$. Then
$\I_{\eta'} = Q = \I_{\eta}$, so $\I_{\eta}$ is a minimal prime
ideal of $K[M]$. We define $g(\eta) = \I_{\eta}$.

Therefore
\[gf(P) = g(\rho) =
\I_{\rho} = P\] and
\[fg(\eta) = f (\I_{\eta}) = f (\I_{\rho (e)}) =\rho (e) =\eta .\]
It follows that $f$ and $g$ establish the desired bijection.
\end{proof}

If $P$ is a minimal prime ideal of $K[M]$ then the congruence
$\{(s,t)\in M\times M : s-t \in P\}$ is denoted by $\rho_P$. This
is a homogeneous congruence, because minimal prime ideals of a
$\Z$-graded ring are homogeneous, see for example \cite{grad}.
Clearly, if $P=\I_{\rho (d)}$ for a leaf $d$ of $D$ then
$\rho_{P}=\rho (d)$.

A careful analysis of the proof of Theorem~\ref{liscieidpier}
leads to the following description of the monoid $M_{P}=M/\rho_P$
for a minimal prime ideal $P=\I_{\rho (d)}$ of $K[M]$. Recall that
$\rho_{P}=\rho_r$ is the last of the congruences in the chain
$\rho_1 \varsubsetneq \rho_2 \varsubsetneq \ldots \varsubsetneq
\rho_r$, constructed for $d$ in \ref{interpr}.

\begin{wn}
\label{zanMr}
  For every minimal prime ideal $P$ of $K[M]$ there exists an embedding
  \[M/\rho_P \hrar \N^{c_P} \times (B \times \Z)^{d_P},\]
  where $c_P+2d_P=n$. Moreover,

($\h$) if  $\rho_1$ is of type $\h$, then
\[M_P \simeq T \times \la \wh{r}{a}_s \ra \simeq T \times \N,\]
where $K[T] \simeq K[M_{n-1}]/Q$ for some minimal prime ideal $Q$
of $K[M_{n-1}]$;

($\diam$) if $\rho_1$ is of type $\diam$, then
\[M_P \ssq M_P (\wh{r}{A}_j)^{-1} \simeq T \times \N^t \times B \times \Z,\]
where $1 \leq j \leq r$ and $K[T] \simeq K[M_{n-2-t}]/Q$ for some
$0 \leq t \leq n-2$ and a minimal prime ideal $Q$ in
$K[M_{n-2-t}]$. For $t=n-2$ we put $K[M_0]=K$, $Q=0$ and
$T=\{1\}$.

\end{wn}

\begin{proof}
  Using the notation of the proof of  Theorem~\ref{liscieidpier}, we know that $\rho_r=\rho_P$ and
  $(M/\rho_r)(\wh{r}{A}_r)^{-1} \simeq \N^* \times (B \times \Z)^*$. Hence there is an embedding
  \[M/\rho_P \hrar \N^{c_P} \times (B \times \Z)^{d_P}\]
for some positive integers $c_P, d_P$. From the algorithm used in
the process of building the latter direct product we know that a
factor $\N$ appears each time a single generator is used (as a
dot), while a factor $B \times \Z$ appears each time a pair of
generators is used (as an arc).  After the extreme arc is added to
a diagram, the submonoid generated by the unused generators is
free abelian. Hence $c_P + 2 d_P = n$.

We keep the notation used in Construction~\ref{interpr} and in
\ref{oznphi}. For $\rho_1$ of type $\h$ we have a commuting
diagram
\[
\xymatrix{ M \ar@{->>}[r]^{\psi_1 = \psi_\h}
\ar@{->>}[rd]_{\psi_r} & M/{\rho_1} \ar@{->>}[d]^{\varphi_i \circ
\ldots \circ \varphi_1}
\ar@{}[r]^{\lambda_1}|{\mbox{\normalsize{$\simeq$}}} &
\ol{M_{n-1}^s} \times \la \wh{r}{a}_s \ra \ar[d]^{\wh{r-1}{\kappa} \circ \ldots \circ \wh{1}{\kappa}} \ar@{}[r]|{\mbox{\normalsize{$\simeq$}}} & \ol{M_{n-1}^s} \times \N  \ar[d]^\mu \\
& M/{\rho_r} \simeq M_P  \ar@{^{(}->}[r]_{\lambda_r} &
M_{i_r}/{\rho_r} \times S_r  \ar@{^{(}->}[r] & (\N^* \times (B
\times \Z)^*) \times \N }
\]
where $\lambda_r$ is as in Lemma~\ref{diagr-przem} and the last
embedding is identity on $S_r$, while $\mu$ is a homomorphism that
makes the diagram commute.

From the construction we know that $\wh{r-1}{\kappa} \circ \ldots
\circ \wh{1}{\kappa}$ is identity on $\la \wh{r}{a}_s \ra \simeq
\N$. Hence, $\mu$ has the form $\theta \times id$, where $\theta$
acts on $\ol{M_{n-1}^s}$, and $id$ acts on $\N$. Let
\[T \nadrow{def} \theta(\ol{M_{n-1}^s}) \ssq \N^* \times (B \times
\Z)^*.\] Then $T$ is a homomorphic image of $M_{n-1}$ and $M_P
\simeq T \times \N$, where $\N$ is an isomorphic image of $\la
\wh{r}{a}_s \ra$.

Denote by $d$ the diagram corresponding to the ideal $P$ (in the
sense of Theorem~\ref{bijekcja}). In the second step of the
construction of $d$, the dot $a_s$, corresponding to $\rho_1$ of
type $\h$, must have been followed by the arc $a_{s+1}a_{s-1}$,
corresponding to $\rho_2$. We know that $M/\rho_1 \simeq
\ol{M_{n-1}^s} \times \N$ and the congruence $\rho_r = \rho_P$
corresponds to the homomorphism $\ol{M_{n-1}^s} \times \N
\nad{\mu}{\na} T \times \N$, so that $M/\rho_P = M_P \simeq T
\times \N$.

We remove the dot $a_s$ from the diagram $d$. Then we get a
diagram $d'$ in the tree built for the Chinese monoid on $n-1$
generators. Such $d'$ corresponds to a leaf of this new tree,
whence to a minimal prime ideal of $K[M_{n-1}]$. On the other
hand, $d'$ corresponds to the kernel of the homomorphism $M_{n-1}
\rar \ol{M_{n-1}^s} \rar T$, which is a consequence of the
construction of $\rho_1, \rho_2, \ldots, \rho_r$. So $T$ is a
homomorphic image of $M_{n-1}^s$. Let $Q$ be the kernel of the
epimorphism $K[M_{n-1}^s] \na K[T]$. Then $K[T] \simeq
K[M_{n-1}^s]/Q$. Since $d'$ corresponds to a minimal prime ideal
of $K[M_{n-1}]$, $Q$ is a minimal prime ideal. This completes the
proof in case $\rho_1$ is of type $\h$.

Assume now that $\rho_1$ is of type $\diam$. We consider two
cases.

(a) $r=1$, so that $\rho_1 = \rho_P$ corresponds to a diagram
\[
\diag{ \bul \ar@/^0.4pc/@{-}[r] & \bul & \circ & \ldots & \circ} \
\ \ \ \ \textup{or} \ \ \ \ \ \diag{\circ & \ldots & \circ & \bul
\ar@/^0.4pc/@{-}[r] & \bul} \] Then $\ol{M_{n-2}} \simeq
\N^{n-2}$, so Lemma~\ref{Mkaro} yields
\[M_P = M/\rho_P \hrar \ol{M_{n-2}} \times B \times \Z \simeq \N^{n-2} \times B \times \Z\]
and the assertion follows with $t=n-2$, $K[M_0]=K$, $Q=0$ and
$T=\{1\}$.

(b) $r >1$, so in the construction of the diagram $d$
corresponding to the ideal $P$, after an initial arc corresponding
to the congruence $\rho_1$, there were more steps leading to the
leaf $d$ of $D$. Recall that such a construction must finish with
an extreme arc (see Definition \ref{defl}). Hence, in $d$, after
the initial arc $a_sa_{s-1}$, a number $t \geq 0$ of dots have
been built, followed by another arc. Hence, for some $1< j +1 \leq
r$ the congruence $\rho_{j+1}$ corresponds (for some $t \geq 0$)
to the diagram
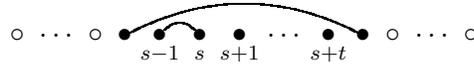
\[ \diag{
 \circ & \ldots & \circ & \bul \ar@/^1pc/@{-}[rrrrrr]
 & \podpis{\bul}{s-1} \ar@/^0.4pc/@{-}[r] & \podpis{\bul}{s} &
\podpis{\bul}{s+1} & \ldots & \podpis{\bul}{s+t} & \bul & \circ &
\ldots & \circ } \] or to an analogous diagram with $t$ dots on
the left of the arc $a_sa_{s-1}$. Then $\rho_j$ corresponds to the
diagram
\[ \diag{
 \circ & \ldots & \circ
 & \podpis{\bul}{s-1} \ar@/^0.4pc/@{-}[r] & \podpis{\bul}{s} &
\podpis{\bul}{s+1} & \ldots & \podpis{\bul}{s+t} & \circ & \ldots
& \circ } \] or to the analogous diagram with $t$ dots on the left
of the arc $a_sa_{s-1}$. Then the number of unused generators is
equal to $i_j = n-2-t$ and $S_{j}= \N^t \times B \times \Z$, while
Lemma~\ref{diagr-przem} yields a natural embedding
\[M/\rho_{j} \hrar M_{i_{j}}/\rho_{j} \times S_{j} = M_{i_{j}}/\rho_{j} \times \N^t \times B \times \Z.\]
Moreover, $S_r = Y \times S_{j}= Y \times \N^t \times (B \times
\Z)$, where $Y = \N^* \times (B \times \Z)^*$  and the
construction of $M/\rho_P$ yields natural homomorphisms
\[\hspace{-1cm}
\xymatrix{ M/{\rho_{j}} \ar@{->>}[d]_{\varphi_{r-1} \circ \ldots
\circ \varphi_{j}} \ar@{^{(}->}[r] & M_{i_{j}}/{\rho_{j}} \times
S_{j} \ar[d]^{\wh{r-1}{\kappa} \circ \ldots \circ \wh{j}{\kappa}}
\ar@{}[r]|(0.4){\mbox{\normalsize{$=$}}}
& M_{i_{j}}/{\rho_{j}} \times \N^t \times B \times \Z \ar[d] \\
M/{\rho_r} \simeq M_P  \ar@{^{(}->}[r] & M_{i_r}/{\rho_r} \times Y
\times S_j  \ar@{}[r]|(0.4){\mbox{\normalsize{$=$}}}
 & \hspace{-0.3cm} M_{i_r}/{\rho_r} \times Y \times \N^t \times B \times \Z
}\] where $\wh{r-1}{\kappa} \circ \ldots \circ \wh{j}{\kappa}$ is
identity on $S_j$, so it is of the form $\theta \times id$, with
$\theta \colon M_{i_{j}}/{\rho_{j}} \rar M_{i_r}/\rho_r \times Y$
and
 $id \colon S_j \rar S_j$. Let \[T \nadrow{def}
\theta(M_{i_{j}}/{\rho_{j}}),\] so $T \times S_j$ is the image of
$M_{i_{j}}/{\rho_{j}} \times S_{j}$ under $\wh{r-1}{\kappa} \circ
\ldots \circ \wh{j}{\kappa} = \theta \times id$.

By Proposition~\ref{centrlok}, $M_{i_j} / \rho_j \times S_j =
(M/\rho_j) (\wh{j}{A}_j)^{-1}$ (under an appropriate
identification). Consider the following diagram, similar to (\#\#)
used in Proposition~\ref{centrlok}:
\[ \xymatrix{
 & (M/\rho_j)(\wh{j}{A}_j)^{-1} \ar@{}|{\mbox{\normalsize{$=$}}}[r] \ar@{->>}_{\varphi'_{r-1} \circ \ldots \circ \varphi'_j}[d] &  M_{i_j} / \rho_j \times S_j \ar^{\wh{r-1}{\kappa} \circ \ldots \circ \wh{j}{\kappa}= \theta \times id}[d]\\
M_P(\wh{r}{A}_{j})^{-1} \ = \hspace{-4ex}&
(M/\rho_{r})(\wh{r}{A}_{j})^{-1} \ar@{^{(}->}^{\lambda''_{r}}[r] &
M_{i_{r}} / \rho_{r} \times Y \times S_j
\ar@{}|(0.55){\mbox{\normalsize{$=$}}}[r] & M_{i_r}/\rho_r \times
S_r } \tag{\#\#\#}\] where $\lambda''_{r}$ is the restriction of
$\lambda'_{r}$ to $M_P(\wh{r}{A}_{j})^{-1}$, and every
$\varphi'_k$, for  $k=j, \ldots, r-1$, is the natural extension of
$\varphi_k$ to the appropriate localization. Then $\varphi'_{r-1}
\circ \ldots \circ \varphi'_j$ maps $M/\rho_j$ onto $M/\rho_r$,
while $\wh{j}{A}_j$ is mapped onto $\wh{r}{A}_j$. Thus, this is an
epimorphism onto $(M/\rho_r)(\wh{r}{A}_j)^{-1}$.

We know that $T \times S_j \ssq M_{i_{r}} / \rho_{r} \times S_{r}$
is the image of $(M/\rho_j)(\wh{j}{A}_j)^{-1}  =
M_{i_{j}}/{\rho_{j}} \times S_{j}$ under $\wh{r-1}{\kappa} \circ
\ldots \circ \wh{j}{\kappa} = \theta \times id$. Since diagram
(\#\#\#) commutes, this image must be equal to
\[\lambda''_{r} \circ \varphi'_{r-1} \circ \ldots \circ \varphi'_j ((M/\rho_j)(\wh{j}{A}_j)^{-1}) =
\lambda''_{r} (M_P(\wh{r}{A}_{j})^{-1}) \simeq
M_P(\wh{r}{A}_{j})^{-1} \ssq M_{i_r}/\rho_r \times S_r.\]
Therefore
\[M_P \ssq M_P (\wh{r}{A}_j)^{-1} \simeq T \times S_j = T \times \N^t \times B \times \Z,\]
which proves the first part of the assertion in case $\rho_{1}$ is
of type $\diam$.

Removing from $d$ the dots $s-1, s, s+1, \ldots, s+t$ leads to a
diagram $d'$ in the tree constructed for the Chinese monoid
$M_{n-2-t}$. This diagram $d'$ starts with an arc
$a_{s+t+1}a_{s-2}$. Hence, as in the last part of the above proof
in case $\h$, the diagram $d'$ corresponds to the kernel of the
homomorphism $M_{n-2-t} \rar M_{n-2-t}/\rho_j \rar T$ and we get
$K[T] \simeq K[M_{n-2-t}]/Q$ for a minimal prime ideal $Q$ in
$M_{n-2-t}$. This completes the proof in case $\diam$, and hence
the proof of the proposition.
\end{proof}

\section{Applications}
\label{r-zast}

Our final goal is to derive certain important consequences of the
main result of Section~\ref{r-min-id-pier}. First, in Part
\ref{s-BJ}, we show that the prime radical of the Chinese algebra
$K[M]$ coincides with its Jacobson radical. Next, in Part
\ref{s-liczba}, we obtain a formula for the number of minimal
primes of $K[M]$. A surprising new representation of the monoid
$M$ as a submonoid of the direct product $B^d \times \Z^e$ for
some $d,e \geq 1$ is found in Part \ref{s-zanurz}. In particular,
the latter implies that $M$ satisfies a nontrivial identity.

\subsection{The prime radical and the Jacobson radical of $K[M]$ coincide}
\label{s-BJ}

\

Recall that $J(R),B(R)$ denote the Jacobson and the prime
radical of a ring $R$, respectively. We start with the following
result.

\begin{tw}
\label{B=Jwstep} If $P$ is a minimal prime ideal of the Chinese
algebra $K[M]$ then the algebra $K[M]/P$ is semiprimitive.
\end{tw}

\begin{proof}
Let $n$ be the rank of $M$. If $n=1$ then $K[M] = K[x]$. If $n=2$
then from \cite{jofc} we know that $K[M]$ is also prime and
semiprimitive. Thus, we may assume that $n\geq 3$. By induction,
we may also assume that the assertion is satisfied for all Chinese
algebras of rank less than $n$. We shall consider the two cases,
denoted by $\h$ and $\diam$, as in Corollary~\ref{zanMr}.

First, consider case $\h$. From Corollary \ref{zanMr} we know that
$K[M]/P\simeq K[M_P] \simeq K[T][x]$, where $K[T]$ is an algebra
of the form $K[M_{n-1}]/Q$ for some minimal prime ideal $Q \lhd
K[M_{n-1}]$. By the inductive hypothesis, we get
$J(K[M_{n-1}]/Q)=0$. Since $K[M_{n-1}]/Q \simeq K[T]$, this
implies that $J(K[M]/P)\simeq J(K[T][x])=0$, as desired.

Next, consider case $\diam$. Suppose that $J(K[M_P]) \neq 0$ and
choose some nonzero $a \in J(K[M_P])$. From Corollary \ref{zanMr}
we know that $M_P \hrar T \times \N^t \times B \times \Z$ for an
appropriate $T$. Hence $K[M_P]$ can be viewed as a $\Z$-graded
algebra (according to the last component of the above direct
product) or as an $\N$-graded algebra (for each of the $t$
components $\N$). Therefore, from Theorem~30.28 in \cite{Karp} we
know that $J(K[M_P])$ is homogeneous. Thus we may assume that $a$
is homogeneous with respect to each of the gradations coming from
components $\Z$ or $\N$. Let $a = \sum_{i=1}^k \lambda_i s_i$ for
some $k \geq 1$, $0 \neq \lambda_i \in K$, $s_i \in M_P$. Then all
$s_i$ coincide when restricted to each of these components. This
means that there exist elements  $m \in \N^t$, $z \in \Z$
(independent of $i$) such that $s_i = (t_i, m, b_i, z) \in T
\times \N^t \times B \times \Z$. Since $a \neq 0$, also
$\sum_{i=1}^k \lambda_i(t_i, b_i) \neq 0$.

Consider the natural projection
\[\Pi \colon T \times \N^t \times B \times \Z \rar T \times B\] and the induced
map of semigroup algebras. Clearly  $\Pi(a) = \sum_{i=1}^k
\lambda_i(t_i, b_i)\neq 0$.

We know that  $\Pi(M_P) \ssq \Pi(T \times \N^t \times B \times \Z)
= T \times B$. We will show that the opposite inclusion $\Pi(M_P)
\supseteq T \times B$ also holds.

The monoid $\Pi(M_P)$ contains $(1,p)$ and $(1,q)$, because under
the homomorphism $\psi_r \colon M \rar M_P \ssq T \times \N^t
\times B \times \Z$ we have $a_{s-1} \mapsto (1, 1, p,g)$, $a_s
\mapsto (1, 1, q,1)$. Therefore, for every $b \in B$ we have
$(1,b) \in \Pi(M_P)$.

From the proof of Corollary~\ref{zanMr} and from the commuting
diagrams used in this proof (in case $\diam$) it follows that the
following diagram commutes:
\[ \xymatrix{
M/\rho_j \ar@{->>}[d]_{\varphi_{r-1} \circ \ldots \circ
\varphi_{j}} \ar@{^{(}->}[rr]  &
&  M_{i_j} / \rho_j \times S_j \ar@{->>}^{\theta \times id}[d]\\
M_P \ar@{^{(}->}[r] & M_P(\wh{r}{A}_{j})^{-1}
\ar@{}|{\mbox{\normalsize{$\simeq$}}}[r]  & T \times S_j
\ar@{}|(0.3){\mbox{\normalsize{$=$}}}[r] & \hspace{-4ex} T \times
\N^t \times B \times \Z \ar@{->>}^(0.6){\Pi}[r] & T \times B } \]
The embedding in the first row, composed with the projection
$M_{i_j} / \rho_j \times S_j \rar M_{i_j}/\rho_j$, maps
$M/\rho_{j}$ onto $M_{i_j}/\rho_j$. By the definition
$T=\theta(M_{i_j} / \rho_j)$ and $\Pi$ is a projection, whence the
homomorphism
\[M/\rho_j \hrar M_{i_j}/\rho_j \times S_j \nad{\theta \times id}{\na} T
\times S_j \nad{\Pi}{\na} T \times B,\] composed with the
projection onto $T$, is a map onto $T$. Commutativity of the above
diagram implies now that also the  homomorphism in the second row
\[M_P \hrar T \times S_j \nad{\Pi}{\na} T \times B,\] composed with $T
\times B \rar T$, is a map onto $T$.

It follows that the image of $\Pi(M_P)$ under $T \times B \rar T$
coincides with $T$. Hence, for every $t \in T$ there exists $b
=p^iq^j \in B$ such that $(t,b) \in \Pi(M_P)$. Multiplying by
$(1,q^i) \in \Pi(M_P)$ on the left and by $(1,p^j)\in \Pi(M_P)$ on
the right, we get $(t,1) \in \Pi(M_P)$. This and the fact that
$(1,b) \in \Pi(M_P)$ for every $b \in B$ imply that $\Pi(M_P)
\supseteq T \times B$, as desired.

Therefore, $\Pi(M_P) = T \times B$, so that $\Pi|_{M_P}$ is
surjective and so its natural extension to $K[M_P]$ is also
surjective. Therefore we get $\Pi(J(K[M_P])) \ssq J(K[T \times
B])$. Since $0 \neq a \in J(K[M_P])$ and $\Pi(a) \neq 0$, this
implies that
\begin{equation}  \label{odot}
0 \neq \Pi(a) \in \Pi(J(K[M_P])) \ssq J(K[T \times B]).
\end{equation}
Moreover, $K[T \times B] \simeq K[T][B]$ and from \cite{jofc} we
know that $K[T][B]$ contains an ideal $\I \simeq
\mathcal{M}_\infty(K[T])$ such that $K[T][B]/\I \simeq
K[T][x,x^{-1}]$. Here $\mathcal{M}_\infty(K[T])$ stands for the
algebra of $\N\times \N$ matrices over $K[T]$ with finitely many
nonzero entries.

As in case $\h$, from the inductive hypothesis it follows in view
of Corollary \ref{zanMr} that $J(K[T])=0$. Hence, the above
implies that $J(K[T][B]/\I) \simeq J(K[T][x,x^{-1}])=0$. Moreover,
$J(K[T])=0$ yields
\[J(\I) \simeq J(\mathcal{M}_\infty(K[T])) \simeq \mathcal{M}_\infty(J(K[T])) =
 0.\]
Since $J(\I) = 0$ and $J(K[T][B]/\I)=0$, it follows that $ J(K[T
\times B]) \simeq J(K[T][B])=0$. This contradicts (\ref{odot}),
completing the proof in case $\diam$.
\end{proof}

As a direct consequence we get
\begin{wn}
\label{B=J} The prime radical of the Chinese algebra $K[M]$
is equal to its Jacobson radical.
\end{wn}

Notice that the properties of the algebra $K[M_3]$ are different
than those of the plactic algebra of rank $3$, which is not prime
but is semiprimitive, see \cite{jofc}. Namely, if $n \geq 3$ then
the Chinese algebra $K[M]$ of rank $n$ is not semiprime,
\cite{praca}.

\subsection{Number of minimal prime ideals of $K[M]$}
\label{s-liczba}

\

In order to get a formula for the number of minimal primes of
$K[M]$ we use the construction of the tree $D$ and the bijection
between the leaves of $D$ and the minimal primes in $K[M]$,
established in Theorem~\ref{bijekcja}.

The following analogue of the Fibonacci sequence will be crucial.

\begin{df}
\label{trib} \emph{The Tribonacci sequence} is the sequence
defined by the linear recurrence
\[\begin{cases}
T_0=T_1=T_2=1 \\
T_{n+1}=T_n+T_{n-1}+T_{n-2}   \mbox{ for } n\geq 2.
\end{cases}\]
\end{df}
The properties of this sequence are described in
\cite[A000213]{sloane}. Its initial elements are: $T_0=T_1=T_2=1$,
$T_3=3$, $T_4=5$, $T_5=9$, $T_6=17$, $T_7=31$, $T_8=57$,
$T_9=105$, $T_{10}=193$.

\begin{tw}
\label{liczbaid} Let $M$ be the Chinese monoid of rank $n$. Then
$T_{n}$ is the number of minimal prime ideals of the algebra
$K[M]$.
\end{tw}

\begin{proof}
Recall that, if the rank $n$ of the Chinese monoid $M$ is $1$ or
$2$, then the algebra $K[M]$ is prime. Hence, we may assume that
$n\geq 3$.

By Theorem~\ref{bijekcja}, it is enough to enumerate the leaves of
the tree $D$. From the construction of $D$ in
Definition~\ref{drzewo} we also know that a diagram $f$ is a leaf
of $D$ if and only if the last step in the construction of $f$ is
an arc containing one of the generators $a_1,a_n$, in other words
an extreme arc. Hence, we will count the number of such diagrams.

Let $k$ be the number of generators used in the construction of
$f$ before constructing the respective extreme arc (that is, the
number of generators under this arc). Let $U_k$ denote the number
of all possible configurations of $k$ generators under an arc in a
diagram. For $k=0$ we put $U_0 = 1$. If $k=1$, then $U_1 = 1$,
because the only possibility is a single dot under the arc. If
$k=2$, clearly there is also a single possibility, so that
$U_2=1$. For $k=3$ there are $3$ possibilities. For example, if
$a_1a_5$ is the given arc, then under this arc we can have: either
the dot $a_3$ and the arc $a_2a_4$, or the arc $a_3a_4$ and the
dot $a_2$, or the arc $a_2a_3$ and the dot $a_4$. Hence $U_3=3$.
Similarly, one can easily see that $U_4=5$.

In general, if $k\geq 3$ then there are two types of
configurations of exactly $k$ generators under an arc $A$. The
first type occurs when there is another arc $A'$ directly under
$A$. Then there are $k-2$ generators under $A'$, so the number of
such configurations is the same as for $k-2$, that is $U_{k-2}$.
The second type occurs when directly under $A$ there is a number
$i>0$ of consecutive dots (on one of the sides, right or left) and
another arc covering all other generators. In this case, the
interior arc covers $k-2-i$ generators, and the number of such
configurations is twice the number of configurations for $k-2-i$,
so $2U_{k-2-i}$.

The above implies that $U_0 = U_1 = U_2 = 1$ and $U_k = U_{k-2} +
2 \cdot \sum_{i=1}^{k-2} U_{k-2-i}$ for $k \geq 3$. Notice that
$\sum_{i=1}^{k-2} U_{k-2-i} = \sum_{i=0}^{k-3} U_i$, so that
\[U_k = U_{k-2} + 2 \cdot \sum_{i=0}^{k-3} U_i.\]
Therefore
\[U_{k+1} = U_{k-1} + 2 \cdot \sum_{i=0}^{k-2} U_i = U_{k-1} + 2 \cdot
\sum_{i=0}^{k-3} U_i + 2 U_{k-2}\] and subtracting one of these
equalities from the other one we get $U_{k+1} - U_k = U_{k-1} +
U_{k-2}$. So, for $k \geq 3$,
\[U_{k+1}=U_k+U_{k-1}+U_{k-2}.\]
Let $T'_n$ denote the number of all minimal prime ideals of
$K[M]$. Then we may assume $T'_0=T'_1= T'_2=1$ and from
Example~\ref{przykD} we know that $T'_3=3$ and $T'_4=5$. Recall
that $n \geq 3$. If the last step in the construction of a leaf of
$D$ is the arc $a_{1}a_{n}$ then there are $n-2$ generators under
this arc, hence there are $U_{n-2}$ leaves of this type. On the
other hand, if the extreme arc used in the construction of a leaf
contains only one of the generators $a_1, a_n$, then there are $k
\leq n-3$ generators under it, so the number of such leaves is
$2U_k$. Therefore, for $n \geq 3$ we get $T'_n = U_{n-2} + 2 \cdot
\sum_{k=0}^{n-3} U_k$. Notice that $T'_n=U_n$. The number of
minimal prime ideals of $K[M]$ is therefore given be the linear
recurrence
\[\begin{cases}
T'_0=T'_1=T'_2=1 \\
T'_{n+1}=T'_n+T'_{n-1}+T'_{n-2}.
\end{cases}\]
The assertion follows.
\end{proof}

\subsection{An embedding $M \hookrightarrow \N^c \times (B \times \Z)^d$}
\label{s-zanurz}

\

The construction of the monoids $M/\rho_P$, for all minimal prime
ideals $P$ of $K[M]$ and the associated congruences $\rho_P$,
allows us to find an entirely new faithful representation of $M$
as a submonoid of the direct product $\N^c \times (B \times
\Z)^d$, with $c+2d=nT_n$, where $T_n$ is the $n$-th element of the
Tribonacci sequence.

Let $\mc{P}_k$ be the set of all minimal prime ideals of the
Chinese algebra $K[M_k]$, for any $1 \leq k \leq n$. If $k=n$, we
will simply write $\mc{P}=\mc{P}_k$. By Theorem~\ref{liczbaid}, we
know that $|\mc{P}|=T_n$.

\begin{lem}
\label{kongtryw}
  $\bigcap_{P \in \mc{P}} \rho_P = \rho_0$,
  where $\rho_0$ stands for the trivial congruence on $M$.
\end{lem}

\begin{proof}
If $n=1$ then $K[M_1]=K[x]$, while for $n=2$ the algebra $K[M_2]$
is also prime by \cite{jofc}. Hence, we may assume that $n\geq 3$.

If $n=3$ then there are $3$ minimal primes in $K[M]$, say $P_1$,
$P_2$ and $P_3$, see Example~\ref{przykD} and
Theorem~\ref{bijekcja}, or \cite{jofc}. We prove that if two
elements $w, v \in M$ are such that $(w,v) \in \rho_{P_i}$ for
$i=1,2,3$, then $w=v$. Let $a=a_{1},b=a_{2},c=a_{3}$. Let $w=
(a)^{\alpha_{a}} (ba)^{\alpha_{ba}} (b)^{\alpha_{b}}
(ca)^{\alpha_{ca}} (cb)^{\alpha_{cb}} (c)^{\alpha_{c}}$ and $v=
(a)^{\beta_a} (ba)^{\beta_{ba}} (b)^{\beta_b} (ca)^{\beta_{ca}}
(cb)^{\beta_{cb}} (c)^{\beta_{c}}$ be the canonical forms of
$w,v$, respectively.

For simplicity, we write $(x), (xy)$ for any non-negative powers
of $x$ and $xy$, if $x,y\in \{a,b,c\}$. Let $\wt{u}$ denote the
image of $u \in M$ in $M/{\rho_{P_i}}$, for a fixed $i$.

We know that $\rho_{P_1}$ corresponds to imposing on $M$ the
additional relations $ab=ba$ and $acb=bca$. By the proof of
Lemma~\ref{asas-1reg}, the canonical form of the element $\wt{w}
\in M/{\rho_{P_1}}$ is
$(\wt{a})(\wt{b})(\wt{c}\wt{a})(\wt{c}\wt{b})(\wt{c})$, where the
exponent of $(\wt{b})$ or of $(\wt{c}\wt{a})$ is equal to $0$.
Clearly, all the  exponents are determined by those in the element
$w$, in particular the exponent of $(\wt{c})$ is equal to the
exponent of $(c)$ in $w$. Since $\wt{w} = \wt{v}$, all exponents
in the canonical forms of these two elements of $M/{\rho_{P_1}}$
are equal, so  in particular we get $\alpha_{c}=\beta_c$.

Similarly, the congruence $\rho_{P_2}$ corresponds to imposing
relations $bc=cb$ and $bac=cab$ on $M$. So
$(\wt{a})(\wt{b})(\wt{b}\wt{a})(\wt{c})(\wt{c}\wt{a})$ is the
canonical form of elements of $M/{\rho_{P_2}}$, with the exponent
of $(\wt{b}\wt{a})$ or of $(\wt{c})$ equal $0$ and the exponent of
$(\wt{a})$ equal to the exponent of $(a)$ in the original element
of $M$. This and the equality  $\wt{w} = \wt{v}$ imply that
$\alpha_{a}=\beta_a$.

The congruence $\rho_{P_3}$ corresponds to the relations $ab=ba$
and $bc=cb$ and it leads to the canonical form
$(\wt{a})(\wt{b})(\wt{c})(\wt{c}\wt{a})$ in $M/{\rho_{P_3}}$. For
$\wt{w}=\wt{v}$ this yields equalities of the corresponding
exponents of $(\wt{a}), (\wt{b}), (\wt{c})$ and $(\wt{c}\wt{a})$:
\[
\left\{
\begin{array}{r@{\;}ll}
  \alpha_{a}+\alpha_{ba} &= \beta_a+\beta_{ba} \\
  \alpha_{b}+\alpha_{ba}+\alpha_{cb} &= \beta_b+\beta_{ba}+\beta_{cb} \\
  \alpha_{c}+\alpha_{cb} &= \beta_c+\beta_{cb}\\
  \alpha_{ca} &= \beta_{ca}.
\end{array}
\right.
\]
These equalities, together with the earlier ones:
$\alpha_{c}=\beta_c$ and $\alpha_a=\beta_a$ easily imply that
every exponent in the canonical form of $w$ is equal to the
corresponding exponent in the form of $v$. Hence $w=v$, which
finishes the proof in case $n=3$.

Let $n \geq 4$. Proceeding by induction we assume that the
assertion is true for the monoid $M_{n-1}$. Let $w, v \in M$ and
let $\widetilde{w}, \widetilde{v}$ be their images under some
fixed epimorphism $M \na M_{n-1}$.

If $(w,v) \in  \bigcap_{P \in \mc{P}} \rho_P$, then $w-v \in
\bigcap_{P \in \mc{P}} P = B(K[M])$. Hence $\widetilde{w} -
\widetilde{v} \in \bigcap_{P \in \mc{P}_{n-1}} P = B(K[M_{n-1}])$.
This means that for every $P \in \mc{P}_{n-1}$ one has
$\widetilde{w} - \widetilde{v} \in P$, so that $(\widetilde{w},
\widetilde{v}) \in \rho_P$. By the induction hypothesis the latter
implies that $\widetilde{w} = \widetilde{v}$.

Using the canonical forms of elements of $M_{n-1}$, as in the case
$n=3$, from such equalities we get equalities of the corresponding
exponents. For simplicity, the $k$-th generator of $M$ and its
image will be denoted by $k$, for $k=1,2,\ldots,n$.

Consider the maps $f_{k}$, for $k=1,\ldots, n-1$, defined on the
generators of $M$ by:
\[
\begin{array}{lll}
k,k+1 \mapsto k  & \textup{ and } & i \mapsto i \textup{ for } i
\neq k, k+1.
\end{array}
\]
It is easy to see that every such map transforms the defining
relations of $M$ into the relations defining the Chinese monoid of
rank $n-1$ with generators $1,\ldots, k,k+2,\ldots ,n$. Hence,
every $f_{k}$ defines a surjective homomorphism $M \rar M_{n-1}$.
Notice that for every $w \in M$ and every fixed $k$, the image
$f_{k}(w)$ is of the form
\begin{align*}
\wt{w} = & (1)(21)(2)\ldots(k-1)\\
&(k1)(k2)\ldots(k)\\
&(k1)(k2)\ldots(kk)(k)\\
&(k+2 \ 1)\ldots(k+2 \ k)(k+2 \ k)\\
&\ldots \\
&(n1)(n2)\ldots(n \ k-1)(nk)(nk)(n\  k+2) \ldots (n).
\end{align*}
Since $(ki)$ and $(kj)$ commute for $i,j \leq k$, the latter leads
to
\begin{align*}
\wt{w} = & (1)(21)(2)\ldots(k-1)\\
&(k1)(k2)\ldots(k\ k-1)(k)\\
&(k+2 \ 1)\ldots(k+2 \ k)\\
&\ldots \\
&(n1)(n2)\ldots(n \ k-1)(nk)(n\  k+2) \ldots (n),
\end{align*}
with the (non-indicated) exponents depending on the exponents in
the canonical form of $w$. Moreover, the above is the canonical
form of $\wt{w}$ in the corresponding Chinese monoid of rank $n-1$
(see (\ref{canon})). Hence, from $\wt{w} = \wt{v}$ we derive the
following system of equalities
\[ \left\{
\begin{array}{r@{\;}ll}
  \alpha_{ij} &= \beta_{ij} &\textup{for $i<k$ and~every $j$}\\
  \alpha_{kj}+\alpha_{k+1 \ j} &= \beta_{kj}+\beta_{k+1 \ j} &\textup{for $j<k$}\\
  \alpha_{k}+2\alpha_{k+1\  k}+\alpha_{k+1} &= \beta_k+2\beta_{k+1\  k}+\beta_{k+1} \\
  \alpha_{ij} &= \beta_{ij} &\textup{for $i>k$ and $j \neq k,k+1$}\\
  \alpha_{ik} &= \beta_{ik} &\textup{for every $i$},
\end{array}
\right. \] where $\alpha$'s, $\beta$'s are the exponents in the
canonical form of $w,v$, respectively, and with the convention
that $\alpha_{k}=\alpha_{kk}$ and $\beta_k=\beta_{kk}$. The
homomorphism of the above type for $k=n-1$, with $n \geq 4$, leads
in particular to the following equalities
\[ \left\{
\begin{array}{r@{\;}ll}
\alpha_{ij} &= \beta_{ij} &\textup{for $i<n-1$ and every $j$}\\
\alpha_{n-1 \ 1}+\alpha_{n1} &= \beta_{n-1 \ 1}+\beta_{n1}\\
\alpha_{n-1 \ 2}+\alpha_{n2} &= \beta_{n-1 \ 2}+\beta_{n2}.
\end{array}
\right. \] On the other hand, for $k=1$ we get in particular
\[\left\{
\begin{array}{r@{\;}ll}
  \alpha_{n-1 \ j} &= \beta_{n-1\ j} &\textup{for $j \neq 1,2$}\\
  \alpha_{nj} &= \beta_{nj} &\textup{for $j \neq 1,2$},\\
\end{array}
\right. \] while for $k=3$, with $n \geq 4$, we get
\[\begin{cases}
\alpha_{n1} =  \beta_{n1}\\
\alpha_{n2} =  \beta_{n2}.
\end{cases}\]
It is easy to see that the above three systems of equalities lead
to the conclusion that all exponents in the canonical form of $w$
are equal to the corresponding exponents in $v$. Hence $w=v$,
which completes the proof.
\end{proof}

\begin{tw}
\label{zanurz-pi} There exists an embedding $M \hookrightarrow
\prod_{P \in \mc{P}} M/\rho_P$, where $\mc{P}$ denotes the
set of minimal prime ideals of the Chinese algebra $K[M]$.
\end{tw}

\begin{proof}
Let $m\in M$. Let $m_P$ denote the image of $m$ in $M/\rho_P$, for
$P\in \mc{P}$.  Then $m \mapsto (m_P)_{P\in \mc{P}}$ determines a
homomorphism, which is injective by Lemma~\ref{kongtryw}.
\end{proof}

\begin{wn}
\label{zanurz} There exists an embedding $M \hookrightarrow \N^c
\times (B \times \Z)^d$, where $c+2d=nT_n$.%
\end{wn}

\begin{proof}
From Corollary~\ref{zanMr} we know that for every $P \in \mc{P}$
there is an embedding $M/\rho_P \hookrightarrow \N^{c_P} \times (B
\times \Z)^{d_P}$ such that $c_P+2d_P=n$. In view of
Theorem~\ref{zanurz-pi} this yields an embedding
\[M \hookrightarrow \N^{c} \times (B \times \Z)^{d},\]
with $c+2d=n \cdot |\mc{P}| = nT_n$.
\end{proof}

It is well known that the bicyclic monoid $B$ satisfies the
identity $xy^2x xy xy^2x = xy^2x yx xy^2x$, \cite{adjan}. The
following surprising result is an immediate consequence.

\begin{wn}
\label{tozs}
  The Chinese monoid $M$ satisfies the identity
  \[xy^2x xy xy^2x = xy^2x yx xy^2x.\]
\end{wn}

\begin{wn}
\label{BnieIrho} The prime radical $B(K[M])$ is not of the form
$\I_\rho$ for any congruence $\rho$ on $M$.
\end{wn}

\begin{proof}
Suppose that $B(K[M]) = \I_\rho$ for a congruence $\rho$ on $M$.
Then $\rho=\rho_{B(K[M])} \ssq \rho_P$ for every prime ideal $P$
of $K[M]$.  Thus, $\rho \ssq \bigcap_{P \in \mc{P}} \rho_P$. From
Lemma~\ref{kongtryw} we know that $\bigcap_{P \in \mc{P}} \rho_P =
\rho_0$, where $\rho_0$ is the trivial congruence. Hence
$\rho=\rho_0$ and $B(K[M]) = \I_{\rho_0} = 0$. As recalled after
Corollary~\ref{B=J}, this contradicts \cite{praca}. The assertion
follows.
\end{proof}


\begin{thebibliography}{10}

\bibitem{adjan}
S.~I. Adjan, \emph{Defining relations and algorithmic problems for
groups and semigroups}, Trudy Matematicheskogo Instituta imeni V.
A. Steklova \textbf{85} (1966), 1--124.

\bibitem{cass}
J.~Cassaigne, M.~Espie, D.~Krob, J.-C. Novelli, and F.~Hivert,
\emph{The {C}hinese monoid}, International Journal of Algebra and
Computation \textbf{11} (2001), no.~3, 301--334.

\bibitem{jofc}
F.~Ced\'o and J.~Okni\'nski, \emph{Plactic algebras}, Journal of
Algebra \textbf{274} (2004), no.~1, 97--117.

\bibitem{CP}
A.~H. Clifford and G.~B. Preston, \emph{The Algebraic Theory of
Semigroups}, vol.~1, American Mathematical Society, Providence,
Rhode Island,  1961.

\bibitem{duchamp}
G.~Duchamp and D.~Krob, \emph{Plactic-growth like monoids}, Words,
Languages and Combinatorics, {II} ({K}yoto, 1992), World
Scientific, Singapore, 1994, pp.~124--142.

\bibitem{fulton}
W.~Fulton, \emph{Young Tableaux}, Cambridge University Press, New
York, 1997.

\bibitem{gateva}
T.~Gateva-Ivanova, \emph{A combinatorial approach to set-theoretic
solutions of the {Y}ang-{B}axter equation}, Journal of
Mathematical Physics \textbf{45} (2004), no.~10, 3828--3858.

\bibitem{praca}
J.~Jaszu\'nska and J.~Okni\'nski, \emph{{Chinese algebras of
rank}~3}, Communications in Algebra \textbf{34} (2006), no.~8,
2745--2754.


\bibitem{grad}
E.~Jespers, J.~Krempa, and E.~Puczylowski, \emph{On radicals of
graded rings}, Communications in Algebra \textbf{10} (1982),
no.~17, 1849--1854.

\bibitem{binom}
E.~Jespers and J.~Okni\'{n}ski, \emph{Binomial semigroups},
Journal of Algebra \textbf{202} (1998), no.~1, 250--275.

\bibitem{f}
E.~Jespers and J.~Okni\'nski, \emph{{Noetherian Semigroup
Algebras}}, Springer-Verlag, Dordrecht, 2007.

\bibitem{Karp}
G.~Karpilovsky, \emph{The Jacobson Radical of Classical Rings},
Longman, New York, 1991.

\bibitem{KL}
G.~R. Krause and T.~H. Lenagan, \emph{Growth of Algebras and
Gelfand-Kirillov Dimension, revised ed.}, Graduate Studies in
Mathematics, vol.~22, American Mathematical Society, Providence,
Rhode Island, 2000.

\bibitem{Krempa}
J.~Krempa, \emph{On semisimplicity of tensor products}, Ring
theory, Lecture Notes in Pure and Applied Mathematics, vol.~51,
Marcel Dekker, New York, 1979, pp.~105--122.

\bibitem{TYLam}
T.~Y. Lam, \emph{A~First Course in Noncommutative Rings}, second
ed., Graduate Texts in Mathematics, vol. 131, Springer-Verlag, New
York, 2001.

\bibitem{las-lec}
A.~Lascoux, B.~Leclerc, and J.~Y. Thibon, \emph{The plactic
monoid}, Algebraic Combinatorics on Words, Encyclopedia of
Mathematics and Its Applications, vol.~90, Cambridge University
Press, Cambridge, 2002.

\bibitem{las-schut}
A.~Lascoux and M.~P. Sch\"{u}tzenberger, \emph{Le mono\"{\i}de plaxique},
Noncommutative Structures in Algebra and Geometric Combinatorics,
Marcel Dekker, Naples, 1978, pp.~129--156.

\bibitem{JO}
J.~Okni\'{n}ski, \emph{Semigroup Algebras}, Marcel Dekker, New
York, 1991.

\bibitem{sloane}
N.~J.~A. Sloane, \emph{The on-line encyclopedia of integer
sequences}, \\
\url{http://www.research.att.com/~njas/sequences/},
2007.


\end{thebibliography}
\end{document}